\documentclass[11pt,reqno]{amsart}
\usepackage[utf8]{inputenc}
\usepackage{amsfonts}
\usepackage{amsthm}
\usepackage{amsmath}
\usepackage{amssymb}
\usepackage{stmaryrd}
\usepackage{yfonts}
\usepackage{MnSymbol}
\usepackage{mathtools}
\usepackage{fullpage}
\usepackage{color}
\usepackage[shortlabels]{enumitem}
\usepackage{bbm}
\usepackage{tikz}
\usepackage{capt-of}
\usepackage{tikzscale} 
\usepackage{pb-diagram}
\usepackage{comment}
\usepackage{doc}
\usepackage[scr=boondox,  
            ]   
           {mathalpha}
\usepackage[hidelinks]{hyperref}

\newcommand{\de}{\delta}
\newcommand{\vphi}{\varphi}
\newcommand{\mc}{\mathcal}
\newcommand{\mk}{\mathfrak}
\newcommand{\mr}{\mathrm}
\newcommand{\mb}{\mathbf}
\newcommand{\conj}[2]{\{{#1}\,:\,{#2}\}}
\newcommand{\conjbig}[2]{\left\{{#1}\,:\,{#2}\right\}}
\newcommand{\con}{\subseteq}
\newcommand{\nrm}[1]{\|#1\|}
\newcommand{\norm}[1]{\left\|#1\right\|}
\newcommand{\abs}[1]{\left|#1\right|}

\newcommand{\al}{\alpha}
\newcommand{\be}{\beta}
\newcommand{\sig}{\sigma}
\newcommand{\vep}{\varepsilon}
\newcommand{\ga}{\gamma}
\newcommand{\Ga}{\Gamma}
\DeclareMathOperator{\supp}{\mathrm{supp}}
\newcommand{\buit}{\emptyset}
\newcommand{\rest}{\upharpoonright}

\newcommand{\ccal}{\mathscr{c}}
\newcommand{\normIII}[1]{{\left\vert\kern-0.25ex\left\vert\kern-0.25ex\left\vert #1\right\vert\kern-0.25ex\right\vert\kern-0.25ex\right\vert}}

\date{\today}

\newtheorem{teo}{Theorem}[section]
\newtheorem{claim}{Claim}[teo]
\newtheorem{coro}[teo]{Corollary}
\newtheorem{lema}[teo]{Lemma}
\newtheorem{rem}[teo]{Remark}
\newtheorem{prop}[teo]{Proposition}
\newtheorem{defi}[teo]{Definition}
\newtheorem{ex}[teo]{Example}
\newtheorem{question}[teo]{Question}

\def\F{\mathcal{F}}
\DeclareMathOperator{\Sh}{\mathcal{S}}
\newcommand{\C}{\mathcal{C}}
\newcommand{\B}{\mathcal{B}}
\newcommand{\N}{\mathbb{N}}

\def\R{\mathbb{R}}

\DeclareMathOperator{\Q}{\mathbb{Q}}

\DeclareMathOperator{\h}{hom}

\DeclareMathOperator{\Hom}{Hom}
\DeclareMathOperator{\fin}{\mr FIN}

\newcommand{\cantor}{2^{\N}}
\newcommand{\ideal}{\mathcal{I}}
\newcommand{\idealj}{\mathcal{J}}
\newcommand{\binary}{2^{<\omega}}

\def\exh{\mbox{\sf Exh}}
\def\su{\subseteq}
\def\suma{\mbox{\sf Sum}}

\usepackage[left=2cm,right=2cm,bottom=2.5cm,top=2.5cm]{geometry}
\SetSymbolFont{stmry}{bold}{U}{stmry}{m}{n}
\title{$F_\sig$-ideals, colorings, and representation in Banach spaces}
\author{Jordi Lopez-Abad}
\address{Departamento de Matem\'{a}ticas Fundamentales,
Facultad de Ciencias, UNED, 28040 Madrid, Spain}
\email{abad@mat.uned.es}
\author{V\'{i}ctor Olmos-Prieto}
\address{Departamento de Matem\'{a}ticas Fundamentales,
Facultad de Ciencias, UNED, 28040 Madrid, Spain}
\email{volmos@mat.uned.es}
\thanks{The second author received an FPI predoctoral contract from the Universidad Nacional de Educaci\'on a Distancia (Spain).}
\author{Carlos Uzc\'{a}tegui-Aylwin}
\address{Escuela de Matem\'aticas, Universidad Industrial de Santander, Bucaramanga, Colombia.}
\email{cuzcatea@saber.uis.edu.co}
\thanks{The third author thanks the partial support from Universidad Nacional de Educaci\'on a Distancia (Spain)  and  the grant VIE \#4248 by Universidad Industrial de Santander (Colombia).}

\date{}
 
\begin{document}

\begin{abstract}

In recent works by L. Drewnowski and I. Labuda \cite{Drewnowski} and J. Martínez et al. \cite{MMU2022}, non-pathological analytic \( P \)-ideals and non-pathological \( F_\sigma \)-ideals have been characterized and studied in terms of their representations by a sequence \( (x_n)_n \) in a Banach space, as \( \mathcal{C}((x_n)_n) \) and \( \mathcal{B}((x_n)_n) \). The ideal \( \mathcal{C}((x_n)_n) \) consists of sets where the series \( \sum_{n \in A} x_n \) is unconditionally convergent, while \( \mathcal{B}((x_n)_n) \) involves weak unconditional convergence.

In this paper, we further study these representations and provide effective descriptions of \( \mathcal{B} \)- and \( \mathcal{C} \)-ideals in the universal spaces \( C([0,1]) \) and \( C(2^{\mathbb{N}}) \), addressing a question posed by Borodulin-Nadzieja et al. \cite{Borodulinetal2015}. A key aspect of our study is the central role of the space \( c_0 \) in these representations. We focus particularly on \( \mathcal{B} \)-representations in spaces containing many copies of \( c_0 \), such as \( c_0 \)-saturated spaces of continuous functions.

A central tool in our analysis is the concept of \( c \)-coloring ideals, which arise from homogeneous sets of continuous colorings. These ideals, generated by homogeneous sets of 2-colorings, exhibit a rich combinatorial structure. Among our results, we prove that for \( d \geq 3 \), the random \( d \)-homogeneous ideal is pathological,
 we construct hereditarily non-pathological universal \( c \)-coloring ideals, and we show that every \( \mathcal{B} \)-ideal represented in \( C(K) \), for \( K \) countable, contains a \( c \)-coloring ideal.
Furthermore, by leveraging \( c \)-coloring ideals, we provide examples of \( \mathcal{B} \)-ideals that are not \( \mathcal{B} \)-representable in \( c_0 \). These findings highlight the interplay between combinatorial properties of ideals and their representations in Banach spaces.

\end{abstract}


\maketitle



\section{Introduction}

(Set-theoretic) Ideals of sets are a way to 
describe ``smallness" or ``negligibility" of subsets in mathematical analysis. 
 Classical examples include the {\em Fréchet ideal} of finite sets  \( \mathrm{FIN} \) and the ideal of measure-zero sets in a measure space. Definable, in particular, analytic ideals have been vastly investigated (see for example, \cite{fa2000}, \cite{Hrusak2011},  \cite{Solecki1999}, \cite{todorcevic-definable}).  In this paper we will focus on ideals over a countable set $X$, and of simple definition, that is,  $F_\sig$-ideals --as subsets of $\{0,1\}^X$ via characteristic functions--.  The first example of them are the summable ideals $\fin(\mu)$, collections of finitely $\mu$-measured subsets of $X$ for some total ($\sig$-additive) measure $\mu$ on $X$.   Although not all $F_\sig$-ideals are summable ideals, they are structurally similar, as it was proved by K. Mazur \cite{Mazur91} that every $F_\sig$-ideal is of the form $\fin(\vphi)$ for some total lower semicontinuous {\em submeasure} $\vphi$ on $X$. So, the simplest ones are those defined from measures,  then the ones that use lower semicontinuous submeasures that are supremum of measures, called {\em non-pathological}, and finally the pathological ones.  

The second sort of ideals we are interested in are the  $P$-ideals, again somehow extending the summable ones. Recall that a \( P \)-ideal is defined as an ideal \( \mc I \) on \( \mathbb{N} \) such that for every sequence \( (A_n)_{n \in \mathbb{N}} \subseteq \mc I \), there exists \( A \in\mc I \) such that \( A_n \setminus A \) is finite for all \( n \).  Again, the simplest $P$-ideals are the summable ones, and, in general, they are  a kind of summable:   S. Solecki \cite{Solecki1999} demonstrated that every analytic \( P \)-ideal can be expressed as \( \mathrm{Exh}(\vphi):=\conj{A \subseteq X}{  \lim_{n \to \infty} \vphi(A \setminus \{0, 1, \dots, n\}) = 0} \) for some lower semicontinuous submeasure \( \vphi \) on $X$. Once more, we have the same simple classification of analytic $P$-ideals: the summable ones $\mr{Exh}(\mu)= \fin(\mu)$ for a measure $\mu$, the non-pathological ones, $\mr{Exh}(\mu)$ for a non-pathological submeasure $\mu$, and finally the pathological ones. 

In recent years, some interesting ideals on $\mathbb{N}$ related to series in  Banach spaces have been studied \cite{Borodulinetal2015,Borodulin-Farkas2020,Drewnowski,Drewnowski2017}. Given a sequence ${\mathbf{x}} = (x_n)_n$ in a Banach space, L. Drewnowski and I. Labuda \cite{Drewnowski} defined an ideal on $\mathbb{N}$ as follows:
\[
\mathcal{C}({\mathbf{x}}) := \left \{ A \subseteq \mathbb{N} : \sum_{n \in A} x_n \text{ is unconditionally convergent} \right \},
\]
where the sum $\sum_{n \in A} x_n$ is \textit{unconditionally convergent} when it converges under every permutation of the index set. As a consequence of a classical Bessaga-Pełczyński's theorem \cite{BP}, Drewnowski and Labuda showed that a Banach space does not contain an isomorphic copy of $c_0$ if and only if $\mathcal{C}({\mathbf{x}})$ is $F_\sigma$ (as a subset of $\{0,1\}^\mathbb{N}$ via characteristic functions) for every ${\mathbf{x}}$. Later, P. Borodulin-Nadzieja, B. Farkas, and G. Plebanek conducted a thorough analysis of the ideals $\mathcal{C}({\mathbf{x}})$ in \cite{Borodulinetal2015} and established a tight connection with the well-known class of analytic P-ideals. The ideal $\mathcal{C}({\mathbf{x}})$ can be naturally interpreted as the generalization of summable ideals, a prototype being $\mathcal{I}_{1/n}$, which consists of all $A \subseteq \mathbb{N}$ such that $\sum_{n \in A} 1/n < \infty$; it plays a central role in the classification of ideals on countable sets.

In the aforementioned work \cite{Drewnowski}, the authors introduced a new family of ideals associated with a different notion of summability in Banach spaces:
\[
\mathcal{B}({\mathbf{x}}) := \left \{ A \subseteq \mathbb{N} : \sum_{n \in A} x_n \text{ is weakly unconditionally convergent} \right \}.
\]
We recall that $\sum_{n \in A} x_n$ is \textit{weakly unconditionally convergent} (or \textit{perfectly bounded}) when the numerical series $\sum_{n \in A} x^*(x_n)$ is unconditionally convergent for every $x^* \in X^*$. It is well-known that this is equivalent to saying that $\sup_{F \subseteq A \text{ finite}} \left\| \sum_{n \in F} x_n \right\|$ is finite. Observe that $\mc C(\mathbf x)\con \mc B(\mathbf x)$, but in general they are different. In fact,  in \cite{Drewnowski} it is shown that a Banach space does not contain an isomorphic copy of $c_0$ exactly when  $\mathcal{C}({\mathbf{x}}) = \mathcal{B}({\mathbf{x}})$ for every ${\mathbf{x}}$. Following the work developed in \cite{Borodulinetal2015}, Martínez et al. \cite{MMU2022} characterized the ideals $\mathcal{B}({\mathbf{x}})$ as the \textit{non-pathological} $F_\sigma$ ideals, that is to say, ideals of the form $\fin(\varphi) := \{ A \subseteq \mathbb{N} : \varphi(A) < \infty \}$ for $\varphi$ a non-pathological lower semicontinuous submeasure on $\mathbb{N}$ (i.e., a lscsm which is a supremum of measures).
However, this characterizations are somewhat inefficient as they are realized in the universal space \( \ell_\infty \). In the Section \ref{repre}, we provide effective representations of $\mathcal{B}$ and $\mc C$-ideals in the universal spaces \( C([0,1]) \) and \( C(2^{\mathbb{N}}) \), addressing a question posed in \cite[Question 7.1]{Borodulinetal2015}.
The representation of ideals on \( C[0,1] \) that we  present has the advantage of being done in a Polish space. This opens up the possibility of properly analyze classes of ideals from a descriptive set-theoretic perspective, thereby extending the study of the complexity of the codes of families of \( F_\sigma \) ideals (\cite{GrebikHrusak2020, GrebikUzca2018}) to classes of non-pathological ideals.

In general, it is natural to investigate the consequences on ideals that can be $\mathcal{B}$- or $\mathcal{C}$-represented in a particular class of Banach spaces. One example of them is the result by Drewnowski and Labuda of $\mc B$-representation in spaces without copies of $c_0$, and  conversely, P. Borodulin-Nadzieja, B. Farkas, and G. Plebanek showed in \cite{Borodulinetal2015} that the only $F_\sigma$ ideals that are $\mathcal{C}$-representable in $c_0$ are the summable ones. Similarly, ideals that are $\mathcal{B}$-representable in a finite-dimensional space are also summable.


Analogous to the representation of non-pathological \( F_\sigma \)-ideals as \(\mathrm{Fin}(\varphi)\), the \(\mathcal{B}\)-representation is also non-unique. It is straightforward to observe that any \(\mathcal{B}\)-ideal can be represented using a 1-unconditional basis (see Proposition \ref{lkmngtktjgkdlfgdf}). Notably, we will show that if this representation is carried out in a Hilbert space, the ideal must necessarily be summable. This result is a consequence of the \textit{generalized parallelogram identity}.
More generally, when the underlying space satisfies a slightly weaker version of this identity, specifically possessing \textit{non-trivial cotype} \(q \geq 2\), we can prove that the \(\mathcal{B}\)-ideal \(\mc B((x_n)_n)\) is included in the summable ideal \(\mathrm{Sum}((\|x_n\|^q)_n)\). However, it remains an open question whether one can choose the \(\mc B\)-representation such that this summable ideal is non-trivial. 

A classical reformulation of the property of having non-trivial cotype is that \(c_0\) cannot be \textit{finitely representable}, highlighting once again  the importance of the space $c_0$ in the theory of $\mathcal{B}$-representation.

As every ideal $\mathcal{B}$-representable in a space without isomorphic copies of $c_0$ is automatically a $P$-ideal, it is natural to closely study the relationship between $\mathcal{B}$-representation and the presence of $c_0$ in a space to understand, in particular, $\mathcal{B}$-representable non-$P$-ideals. Recall that an ideal $\mathcal{I}$ is \textit{tall} when every infinite subset of $\mathbb{N}$ has a further infinite subset belonging to $\mathcal{I}$. 
We prove in Theorem \ref{c0saturacion} that a $\mathcal{B}$-ideal is tall exactly when it can be $\mathcal{B}$-represented by a seminormalized unconditional basis $(x_n)_n$ that is sequentially saturated by equivalent copies of the unit basis of $c_0$; that is, every subsequence of $(x_n)_n$ has a further subsequence $(y_n)_n$ such that $\left\|\sum_n a_n y_n\right\| \approx \max_n |a_n|$. We do not know if this can be strengthened to assert that tall $\mathcal{B}$-ideals can be $\mathcal{B}$-represented in spaces that are saturated by isomorphic copies of $c_0$, i.e. such that every infinite-dimensional subspace has a subspace isomorphic to $c_0$.

A representative family of $c_0$-saturated spaces consists of the spaces of continuous functions on a countably infinite compact space, which will be the focus of Subsection \ref{lmklfmklsdfs}. In this subsection, we demonstrate that when a tall ideal is representable in one of these function spaces, there exists an effective procedure to extract a subset within the ideal from any given infinite set. This extraction process is carried out indirectly using what we call \textit{c-coloring ideals}. 
These ideals are generated by the homogeneous  sets of a given 2-coloring of $[\mathbb{N}]^2$, or more generally, by a \textit{front} on $\mathbb{N}$—a natural generalization introduced by C. St. J. A. Nash-Williams \cite{nash-williams-transfinite} for the families $[\mathbb{N}]^d$. Specifically, we prove in Theorem \ref{lkmkdlmfmskl4we4ret} that every ideal $\mathcal{B}$-represented in some $C(K)$, where $K$ is countable, contains a $c$-coloring ideal. Moreover, if $K$ has finite Cantor-Bendixson rank, it contains a $c$-coloring ideal generated by a coloring of $[\mathbb{N}]^2$. Consequently, in Theorem \ref{exG}, we provide an example of a $\mathcal{B}$-ideal that is not $\mathcal{B}$-representable in $c_0$.

The study of $c$-coloring ideals, presented in Section \ref{sect:colorings}, is noteworthy because these ideals possess effective tallness, making them a valuable tool in the analysis of tall ideals. Furthermore, $c$-coloring ideals are particularly interesting for two main reasons. 

First, they provide a natural source of examples of $F_\sigma$ ideals that can exhibit pathological behavior. For instance, based on the work in \cite{MMU2022} on Mazur's pathological ideal,  we demonstrate that the ideals of cliques and anticliques in the random $d$-uniform hypergraph are also pathological when $d \geq 3$, while the pathology of the random ideal $\mc R$ remains an open question.

Second, the techniques employed in their study are of significant interest. In particular, we draw on ideas from measure concentration to analyze certain hypergraphs and on the  work of F. Galvin \cite{Galvin1968} and D.~C.~Devlin \cite{devlin1979} in Ramsey theory for canonical colorings of $[\mathbb{Q}]^d$. These methods enable us to identify universal examples of $c$-coloring ideals that are locally  non-pathological.

Finally, we employ another combinatorial tool—Dilworth's theorem—to construct examples of $c$-coloring ideals that are non-pathological.


\section{Preliminaries}\label{klkl4tmerlkgrfg}

We use  \cite{Kechris94} and  \cite{Lind-Tza}  as a general reference for all notions and terminology  of descriptive set theory and of Banach space theory, respectively.  Recall that an ideal over a set $X$ is a collection of subsets of $X$ closed under taking subsets and finite unions of its elements. We will assume always that every ideal over $X$ contains the corresponding Fréchet ideal of finite subsets of $X$. Let $s,t$ be  subsets of $\N$ with $s$ finite, we write $s\sqsubseteq t$ if $s$ is an initial segment of $t$, i.e., there is $n\in \N$ such that $s=t\cap \{0,\ldots, n\}$.  The collection of binary sequences is denoted by $\binary$. For each $s\in \binary$,  $|s|$ denotes the length of $s$. A basis for the product topology on $\cantor$ is given by the following sets:
\[
[s]:=\{\alpha\in\cantor: \; s\prec \alpha\}
\]
where $s\prec\alpha$ means that $s$ is an initial segment of $\alpha$, i.e. $s(i)=\alpha(i)$ for all $i<|s|$. 
For any set $X$, we denote by $[X]^\omega$ (resp. $[X]^{<\omega}$) the collection of all infinite (resp. finite) subsets of $X$. For $X$ countable, $[X]^\omega$ is a Polish space with the product topology it inherits from   $\{0,1\}^X$ via characteristic functions. This allows us to discuss Borel and analytic ideals.  If $\mc A$ and $\mc B$ are two collections of subsets of some set, we denote by $\mc A\sqcup \mc B$ the collection of sets of the form $A\cup B$ with $A\in\mc A$ and $B\in \mc B$. 

A function  $\varphi:\mathcal{P}(\N)\to[0,\infty]$ is a {\em lower semicontinuous submeasure (lscsm)} if $\varphi(\emptyset)=0$,
$\varphi(A)\leq\varphi(A\cup B)\leq\varphi(A)+\varphi(B)$, $\varphi(\{n\})<\infty$ for all $n$ and $\varphi(A)=\lim_{n\to\infty}\varphi(A\cap\{0,1,\dots, n\})$. Associated to a lscsm there are the following ideals:
\[
\begin{array}{lcl}
\mathrm{Sum}(\varphi) & := & \{ A\con \N :  \sum_{n\in A}\varphi(\{n\})<\infty\}\\
\fin(\varphi)& :=& \{A\subseteq\N:\varphi(A)<\infty\}
\\
\exh(\varphi)& := &\{A\subseteq\N: \lim_{n\to\infty}\varphi(A\setminus\{0,1,\dots, n\})=0\}.
\end{array}
\]
Notice that $\mr{Sum}(\varphi)\con \exh(\varphi)\subseteq \fin(\varphi)$, and that $\mathrm{Sum}(\varphi)=\fin(\varphi)$ when $\varphi$ is indeed a measure. 
An ideal $\ideal$ is a {\em $P$-ideal} if for every sequence $(A_n)_n$ of sets in $\ideal$ there is $A\in \ideal$ such that 
$A_n\setminus A$ is finite for all $n$. A fundamental result due to S. Solecki \cite{Solecki1999} says that any  analytic $P$-ideal is of the form $\exh(\varphi)$ for some lscsm $\varphi$. 

We say that a measure  $\mu $ is {\em dominated by a lscsm $\varphi$}, if $\mu(A)\leq \varphi(A)$ for all $A$. 
A lscsm $\varphi$ is {\em non-pathological} if it is the supremum of all  ($\sigma$-additive) measures dominated by $\varphi$. An $F_{\sigma}$ ideal $\ideal$ is \emph{non-pathological} if there is a non-pathological lscsm $\varphi$ such that $\ideal=\fin(\varphi)$.

Let $\ideal$ and $\idealj$ be two ideals on $\N$. The Kat\v{e}tov pre-order is defined as follows. We say that $\ideal\leq_K \idealj$ when there is $f:\N\to \N$ such that $f^{-1}(A)\in \idealj$ for all $A\in \ideal$.
Kat\v{e}tov pre-order is  a very useful tool to study ideals on countable sets (see, for instance, \cite{Hrusak2017}).  
An ideal that plays a pivotal role is  $\mathcal{R}$, which is generated by the cliques and independent sets of the random graph. From  the universal property of the random graph, it is known that $\mathcal{R}\leq_K \ideal$ exactly when there is a coloring $c:[\N]^2\to \{0,1\}$ such that  $\ideal$ contains the collection  $\hom(c)$,  the family of all $c$-homogeneous sets (i.e. sets $A$ where   $c$ is constant on $[A]^2$). When this happens the ideal $\ideal$ is {\em tall}, meaning that every infinite subset of $\N$ contains an infinite subset belonging to $\ideal$. To illustrate this fact, let $\varphi$ be a lscsm such that $\exh(\varphi)$ is tall (that is, $\varphi(\{n\})\to_{n\to \infty} 0$).  We show a coloring $c$ such that $\hom(c)\subseteq \exh(\varphi)$. Let  $A_{k+1}=\{n\in\N:\; 1/2^{k+1}\leq \varphi\{n\}<1/2^{k}\}$ for $k\in\N$ and $A_0=\{n\in\N:\; 1\leq \varphi\{n\}\}$. Observe that each $A_k$ is finite. Let $c\{n,m\}=0$ iff $\{n,m\}\subseteq A_k$ for some $k$.  Notice that if $H$ is an infinite homogeneous set, then it  is $1$-homogeneous and  $\sum_{n\in H}\varphi\{n\}<\infty$.
Thus $H\in \exh(\varphi)$.

As we said if $\mathcal{R}\leq_K \ideal$, then $\ideal$ is tall. For a while it was conjecture that the converse also holds, i.e., that  every analytic tall ideal is $\leq_K$-above $\mathcal{R}$. This  turned out to be false \cite{GrebikHrusak2020,grebik2020tall}, as there are $F_\sigma$ tall ideals which are not Kat\v{e}tov above $\mathcal{R}$ {(see section \ref{nonc0} for a concrete example of a non-pathological ideal satisfying this condition).}

A tall family  $\mathcal{A}$  of subsets of  a countable set $X$  admits a {\em Borel selector}, when there is a Borel function $S: [X]^\omega\to  [X]^\omega$  that assigns to each $E\subseteq X$ a subset $S(E)\su E$  that belongs to $\mathcal A$.  A typical example of a family admitting a Borel selector is given by the following result. 

\begin{teo}
\label{selector}
\cite{GrebikUzca2018}
 The tall family $\hom(c)$ of homogeneous sets of a coloring $c: [\N]^2\to \{0,1\}$  admits a Borel selector. \qed 
 \end{teo}

We recall $\mc B$ and $\mc C$ representability.
\begin{defi}[Representability]
We say that ideal $\mc I$ on $\N$ is {\em $\mathcal B$-representable} in a Banach space $X$ when there is a sequence $\mb x=(x_n)_n$ in $X$ such that $\mc I= \mc B(\mb x)= \conj{M\con \N}{\sup_{F\con M \text{ finite}} \nrm{\sum_{n\in F} x_n} < \infty}$. The ideal $\mc I$ is a {\em $\mc B$-ideal} if it is $\mc B$-representable in some $X$.  Similarly, we say that $\mc I$ is {\em $\mathcal C$-representable} in $X$ when there is a sequence $\mb x=(x_n)_n$ in $X$ such that $\mc I= \mc C(\mb x)= \conj{M\con \N}{\sum_{n\in M} x_n \text{ is unconditionally convergent}}$. Such ideals are called {\em $\mc C$-ideals}.
\end{defi}

{As mentioned in the introduction, $\mathcal{B}$-ideals coincide with non-pathological $F_{\sigma}$ ideals (see \cite{MMU2022}) while $\mathcal{C}$-ideals correspond to non-pathological analytic $P$-ideals:
given a sequence $\mb x$ the function $\vphi_{\mb x}: F\in \mc P(\N)\mapsto \sup_{F\in [A]^{<\infty}} \nrm{\sum_{n\in F}x_n}$ is a lower semicontinuous submeasure such that $\mc B(\mb x)= \fin(\vphi_{\mb x})$. In addition it is non-pathological: for each $F$ finite choose a norm-one functional $x^*_F\in S_{X^*}$ such that $x^*_F(\sum_{n\in F} x_n)= \nrm{\sum_{n\in F}x_n}$. Now define $\mu_F$  as 
the measure with support $F^+=\conj{n\in F}{ x_F^*(x_n)\ge 0}$ and such that $\mu_F(\{n\})=x_F^*(x_n)$ for $n\in \supp \mu_F$.
Then $\sup_{F\in [\N]^{<\infty}}\mu_F =\vphi_{\mb x}$. To see this, given $A\con \N$ and $F\con A$ finite, then $\nrm{\sum_{n\in F}x_n}\le \mu_F(F)$. Hence, $\vphi_{\mb x}(A)\le \sup_{F\in [\N]^{<\infty}}\mu_F(A)$. 
On the other hand, given $F$ finite, we have that $\mu_F(A)=\mu_{F} (A\cap F^+)\le \nrm{\sum_{n\in A\cap F^+} x_n}\le \vphi_{\mb x}(A)$, and consequently $\sup_{F\in [\N]^{<\infty}} \mu_F(A)\le \vphi_{\mb x}(A)$.  
Moreover, $\mc C(\mb x) = \mr{Exh}(\vphi_{\mb x})$, because $A\in \mc C(\mb x)$ exactly when for every $\vep>0$ there is some $n$ such that $\sup_{F\in [A\cap [n,\infty[]^{<\infty}} \nrm{\sum_{n\in F} x_n}= \vphi_{\mb x}([n,\infty[)\le \vep$. 

Conversely, suppose that $\vphi$ is a non-pathological lscsm, and let $(\mu_k)_k$ be a sequence of measures such that $\sup_k \mu_k = \vphi$. For each $n$ define the sequence $x_n := (\mu_k(\{n\}))_k\in \ell_{\infty}$, which is bounded because $\|x_n\|_{\infty} = \vphi(\{n\}) < \infty$. Then $\vphi=\vphi_{\mb x}$ for $\mb x := (x_n)_n$, because for $A\subseteq\N$
\[
\vphi_{\mb x}(A) = \sup_{F\in [A]^{<\infty}} \norm{\sum_{n\in F} x_n}_{\infty} = \sup_{k\in\N} \sup_{F\in [A]^{<\infty}} \mu_k(F) = \sup_{k\in\N} \mu_k(A) = \vphi(A).
\]
}

$\mc B$-representability (or $\mc C$-representability) on finite dimensional spaces characterize the summable ideals. To see this, observe that in those spaces unconditional convergence of a series is equivalent to its absolute convergence, that is $\mc C( (x_n)_n)= \mr{Sum}((\nrm{x_n})_n)$, and  in finite dimensional spaces  the weak and the norm topology coincide in those spaces $X$, so we obtain that $\mc B((x_n)_n)=\mc C((x_n)_n)$. This last equality is not only true in finite dimensional spaces but in spaces not containing isomorphic copies of $c_0$.   On the opposite direction, there are several examples of non-$P$-ideals that are $\mc B$-representable in $c_0$. One of them is the following.

{
\begin{ex} There is a tall non $P$-ideal that is $\mc B$-representable in $c_0$:
    Let $\{A_n : n\in\N\}$ be an infinite partition of $\N$ in infinite sets, and write each $A_n = \{a^n_{m}\}_{m\in \N}$ by its increasing enumeration. Let $(u_k)_k$ be the canonical basis of $c_0$ and for each $m,n\in \N$, let 
    $(I_m^n)_{m\in \N}$ be a partition of $A_n$ in consecutive intervals each $I_m^n$ of cardinality $(n+1)(n+m+1)$. Finally we define the sequence $(x_k)_{k\in \N}$ of vectors in $c_0$ as follows. For each $k\in \N$, we set
    $x_k:= (1/(n(k)+m(k)+1))u_{m(k)}$ where $m(k),n(k)\in \N$ are the unique integers such that $k\in I_{m(k)}^{n(k)}$.  We have that $\bigcup_{m,n\le l} I_{m}^n$ is finite, so if $k\ge \max \bigcup_{m,n\le l} I_{m}^n$, then $\nrm{x_k}_\infty \le 1/(l+1)$. This shows that the sequence $(x_k)_k$ is norm-null, and consequently $\mc B((x_k)_k)$ is tall (see Theorem \ref{c0saturacion}). We see that $\mc B((x_k)_k)$ is not a $P$-ideal. To see this, observe that each $A_n\in \mc B((x_k)_k)$ because for each $F\con A_n$ we have that $\nrm{\sum_{k\in F} x_k}_\infty = 1/(n+m+1)\le 1/(n+1)$ where $m$ is the minimal integer such that $F\cap I_{m}^n\neq \buit$. On the other hand, suppose that $A$ is such that $A_n\setminus A$ is finite for every $n$. Given $n$, let $m(n)$ be large enough so that $I_{m(n)}^n\con A$. Since each $\nrm{\sum_{k\in I_{m(n)}^n } x_k}_\infty= \# (I_{m(n)}^n)/(m(n)+n+1)=n+1$, it follows that $\sup_{F\in [A]^{<\infty}} \nrm{\sum_{k\in F} x_k}_\infty=\infty$, and consequently $A\notin \mc B((x_k)_k)$.  

    
\end{ex}
In \cite{MMU2022} there is another simpler example of a $\B$-ideal representable in $c_0$ that is not a $P$-ideal, but it is not tall. This leads to the following question.

\begin{question}
Is every  $F_\sig$ $\mc C$-ideal  automatically a $\mc B$-ideal?  
\end{question}
The answer is Yes for $c_0$, by Theorem 5.7 \cite{Borodulinetal2015} which says that a $F_\sigma$ ideal $\C$-representable in $c_0$ is summable (and thus $\B$-representable). 
S. Solecki proved in \cite[Thm 3.4]{Solecki1999} that if $\ideal=\exh(\varphi)$ is $F_\sigma$, then there is another lscsm $\psi$ such that $\ideal=\exh(\psi)=\fin(\psi)$. We do not know if $\psi$ can be found non-pathological assuming $\varphi$ is.




}

The representation of a $\mc B$ or $\mc C$-ideal is not unique. One can take advantage of this and represent the ideal in a way that the representative sequence is, for example, a Schauder basis, even 1-unconditional.   Recall that a {\em Schauder basis} is a sequence $(x_n)_n$ of a Banach space $X$ such that any vector $x\in X$ has a unique representation as the sum of a series $\sum_{n} a_n x_n$, and the basis $(x_n)_n$ is $1$-unconditional when any such a convergent series $\sum_n a_n x_n$ is unconditionally convergent and norm preserving (that is, $\sum_n \theta_n a_n x_n$ converges for every sequence of signs $(\theta_n)_n$ and $\nrm{\sum_n \theta_n a_n x_n}=\nrm{\sum_n a_n x_n}$). The basis is called {\em seminormalized} if $0 < \inf_n \|x_n\| \leq \sup_n \|x_n\| < \infty$.

\begin{prop}\label{lkmngtktjgkdlfgdf}
Let $\mc I$ be an ideal.
\begin{enumerate}[(i),wide=0pt]
    \item   $\mc I$ is $\mc B$-representable iff it is $\mc B$-representable by a 1-unconditional basis $\bf u$ such that $\mc C(\bf u) = \fin$.
    \item\label{item:lkmngtktjgkdlfgdf:C} $\mc I$ is $\mc C$-representable iff it is $\mc C$-representable by a 1-unconditional basis $\bf u$.
\end{enumerate}
\end{prop}
\begin{proof}
{\it (i)}: Suppose that $\mc I = \mc B(\bf x)$ for a sequence ${\bf x} = (x_n)_n$ in a Banach space $X$. On $c_{00}$ we define the norm
\[
\normIII{(a_n)_n} := \max\left\{\max_{n} |a_n|, \max_{(\theta_n)_n\in \{-1,1\}^\N} \norm{\sum_n \theta_n a_n x_n}\right\},
\]
and we let $\widetilde{X}$ to be completion of $c_{00}$ with respect to this norm. It is easy to see that the canonical basis $(u_n)_n$ of $c_{00}$ is a 1-unconditional, not necessarily normalized, basis of $\widetilde{X}$. It is clear that $\mc B(\bf u)\subseteq\mc B(\bf x)$, while if $A\in \mc B(\bf x)$, then for every finite set $F\con A$ we have that 
\[
\normIII{\sum_{n\in F}u_n} = \max\left\{1, \max_{(\theta_n)_n\in \{-1,1\}^F}\norm{\sum_{n\in F}\theta_n x_n}\right\}\le \max\left\{1,2\max_{G\con F}\norm{\sum_{n\in G}x_n}\right\}\le 2\sup_{G\con A}\norm{\sum_{n\in G}x_n}.
\]
{Thus $\mc B (\bf x) \subseteq \mc B (\bf u)$.} On the other hand, since $\nrm{u_n}\ge 1$ for every $n$, 
no infinite $A$ belongs to $\mc C(\bf u)$.  

{\it (ii)}: 
As before, given $\mc I = \mc C(\bf x)$ we define a norm on $c_{00}$ through
\[
\normIII{(a_n)_n} := \max\left\{\max_{n} \frac{1}{n+1}|a_n|, \max_{(\theta_n)_n\in \{-1,1\}^\N} \norm{\sum_n \theta_n a_n x_n}\right\}.
\]
A similar argument as above shows that the standard basis $(u_n)_n$ of $c_{00}$ is $1$-unconditional in the completion $\widetilde{X}$, and that $\mc C(\bf u) = \mc C(\bf x)$.
\end{proof}

We do not know if a $\mc C$-representable ideal $\mc I$ can be rewritten as $\mc I = \mc C(\bf e)$ for some sequence $\bf e$ with $\mc B(\bf e) = \mc P(\N)$. Note that in the proof of \ref{item:lkmngtktjgkdlfgdf:C} above the resulting sequence $\bf u$ satisfies $\mc B(\bf u) = \mc B(\bf x)$.


A basis $(x_k)_k$ is called \emph{1-subsymmetric} (or 1-spreading) if 
\[
\left\|\sum_{k=1}^n a_k x_k\right\| = \left\|\sum_{k=1}^n a_{k} x_{l_k}\right\|
\]
for every sequence of scalars $(a_k)_{k=1}^n$ and every increasing sequence of indices $l_1 < \cdots < l_n$. If, in addition, $(x_k)_k$ is weakly null, then it is automatically 1-unconditional.

\begin{question}
Study the properties of $\mathcal{B}((x_k)_k)$ when $(x_k)_k$ is a 1-subsymmetric basis.
\end{question}

\begin{question}
Let us say that two sequences ${\bf x}=(x_n)$ and  ${\bf y}=(y_n)$ in a Banach space $X$ are $\B$-equivalent if
$\B({\bf x})= \B({\bf y})$. How complex is this equivalence relation? 
\end{question}
\section{\texorpdfstring{Tallness, $c_0$-saturation and summability}{Tallness and c0-saturation }}

In this section we explore the relationship between the tallness of a $\mathcal{B}$-ideal and the sequential $c_0$-saturation of the sequence that represents it. This relationship is established in Theorem \ref{c0saturacion}, whose proof employs standard techniques from Banach space theory. The theorem highlights the central role of the space $c_0$ in the $\mathcal{B}$-representation of ideals. We demonstrate that whenever an ideal $\mc B((x_n)_n)$ is tall and represented by a seminormalized sequence $(x_n)_n$, then this sequence must be saturated by sequences equivalent to the unit basis of $c_0$. It is natural to then ask if the space spanned by $(x_n)_n$ is saturated by isomorphic copies of $c_0$. In Proposition \ref{mlkdmflkdsfds}, we show that a natural class of spaces $\mc I(X,\varphi)$, with $X$ not containing copies of $c_0$, contradicts this. Specifically, these examples can be $\mc B$-represented in a space that does not contain copies of $c_0$ and consequently are simultaneously $\mc B$ and $\mc C$-represented by the same sequence.

Interestingly, if we strengthen the hypothesis from not containing “full” isomorphic copies of $c_0$ to not containing “asymptotic” finite-dimensional copies of $c_0$, we find that the corresponding ideal $\mc B((x_n)_n)$ satisfies the inclusion
\[
\mathrm{Sum}((\|x_n\|)_n) \subseteq \mc B((x_n)_n) \subseteq \mathrm{Sum}((\|x_n\|^q)_n),
\]
for some $2 \leq q < \infty$. 
The proof of this extension theorem relies on the local theory of Banach spaces, more precisely on the notion of cotype of a Banach space, which was primarily developed by B. Maurey and G. Pisier \cite{MauPis}. A particular case of independent interest is the $\mc B$-representability in a Hilbert space.

The following is a classical characterization that uses the uniform boundedness principle. 
 
\begin{prop}
\label{wuc}
A series $\sum_{n\in \N}x_n$ in a Banach space  $X$ is weakly unconditionally convergent exactly when the numerical series $\sum_n x^*(x_n)$ converge unconditionally for every   $x^*\in X^*$.  Consequently, if the ideal $\B({\mb x})$ is tall, then the sequence ${\mb x}$ is weakly null. \qed
\end{prop}

Recall that an ideal of the form $\mc C(\mb x)$ is tall if and only if $\mb x$ is norm-null (i.e. norm-converging to zero). Moreover, if $c_0$ does not embed in $X$, then $\mc B(\mb x) = \mc C(\mb x)$ for every sequence, hence we have the same characterization of tallness. However, using the well-known fact that weakly unconditional bases are precisely the sequences equivalent to the unit basis of $c_0$, we can show a general characterization which explicitly takes $c_0$ into account.


{
\begin{teo}
\label{c0saturacion}
Let ${\mb x}=(x_n)_n$ be a  sequence in a Banach space $X$.  The following are equivalent.
\begin{enumerate}[(i), wide=0pt]
     \item $\B({\mb x})$ is tall.
    \item  Every subsequence of ${\mb x}$ has a further subsequence that is either norm-null or equivalent to the unit basis of $c_0$. 
\end{enumerate}
Consequently, 
\begin{enumerate}[(i), wide=0pt] \addtocounter{enumi}{2}
    \item if   $(x_n)_n$ does not have subsequences equivalent to the unit basis of  $c_0$, then $\mc B(\mb x)$ is tall exactly when $\mb x$ is a norm-null sequence, and 
\item if $X$ is isomorphic to a subspace of $c_0$, then $\mc B(\mb x)$ is tall exactly when $\mb x$ is weakly-null.
\end{enumerate}
\end{teo}

\proof
$(ii) \Rightarrow (i)$. Fix $M\in [\N]^\omega$. Then one of the two following options occur: there is  an infinite subset $N$ of $M$ such that $\sum_{n\in N} \|x_n\|\le 1$, and consequently $N\in \B({\mb x})$, or there is an infinite subset $N$ of $M$ and  $C\geq 1$ such that for all $(a_n)_n$ in $c_{00}$
\[
\frac{1}{C}\sup_{n\in N}|a_n| \leq \norm{\sum_{n\in N} a_n x_n}\leq C\sup_{n\in N} |a_n|,
\]
and, in particular, this implies that $N\in\B({\mb x})$. 

\bigskip 

$(i) \Rightarrow (ii)$. This is a direct consequence of the Bessaga-Pelczynski theorem \cite{BP} and the fact that non-trivial weakly-null sequences have basic subsequences. For the sake of completeness we give a detailed proof. First of all, it follows from the Proposition \ref{wuc} that ${\mb x}$ is weakly null, and in particular, $K:=\sup_n \|x_n\|<\infty$. Fix $M\in [\N]^\omega$. If $\inf_{n\in M}\|x_n\|=0$, then  we can easily find $N\subseteq M$ infinite such that $(x_n)_{n\in N}$ is norm null. Otherwise, i.e., if $\delta:=\inf_{n\in M} \|x_n\|>0$, then by a classical result (see for example \cite[Proposition 1.5.4.]{AK}) there is an infinite subset $M_0$ of $M$ such that $(x_n)_{n\in M_0}$ is a basic sequence with basic constant at most $2$, that is, 
$\|\sum_{n\in I} a_n x_n\|\le 2\|\sum_{n\in J} a_n x_n\|$ for every $I\sqsubseteq J \sqsubset M_0$ and every sequence of scalars $(a_n)_{n\in J}$. In particular, it follows easily from here and the triangle inequality that  for each $n\in s\subset M_0$, 
\[
|a_n|= \frac{\|a_n x_n\|}{ \|x_n\|} \le \frac1{\|x_n\|} \left(\norm{\sum_{m\in s, \, m\le n} a_m x_m} + \norm{\sum_{m\in s, \, m< n} a_m x_m}\right) \le \frac{4}{\|x_n\|} \norm{\sum_{m\in s} a_m x_m},
\]
and consequently,
\begin{equation}
\label{ioh43th34rt}
\max_{n\in s} |a_n| \le \frac{4}{\delta} \norm{\sum_{n\in s} a_n x_n}
\end{equation}
for every $s\subseteq M_0$ finite. We use now that $\B(\mb x)$ is tall to find $N\subseteq M_0$ infinite and $K>0$ such that $\sup_{s\in [N]^{<\omega}} \|\sum_{n\in s} x_n\|\le K$. 
Now, given $s\subseteq N$ finite, $x^*$ in the unit ball of the dual space $X^*$ of $X$ and $i,j\in \{0,1\}$, we define
$S_{i,j}:=\{n\in s\, :\,(-1)^i x^*(x_n), (-1)^j a_n\ge 0\}$.  
We have that
  \begin{align}\label{kmetrkomwer}
      \abs{x^*\!\left(\sum_{n\in s}a_n x_n\right)} & \le \sum_{i,j\in \{0,1\}}\left| x^*\!\left(\sum_{n\in s_{i,j}} a_n x_n\right)\right| = \sum_{i,j\in \{0,1\}}(-1)^{i+j}\sum_{n\in s_{i,j}} a_n x^*(x_n) \le \nonumber \\
      & \le \sum_{i,j\in \{0,1\}}\max_{n\in s_{i,j}}|a_n| x^*\!\left(\sum_{n\in s_{i,j}}x_n\right) \le \sum_{i,j\in \{0,1\}}\max_{n\in s_{i,j}}|a_n|\cdot \norm{\sum_{n\in s_{i,j}} x_n} \le 4 K \max_{n\in s} |a_n|
  \end{align}
It follows from this and \eqref{ioh43th34rt} that
\[
\frac{\delta}{4}\max_{n\in s} |a_n|\le \norm{\sum_{n\in s} a_n x_n}\le 4K \max_{n\in s}|a_n|,
\]
and consequently, $(x_n)_{n\in N}$ is equivalent to the unit basis of $c_0$.

{\it (iii)} is trivial from the equivalence between {\it (i)} and {\it (ii)}, while {\it (iv)} holds because if $(x_n)_n$ is a weakly-null sequence of $c_0$, then every subsequence has a further subsequence that is either norm-null or equivalent to  the unit basis of $c_0$, so {\it (ii)} holds. 
\endproof
}

\begin{coro}
A $\mc B$-ideal is tall exactly when it can be $\mc B$-represented by an unconditional basis that is  saturated by  copies of the unit basis of $c_0$.    
\end{coro}
\begin{proof}
Let $\mc I$ be a $\mc B$-ideal. Suppose that $\mc I$ is tall.  We use Proposition \ref{lkmngtktjgkdlfgdf} to represent $\mc I=\mc B((x_n)_n)$ where $(x_n)_n$ is a 1-unconditional seminormalized basis of some Banach space $X$. It follows from Theorem \ref{c0saturacion} that $(x_n)_n$ has to be saturated by copies of the unit basis of $c_0$, that is every subsequence of $(x_n)_n$ has a further subsequence equivalent to the unit basis of $c_0$. 

Suppose that $\mc I=\mc B((x_n)_n)$ where $(x_n)_n$ is as in the statement of the Corollary. Then given $A\con \N$ infinite, we can choose a subsequence $(x_n)_{n\in B}$ of $(x_n)_{n\in A}$ that is $C$-equivalent to the unit basis of $c_0$, and in particular $\sup_{F\in [B]^{<\infty}}\nrm{\sum_{n\in F}x_n}\le C$, that is, $B\in \mc B((x_n)_n)$. 
\end{proof}


From this result, it is natural to ask whether every tall $\mathcal{B}$-ideal can be $\mathcal{B}$-represented in $c_0$. In Section \ref{nonc0}, we will show that there exist tall $\mathcal{B}$-ideals that are not $\mathcal{B}$-represented in $c_0$. More generally, we aim to understand $\mathcal{B}$-representations in spaces with many copies of $c_0$. The most natural examples, the spaces of continuous functions on a compact and countable space, will be analyzed in Subsection \ref{lmklfmklsdfs}.

On the opposite direction, we may have tall $\mc B$-ideals that can be $\mc B$-represented by a sequence in some space not containing isomorphic copies of $c_0$. This is one of the examples.    
 Given a {\em Banach lattice sequence space} $X$ (i.e. $X\con \R^\N$ is such that $X$ is a Banach lattice  with the lattice structure endowed as subspace of $\R^\N$) and a sequence of measures 
$\boldsymbol{\varphi}:=(\varphi_n)_n$ we define 
\[
\mc I({X,\boldsymbol{\varphi}}):=\conj{A\con \N}{(\varphi_n(A))_n\in X}.
\]
These are ideals: if $A\con B\in \mc I$, then $\varphi_n(B)\le \varphi_n(A)$ for every $n$, so $\nrm{(\varphi_n(B))_n}_X\le \nrm{(\varphi_n(A))_n}_X$, because $X$ is a Banach lattice, while if $A,B\in \mc I$, then each $\varphi_n(A\cup B)\le \varphi_n(A)+\varphi_n(B)$, hence, by unconditionality and the triangle inequality, 
$$\nrm{(\varphi_n(A\cup B))_n}_X\le \nrm{(\varphi_n(A))_n +(\varphi_n(B))_n}_X\le \nrm{(\varphi_n(A))_n}_X +\nrm{(\varphi_n(B))_n}_X.$$
Since we are assuming that all ideals contain $\fin$, we suppose that $(\varphi_n(\{k\}))_n\in X$ for every $k\in \N$, so $\fin \con \mc I(X,\boldsymbol\vphi)$. 
Observe that any $\mc B$-ideal can be represented this way, given that if $\mc I= \fin(\vphi)$, $\vphi=\sup_n \vphi_n$, then we have $\mc I= \mc I(\ell_\infty, (\vphi_n)_n)$ clearly.  Natural Banach lattices sequence spaces are spaces $E$ with a 1-unconditional Schauder basis $(e_n)_n$, because $E$ is obviously isometric to $X:=\conj{(a_n)_n\in \R^\N}{\nrm{\sum_n a_n x_n}_E<\infty}$ endowed with the norm $\nrm{(a_n)_n}_X:= \nrm{\sum_n a_n e_n}_E$. When  $X=c_0$, the corresponding ideal 
 $\mc I(X,\boldsymbol{\vphi})$ is $F_{\sigma\delta}$.  On the opposite we have the following.
\begin{prop}\label{mlkdmflkdsfds}
If $X$ does not have isomorphic copies of $c_0$, the ideal $\mc I(X,\boldsymbol{\vphi})$ is a  $\mc B$ and $\mc C$-ideal.   
\end{prop} 
\begin{proof}
 Let $Y:= (\oplus_n\ell_1)_X$, that is, the space of sequences $(y_n)_n\in \ell_1^\N$ such that $(\nrm{y_n}_1)_n\in X$, with the norm $\nrm{(y_n)_n}_Y:= \norm{(\nrm{y_n}_1)_n}_X$. For each $k$, let $y_k:= (\vphi_n(\{k\})u_k)_n\in Y$. Then $\mc I(X,\boldsymbol{\vphi})=\mc B(\bf y)$.  
 Observe that for a finite set $F\con \N$ we have that
$$\norm{\sum_{k\in F} y_k}_Y= \norm{\left(\sum_{k\in F} \vphi_n(\{k\})u_k\right)_n}_Y=\norm{\left(\norm{\sum_{k\in F} \vphi_n(\{k\})u_k}_{\ell_1}\right)_n}_X=
\norm{( \vphi_n(F))_n}_X.$$
suppose that $A\in \mc I(X,\vphi)$, that is, $(\vphi_n(A))_n\in X$. Since $(u_n)_n$ is a unconditional basis, it follows that $\sup_{F\in [A]^{<\infty}}\nrm{(\vphi_n(F))_n}_X\le \nrm{(\vphi_n(A))_n}_X<\infty$, hence $A\in \mc B(\mb x)$. Suppose that $A\in \mc B(\bf x)$, or in other words, $\sup_{F\in [A]^{<\infty}} \nrm{(\vphi_n(F))_n}_X =C<\infty$.  
The sequence $ ((\vphi_n(F))_n)_{F\in [A]^{<\infty}}$ of elements of $X$ is bounded in norm by $C$, and since $X$ does not contain $c_0$, it is order-complete, that is, $\bigvee_{F\in [A]^{<\infty}} (\vphi_n(F))_n$ belongs to $X$, but its supremum is equal to $(\vphi_n(A))_n$, proving that $A\in \mc I(X,\boldsymbol{\vphi})$.

Next, since neither $\ell_1$ nor $X$ have copies of $c_0$, it is a well known fact that $Y$ does not have isomorphic copies of $c_0$. It follows that $\mc B(\bf y)=\mc C(\bf y)$. 
\end{proof}
Now, by choosing  $X$ and $(\vphi_n)_n$ so that $\lim_{k\to \infty} \nrm{(\vphi_n(\{k\}))_n}_X=0$, it follows that the corresponding ideal $\mc B(\bf y)$ is a tall ideal that is represented in a space $Y$ without copies of $c_0$.  For example, let $X:=\ell_2$ and for each $n$, take the uniform probability 
$$\vphi_n(A):= \frac{ \#(A\cap [2^n-1, 2^{n+1}-1[)}{ 2^n}$$
and observe that $\nrm{y_k}_Y=\nrm{(\vphi_n(\{k\}))_n}_{\ell_2}= 1/2^n $, where $n$ is such that $2^n-1 \le k<2^{n+1}-1$.

\medskip 

A key feature of considering the uniform probability on each of the dyadic intervals $I_n$ is that the collection of possible values of probability is dense in the unit interval.

\subsection{\texorpdfstring{$\mc B$-representability vs summability}{B-representability vs summability}}


R. Filip\'ow and J. Tryba recently proved in \cite[Corollary 11.3]{FilipowTryba2024} that every non-trivial $\mc B$-ideal can be extended to a non-trivial summable ideal. Their proof is rather straightforward, as described in the paragraph preceding Remark \ref{i4jirj4ir4e}, where it is adapted to the $\mc B$-representation in spaces of continuous functions.
The extension of ideals to summable ones has led to intriguing characterizations, including   connections to Riemann summability (see \cite{FilipowSzuca2010}). Notably, several classical ideals, such as the Mazur ideal $\mc M$ and the ideal $\mathcal{Z}$, cannot be extended in this way, and consequently they are pathological. For additional examples, we refer the reader to \cite{FilipowSzuca2010}, \cite{FilipowTryba2024}.

Furthermore, it has been established in \cite{Borodulinetal2015} that every $F_{\sigma}$ ideal that is $\mathcal{C}$-representable in $c_0$ is necessarily summable. This result leads us to the following.

\begin{prop}
An ideal is $\mc B$-representable in a finite dimensional space exactly when it is summable.    
\end{prop}
\begin{proof}
Suppose that $\mathbf x=(x_n)_n$ is a sequence in a finite dimensional normed space $X$. Since $X$ is finite dimensional, it does not contain isomorphic copies of $c_0$, hence, by Drewnowski-Labuda in \cite{Drewnowski} we obtain that  $\mc B(\bf x) = \mc C(\bf x)$. This can be easily seen directly: 
without loss of generality, we assume that $X=(\R^N, \nrm{\cdot}_\infty)$, $N\in \N$. 
Assume that $M\notin\mc C(\bf x)$ for some sequence $\bf x$ in $X$, does not converge unconditionally, i.e. there exists $\vep > 0$ such that for every $F\subseteq M$ finite there is $G$ finite and disjoint from $F$ such that $\norm{\sum_{n\in G} x_n}\geq\vep$. By recursion we can construct a sequence of finite subsets $G_0 < G_1 < \cdots$ such that $\norm{\sum_{n\in G_k} x_n}\geq\vep$ for all $k$. For each $k$ there is a coordinate $1\leq j\leq N$ such that $|\sum_{n\in G_k} x_n(j)|\geq \vep/N$, and passing to a subsequence we can assume that $j$ is the same for all $k$. For each $k$ we can now find $G_k'\subseteq G_k$ and a sign $\theta_k\in\{-1,1\}$ such that $\theta_k \sum_{n\in G_k'} x_n(j)\geq \vep/2N$. Again passing to a subsequence we can assume that $\theta_k = 1$ for all $k$, and therefore
\[
\sup\conjbig{\norm{\sum_{n\in F} x_n}}{\text{$F\subseteq M$ finite}} \geq \norm{\sum_{n\in\bigcup_{i=0}^k G_k'} x_n} 
\geq \sum_{i=0}^k \sum_{n\in G_k'} x_n(j) \geq \frac{k\vep}{2N}.
\]
Since $k$ here is arbitrary, $M\notin \mc B(\bf x)$.

Now, since  $X$ is finite dimensional, a series is unconditionally convergent exactly when it is absolutely convergent, that is, $\mc B(\mathbf x)= \mc C(\mathbf x)= \mr{Sum}((\nrm{x_n})_n)$.
\end{proof}

We have mentioned before that every $\mc B$-ideal can be represented by an unconditional Schauder basis (see Proposition \ref{lkmngtktjgkdlfgdf}). If this can be done in a Hilbert space, we have the following.

\begin{prop}
An ideal can be  $\mc B$-represented in a Hilbert space by an unconditional basic sequence exactly when it is summable.    
\end{prop}
\begin{proof}
Suppose that $\mc I=\mc B(\bf x)$ with $\bf x$ an unconditional basic sequence in a Hilbert space $H$. It is well known --see after this proof--  that unconditional bases of Hilbert space are equivalent to an orthogonal basis. Hence, there is some constant $K\ge 1$ such that  $(1/K)\sum_{n\in A}\nrm{x_n}^2\le \nrm{\sum_{n\in A} x_n}^2\le K \sum_{n\in A} \nrm{x_n}^2$ for every $A\con \N$. This means that $\mc B({\bf x})= \mr{Sum}((\nrm{x_n}^2)_n)$. 
\end{proof}
The proof of the uniqueness of unconditional basic sequences in a Hilbert space follows directly from the {\em generalized parallelogram identity}: for every sequence $(y_k)_{k=1}^n$ in a Hilbert space we have that
\begin{equation}\label{jo4i5tjoijoitrg}
    \mathbb E_{(\theta_k)_k\in \{-1,1\}^n}\!\left(\norm{\sum_{k=1}^n \theta_k y_k}^2\right) := \frac{1}{2^n}\sum_{(\theta_k)_k\in \{-1,1\}^n} \norm{\sum_{k=1}^n \theta_k y_k}^2= \sum_{k=1}^n \nrm{y_k}^2,
\end{equation}
that is, the average $\mathbb E_\theta \nrm{\sum_{k}\theta_k y_k}^2$ of the square of the norm of perturbation by signs is the sum of the square of the norms.  Hence, if $(x_k)_k$ is a $C$-unconditional basic sequence, then for every sequence of scalars $(a_k)_{k=1}^n$ we have that 
$C^{-2}\nrm{\sum_{k=1}^n a_k x_k}^2\le \mathbb E_{(\theta_k)_k}\left( \nrm{\sum_{k=1}^n \theta_k a_k x_k}^2\right)\le {C}^2\nrm{\sum_{k=1}^n a_k x_k}^2$, and since $\mathbb E_{(\theta_k)_k}\left( \nrm{\sum_{k=1}^n \theta_k a_k x_k}^2\right)= \sum_{k=1}^n a_k^2 \nrm{x_k}^2$, we conclude that $(x_k)_k$ is $C$-equivalent to an orthogonal sequence. 

It is natural to ask for properties similar the generalized parallelogram identity in a given Banach space.  Recall that a Banach space $X$ has {\em cotype} $1\le q\le \infty$ if there is a constant $C$ such that  for every finite sequence $(x_k)_{k=1}^n$ one has that 
\begin{equation}\label{orjerewrfewr}
    \sum_{k=1}^n \nrm{x_k}^q \le C \mathbb E_\theta\!\left(\norm{\sum_{k=1}^n \theta_k x_k}^q\right),
\end{equation}
where when for $q=\infty$ the previous inequality has to be interpreted as $  \max_{k=1}^n\nrm{x_k}\le C \mathbb E_\theta \nrm{\sum_{k=1}^n \theta_k x_k}$. This means, by a simple use of the triangle inequality, that every space has cotype $\infty$, and by using constant sequences, that cotypes $q$ are always at least 2.   
A Banach space $X$ has {\em non trivial cotype} when it has cotype $q$ for some $q<\infty$.  By the Kahane-Khintchine inequality, it follows that $X$ has non trivial cotype $2\le q <\infty$ when there is a constant $C>0$ such that for every finite sequence $(x_k)_{k=1}^n$ in $X$ such that 
\begin{equation}\label{orjerewrfewr1}
    \left(\sum_{k=1}^n \nrm{x_k}^q \right)^\frac1q\le C \mathbb E_\theta\!\left(\norm{\sum_{k=1}^n \theta_k x_k}\right).
\end{equation}
We have the following. 

\begin{prop}
For every sequence $\mathbf x=(x_n)_n$ in a space $X$ with non-trivial cotype $q$ one has that 
$$\mr{Sum}((\nrm{x_n})_n)\con \mc B(\mathbf{x})\con \mathrm{Sum}((\nrm{x_n}^q)_n).$$
\end{prop}
\begin{proof}
The first inclusion follows trivially from the triangle inequality. 
Suppose now  that $A\in \mc B(\bf x)$, that is $\sup_{F\in [A]^{<\infty}} \nrm{\sum_{k\in F}x_k}=K<\infty$. Then, by  the $q$-cotype of $X$, for every finite set $F\con A$ one 
that 
\[
\left(\sum_{k\in F}\nrm{x_k}^q\right)^\frac1q \le C \mathbb E_{\theta\in \{-1,1\}^F }\norm{\sum_{k\in F} \theta_k x_k} \le 2C\mathbb E_{G\con F}\norm{\sum_{k\in G} x_k}\le 2 C K,
\]
where we have used that  it follows from the triangle inequality that  
\begin{equation}\label{oi43tio34tj5t45}
    \mathbb E_{(\theta_k)_k\in \{-1,1\}^n} \left(\norm{\sum_{k=1}^n \theta_k y_k}\right) \le 2 \mathbb E_{A\con \{1,\dots,n\}} \left(\norm{\sum_{k\in A} y_k}\right).  
\end{equation}
This implies that $\sum_{k\in A} \nrm{x_k}^q<\infty$, that is $A\in \mr{Sum}((\nrm{x_n}^q)_n)$.
\end{proof}
A remarkable result by B. Maurey and G. Pisier \cite{MauPis} states that a space $X$ has non-trivial cotype exactly when $c_0$ is not {\em finitely representable}, that, in the case of $c_0$, means that there is no constant $C$ such that every $\ell_\infty^n$ has  $C$-isomorphic subspace of $X$.\footnote{By the James distortion theorem for $c_0$, a space has an isomorphic copy of $c_0$ iff it has an almost isometrical copy of $c_0$; from here it follows that for every $\vep>0$, $C>0$ and $m\in \N$ there is $n\in \N$ such that if $F$ is $C$-isomorphic to $\ell_\infty^n$ then $F$ has a subspace $1+\vep$-isomorphic to $\ell_\infty^m$.} Combining this with the  Drewnowski-Labuda Theorem \cite{Drewnowski} we obtain the following. 

\begin{coro}
If $c_0$ is not finitely representable in $X$, then there is some $2\le q<\infty$ such that     
$$
\mr{Sum}((\nrm{x_n})_n)\con \mc B(\mathbf x)=\mc C(\mathbf x)\con \mathrm{Sum}((\nrm{x_n}^q)_n)
$$ for every sequence $\mathbf x=(x_n)_n$ in $X$. \qed
\end{coro}

\begin{question}
\label{question-B-Sumable}
Suppose that an ideal is $\mc B$-representable on a space with non trivial cotype $q$. Is it possible to do it in such a way that the corresponding sequence $(x_n)_n$ satisfies that $(\nrm{x_n}^q)_n$ is not summable? 
\end{question}

\section{\texorpdfstring{Representing non-pathological $F_\sig$ ideals on spaces of continuous functions}{Representing non-pathological Fsigma ideals on spaces of continuous functions}}
\label{repre}

In this section we see how to represent effectively $\mc B$ and $\mc C$-ideals  in a space of continuous functions. The effectiveness will come from the fact that we will use  natural Schauder bases of certain spaces of continuous functions.

\begin{defi}[Evaluation sequence]\label{j4iotorejgidf}
For $K\con [0,1]^\N$ compact we define the corresponding {\em evaluation sequence} $\mb p^K=(p_n^K)_n$ in $C(K)$ by $p_n^K(x):=x(n)$ for $x\in K$. 
\end{defi}  
In particular, when $K$ is a closed subset of the Cantor space $2^\N$ then, identifying it with a family of subsets of $\N$, for $F\subseteq\N$ finite we get the expression
\begin{equation}
\label{eq:simpler_expression_p^K}
    \norm{\sum_{n\in F} p_n^K}_{C(K)} = \sup_{A\in K} |F\cap A|.
\end{equation}
Thus, $M\in\mc B(\mb p^K)$ if and only if there is some $n\in\N$ such that $|M\cap A|\leq n$ for all $A\in K$.

Let $\varphi$ be non-pathological lscsm, $\varphi=\sup_{k}\mu_k$, each $\mu_k$ a finite measure. Define now a measure $\nu_k$ on $\N$ by $\nu_{2k}(\{n\}):=\min \{\mu_k(\{n\}), 1\}$, and $\nu_{2k+1}:= \delta_{\{k\}}/(k+1)$. It is easy to see that $\fin(\varphi)=\fin(\sup_k \nu_k)$ and $\mr{Exh}(\varphi) = \mr{Exh}(\sup_k \nu_k)$. Consequently, without loss of generality, we assume that $\sup_{k,n}\mu_k(\{n\})\le 1$ and that for every $n$ there is some $k_n$ such that $\mu_{k_n} = \delta_{\{n\}}/(n+1)$. 
Define $f_k:\N\to [0,1]$, $f_k(n):=\mu_k(\{n\})$ for every $n$. Now let $\boldsymbol \mu:=(\mu_n)_n$, $K:=K(\boldsymbol\mu)$,  be the pointwise closure of $\{f_k\}_k$ in $[0,1]^\N$.  Then $\norm{\sum_{n\in F} p_n^K}_{C(K)} = \varphi(F)$ for all $F\subseteq\N$ finite, so again $\mc B(\mb p^K) = \fin(\varphi)$ and $\mc C(\mb p^K) = \mr{Exh}(\varphi)$.

There is another Banach space associated canonically to the submeasure $\varphi=\sup_k \mu_k$: on $c_{00}(\N)$, define for each $x\in c_{00}$
\[
\nrm{x} := \sup_{k\in\N} |\langle f_k,x\rangle| =\sup_k \abs{\sum_n f_k(n) x(n)}.
 \]
This is a norm because a multiple of each coordinate function 
\[
u_n:=(0,0,\dots, 0, \overset{(n)}1,0,0,\dots) =(n+1)f_{k_n}.
\] 
Let $X:=X(\boldsymbol \mu)$ be the completion of $(c_{00}, \nrm{\cdot})$. Since $\nrm{\sum_{n\in F} u_n}=\varphi(F)$, we obtain that $    \fin(\varphi)=\mc B((u_n)_n)$ and $\mr{Exh}(\varphi)=\mc C((u_n)_n)$.
More precisely, observe that the sequences $(u_n)_n$ in $X$ and $(p_n^K)_n$ in $C(K)$ are equivalent, that is, $\nrm{\sum_{n\in F} a_n u_n}_X = \nrm{\sum_{n\in F} a_n p_n^K}_{C(K)}$ for all sequences of coefficients $(a_n)_n\in c_{00}$.


For the next result, let $\boldsymbol \mu=(\mu_n)_n$ be a sequence of finite measures on $\N$, $\varphi=\sup_{k}\mu_k$ be the corresponding non-pathological lscsm,  and  $X$ and $K$ be, respectively, the space and the compact set introduced in the previous paragraph. 

\begin{prop}
\label{SpaceXmu}
The following three conditions are equivalent, and they hold if $\fin(\varphi)$ is tall:
\begin{enumerate}[(i), wide=0pt]

\item $(u_k)_k$ is weakly-null in $X(\boldsymbol{\mu})$.
\item $\mb p^K$ is weakly-null in $C(K(\boldsymbol{\mu}))$.
\item $K(\boldsymbol{\mu})$ is a compact subset of the unit ball $B_{c_0}$ of $c_0$ with its weak topology. 

\end{enumerate}

\end{prop}
\begin{proof}
We see the equivalence between those three conditions. {\it (i)} and {\it (ii)} are equivalent because we have observed that the sequences $(u_n)_n$ and $(p_n^K)_n$ are equivalent. {\it (ii)} implies {\it (iii)} trivially, and {\it (iii)} implies {\it (i)} because, by the Rainwater Theorem, it suffices to show that $p_n^K(x)\to_n 0$ for every $x\in K$, and this is the case when $K\con c_0$. The fact that the tallness of $\fin(\varphi)=\mc B((u_n)_n)$ implies that $(u_n)_n$ is weakly null was done in Theorem \ref{c0saturacion}.
\end{proof}

\bigskip

\subsection{\texorpdfstring{An explicit representation in $C[0,1]$ and $C(2^\N)$}{An explicit representation in C[0,1] and C2N}}

By the general results in \cite{Borodulinetal2015,MMU2022}, $\B({\mb g})$ and $\mathcal{C}({\mb g})$ are non-pathological for any $\mathbf g=(g_n)_n$ in $C(\cantor)$.  We present an explicit construction of a  non-pathological lscsm $\varphi$ such that $\B({\mb g})=\fin(\varphi)$ and $\mathcal{C}({\mb g})=\exh(\varphi)$. Conversely, we show how to {\em explicitly} represent non-pathological ideals on the universal spaces $C(\cantor)$ or $C[0,1]$, answering  Question 7.1 from \cite{Borodulinetal2015}. We start with the following fact  generalizing   \cite[Lemma 4.3]{Borodulinetal2015}. In the following statement $K$ denotes a separable compact space.

\begin{prop}\label{prop:positive_functions}
    For every sequence $\mathbf g=(g_n)_n$ of functions in $C(K)$ one has that $\mc C(\mb g)=\mc C(\mb{|g|})$ and $\mc B(\mb g) = \mc B(\mb{|g|})$.
\end{prop}
\begin{proof}
 Clearly, $\mathcal{C}(|{\mb g}|)\subseteq \mathcal{C}({\mb g})$. Suppose $A\notin  \mathcal{C}(|{\mb g}|)$. Then there is $\varepsilon>0$ such that for all $F\subseteq A$ finite there is $G\subseteq A$ finite and disjoint from $F$ such that 
\(
\|\sum_{n\in G}|g_n| \|_\infty\geq \varepsilon.
\)
For a fixed finite subset $F$ of $\N$, let $G\subset \N$ be  disjoint from $F$ and such that $\|\sum_{n\in G} |g_n|\|_\infty\ge \varepsilon$. Let $p\in K$ and $G_0\subset G$  be such that  $|\sum_{n\in G_0} g_n(p)|=\sum_{n\in G_0} |g_n(p)|\ge \varepsilon/2$. Since $F\subset A$ was arbitrary, this proves that $A\notin \mathcal C({\mb g})$. The proof for $\mc B(\mb g)$ is analogous.
%
\end{proof}

The previous result does not hold for Banach lattices that are not function spaces, as demonstrated in the following example.

\begin{ex}\label{ex:ideal_not_positive_coords}
    In \cite[Section 6]{Borodulinetal2015} an example is shown of a $\mc C$-representable ideal in $\ell_1$ which cannot be represented with non-negative elements. The {\em Rademacher ideal} is constructed as follows: for each $n\in\N$ denote by $\Delta_n := \sum_{i=0}^n i$ the $n$-th triangular number. Let $x_0 := (1,0,0,\ldots)$ and, for each $k\geq 1$, define $x_k\in\ell_1$ as
    \[
    x_k(j) := \left\{\begin{array}{ll}
        \frac{1}{n 2^n} (-1)^{\lfloor j/2^{\Delta_{n+1} - k}\rfloor} & \text{if $\Delta_n \leq k < \Delta_{n+1}$ and $2^n - 1\leq j < 2^{n+1} - 1$,} \\
        0 & \text{otherwise.}
    \end{array}\right.
    \]
    Explicitly,
    \[
    x_1 = \frac{1}{2} (0, 1, 1, 0, \ldots),\ x_2 = \frac{1}{2} (0, 1, -1, 0, \ldots),
    \]
    \[
    x_3 = \frac{1}{8} (0,0,0,1,1,1,1,0,\ldots),\ x_4 = \frac{1}{8} (0,0,0,1,1,-1,-1,0,\ldots),\ x_5 = \frac{1}{8} (0,0,0,1,-1,1,-1,0,\ldots),\ \ldots
    \]
    Consider now the Rademacher ideal $\mc I_R := \mc C(\mb x) = \mc B(\mb x)$. It can be shown, through a Khintchine inequality that $\mc I_R$ is not summable, yet clearly every ideal representable in $\ell_1$ with non-negative vectors is summable.
\end{ex}

When $K$ is an uncountable compact metric space, the corresponding function space $C(K)$ is universal for separable Banach spaces (being isomorphic to $C[0,1]$), so     we can represent every non-pathological  ideal in $C(K)$.
We see now how to explicitly define from $\mb g$ in $C(K)$ a sequence of submeasures  defining the ideals $\mathcal B(\mb g)$ and $\mathcal C(\mb g)$.  Let $\mathbf g=(g_n)_n$ in $C(K)$ with  $g_n\geq 0$ for all $n$. 
Let $\Gamma$ be a countable dense subset of $K$ and  let $(\alpha_k)_k$ be an enumeration  of $\Gamma$ such that each element of $\Gamma$ appears infinitely often in the enumeration.  For each $k$ and $A\subseteq \N$, consider the following measure:
\begin{equation}
\label{codes1}
\mu_k(A):= \sum_{n\in A\cap k}g_n(\alpha_k).
\end{equation}
%


Let $\varphi=\sup_k \mu_k$. First we show that $\B({\mb g)}=\fin(\varphi)$. In fact,  suppose $A\in \B({\mb g})$. Since $\mu_k(A)\leq \|\sum_{n\in A\cap k}g_n\|$,  clearly  $\varphi(A)<\infty$. Reciprocally, suppose $\varphi(A)\leq M$. Let 
$\beta\in \cantor$ and $F\subset A$ finite. Pick $\alpha\in \Gamma$ such that 
\[
\abs{\sum_{n\in F} (g_n(\beta)-g_n(\alpha))}\leq 1.
\]
Pick $k$ such that $\max(F)< k$ and $\alpha_k=\alpha$. Then 
\[
\abs{\sum_{n\in F}g_n(\beta)} \leq
\abs{\sum_{n\in F} (g_n(\beta)-g_n(\alpha_k))} + \abs{\sum_{n\in F}g_n(\alpha_k)} \leq 1+ 
\sum_{n\in A\cap k}g_n(\alpha_k)\leq 1+M.
\]
Thus $\|\sum_{n\in s}g_n\|_\infty \leq 1+M$ and $A\in \B({\mb g})$.

\medskip 

Now we verify that  $\mathcal{C}({\mb g})=\exh(\varphi)$.  Let $A\in \exh(\varphi)$. Fix $\varepsilon>0$ and let $n_0$ be such that $\varphi(A\setminus\{0,\ldots,n\})\leq \varepsilon/2$ for all $n\geq n_0$. We will show that 
\[
\norm{\sum_{n\in G} g_n}_\infty\leq\varepsilon
\]
for all finite set $G\subseteq A$ with $\min(G)>n_0$. Fix a finite set $G\subseteq A$ with $\min(G)>n_0$. We need to show that 
\[
\sum_{n\in G}g_n (\beta) \leq\varepsilon
\]
for all $\beta\in K$. Fix $\beta\in K$ and pick $k$ such that $\max(G)<k$ and  
\[
\abs{\sum_{n\in G}(g_n(\beta)-g_n(\alpha_k))} \leq \varepsilon/2.
\]
Then
\begin{align*}
    \sum_{n\in G}g_n (\beta) & \leq \abs{\sum_{n\in G}(g_n(\beta)-g_n(\alpha_k))} + \sum_{n\in G}g_n(\alpha_k)\leq \varepsilon/2 + \sum_{n\in k\cap A\setminus\{0,\ldots,n_0\}}g_n(\alpha_k) \\
    & = \varepsilon/2 + \mu_k(A\setminus\{0,\ldots,n_0\}) \leq \varepsilon/2 + \varphi(A\setminus\{0,\ldots,n_0\}) \leq \varepsilon.
\end{align*}
Conversely, let $A\in \mathcal{C}({\mb g})$.  Fix $\varepsilon>0$ and let $n_0$ be such that
\[
\sum_{n\in G}g_n (\beta) \leq \varepsilon
\]
for all $\beta\in K$ and all $G\subseteq A$ with $\min(G)>n_0$. We will show that  $\varphi(A\setminus\{0,\ldots,n_0\})\leq \varepsilon$. Fix $k\in \N$ and let $G=(k\cap A)\setminus\{0,\ldots,n_0\}$. Then 
\[
\mu_k(A\setminus\{0,\ldots,n_0\})=\sum_{n\in G}g_n(\alpha_k)\leq \varepsilon.
\]


%


Now we treat the other face of the representation of  non-pathological ideals. Let $\varphi=\sup_k \mu_k$ be a lscsm, where $(\mu_k)_k$  is a sequence of measures on $\N$, we will define a sequence $(g_n)_n$ in $C([0,1])$ such that $\B({\mb g})=\fin(\varphi)$ and $\mathcal{C}({\mb g})=\exh(\varphi)$.    We start with the following useful simplification of the lscsm representing the ideals   $\exh(\varphi)$ and $\fin(\varphi)$.

\begin{lema}
\label{finitevalues}
Let $\varphi=\sup_k\mu_k$ be a lscsm where $(\mu_k)_k$ is a sequence of measures on $\N$.    There is another sequence of measures $(\lambda_k)_k$ on $\N$ such that 
\begin{enumerate}[(i), wide=0pt]

\item $\lambda_k\leq \mu_k$ for all $k$.

\item   $\fin(\varphi)=\fin(\psi)$  and $\exh(\varphi)=\exh(\psi)$
where $\psi=\sup_k \lambda_k$.

\item  $\{\lambda_k(\{n\}):\; k\in \N\}$ is finite for all $n$.
\end{enumerate}
\end{lema}

\proof
Notice that  $\mu_k(\{n\})\leq \varphi(\{n\})$ for each $k$ and $n$. Consider the following measure $\lambda_k$.
\[
\lambda_k (\{n\}) :=
\begin{cases}
    i/2^{n} & \mbox{if $i/2^{n} < \mu_k(\{n\})\leq (i+1)/2^{n}$ for some $0\leq i< \lfloor 2^n\varphi(\{n\})\rfloor$}.\\
    0 & \mbox{if $\mu_k(\{n\})=0$.}
\end{cases}
\]
Let $\psi=\sup_k\lambda_k$.  Clearly  $\lambda_k\leq \mu_k$. Thus $\fin(\varphi)\subseteq \fin(\psi)$ and $\exh(\varphi)\subseteq \exh(\psi)$. 
Notice that  $|\mu_k(\{n\})-\lambda_k(\{n\})|\leq 1/2^n$ for all $n$ and $k$. Let $F\subseteq \N$ be a finite set. Then 
\[
\mu_k(F)= (\mu_k(F)-\lambda_k(F))+\lambda_k(F)\leq \sum_{n\in F}1/2^n +\psi(F).
\]
Thus $\fin(\psi)\subseteq \fin(\varphi)$ and $\exh(\psi)\subseteq \exh(\varphi)$.
\endproof

We see now how to represent explicitly the ideals $\fin(\varphi)$ as $\mathcal B(\mb g)$ on $C([0,1])$. For each $n$, let $L_n=\{\mu_k(\{n\}):\; k\in\N\}$ which we assume to be finite   by Lemma \ref{finitevalues}. Consider the following  tree: 
\[
T:=\{\langle \mu_k(0),\ldots, \mu_k(n)\rangle:\; k,n\in \N\}.
\]
Since each $L_n$ is finite, $T$ is clearly finitely branching. Notice that  each measure $\mu_k$ is a branch of $T$. 
 Let $\rho:T\to 2^{<\omega}$ be an embedding, that is, $\rho(s)\prec \rho(s')$ iff $s\prec s'$, for all $s,s'\in T$.

%

We are going to find a sequence ${\mb g}= (g_n)_n$ of continuous functions in $[0,1]$ such that $\|\sum_{n\in F} g_n \|_\infty= \sup_k \mu_k(F)$ for every finite set $F\subset \N$, that easily gives that $\B({\mb g})=\fin(\varphi)$ and $\mathcal{C}({\mb g})=\exh(\varphi)$.

Let  $\theta: \cantor \to [0,1]$ be the binary representation map, $\theta((\alpha_n)_n):= \sum_{n=0}^\infty \alpha_n 2^{-(n+1)}$, and for each $s\in 2^{<\omega}$, let $I_s:= \theta([s])$, that is, $I_s= [\theta(s\frown \bar 0), \theta(s\frown \bar 0)+ 1/2^{|s|}] $, where we recall that $|s|$ is the length of $s$. We find now a sequence $(J_s)_{s\in 2^{<\omega}}$ of closed non-degenerate intervals and a sequence $(f_s)_{s\in 2^{<\omega}}$ of continuous functions in $[0,1]$  such that 

\begin{enumerate}[a), wide=0pt]
	\item  each $J_s\subset  I_{t_s}$ in a way that the mapping $s\in 2^{<\omega}\mapsto t_s\in 2^{<\omega}$  is an embedding,
	\item $f_s\ge 0$, $\|f_s\|_\infty=1$, $f_s\upharpoonright J_s=1$  and $\mathrm{supp}f_s\subseteq \overset{\circ}{I_{t_s}}$.    
\end{enumerate}
Let  $\gamma: \cantor \to \cantor$ be the function $\gamma(\beta):= \bigcup_{n\in \N} t_{\beta\upharpoonright n}$. Observe that  $\{\theta(\gamma(\beta))\}= \bigcap_{n\in \N} J_{\beta\upharpoonright n}$.  It can be easily checked that 
\begin{enumerate}[a), wide=0pt]\addtocounter{enumi}{2}
	\item $f_s(\theta(\alpha))=0$ if $t_s\not\prec \alpha$, and $f_s(\theta(\alpha))\le 1$ always. 
	\item $f_s(\theta(\gamma(\beta)))=1$ if $s\prec \beta$. 
\end{enumerate}  
We define now 
\begin{equation}\label{iohg5udogdgf1}
g_n :=\sum_{s\in T,|s|=n+1}\, s(n)f_{\rho(s)}.    
\end{equation}
Fix a finite set $F\subset \N$. Next, we prove that $\|\sum_{n\in F} g_n \|_\infty=\sup_k \mu_k(F)$. 
Given $\alpha\in \cantor$ and $n\in \N$, observe that if $g_n(\theta(\alpha))\neq 0$ then there must be a unique $s\in T$ of length $n+1$ such that $t_{\rho(s)}\prec\alpha$. It follows from this that 
\[
\sum_{n\in F} g_n(\theta(\alpha))= \sum_{n\in F_0} g_n(\theta(\alpha)) =\sum_{n\in F_0} \mu_k(\{n\})g_n(\theta(\alpha))    \le \sum_{n\in F_0} \mu_k(\{n\})=\mu_k(F).
\]
where as before $F_0:= \max \{ n\in F \, : \, g_m(\theta(\alpha))\neq 0\}$ and $k\in \N $ is such that $t_{\rho(\langle \mu_k(0),\dots, \mu_k(m)\rangle)} \prec \alpha$.  Consequently, $\|\sum_{n\in F} g_n\|_\infty \le \sup_k\mu_k(F)$. We see now the other inequality. For each $k$, let 
\[
\beta_k=: \bigcup_n \rho (\langle \mu_k(0),\ldots, \mu_k(n)\rangle)\in \cantor.
\]
Then    we have that 
\[
\sum_{n\in F}g_n (\theta(\gamma(\beta_k))) =\sum_{n\in F}\mu_k(\{n\})=\mu_k(F).
\]
Thus 
\[
\left\|\sum_{n\in F} g_n\right\|_\infty \ge \sup_k\mu_k(F)=\varphi(F).
\]
Now if we define each $h_n: =g_n \circ \theta\in C(\cantor)$, then for each finite $F\subset \N$ we have that 
\[
\norm{\sum_{n\in F} h_n}_\infty = \max_{\alpha \in \cantor} \abs{\sum_{n\in \N} h_n(\alpha)} = \max_{\alpha \in \cantor} \abs{\sum_{n\in \N} g_n(\theta(\alpha))} = \norm{\sum_{n\in F} g_n}_\infty = \sup_k \mu_k(F),
\]
and consequently, $\mathcal B((h_n)_n)=\fin(\varphi)$ and $\mathcal C((h_n)_n)=\mathrm{Exh}(\varphi)$.  

It is worth pointing out that for representing the ideal on $C(\cantor)$ the simpler, more explicit functions
\begin{equation}\label{iohg5udogdgf}
g_n :=\sum_{s\in T,|s|=n+1}\, s(n)\mathbbm 1_{[\rho(s)]}.    
\end{equation} 
will also work.

 We know from \cite{FilipowTryba2024}   that every  non-pathological, non-trivial $F_\sigma$ ideal is contained in a non-trivial summable ideal. We see now how to get this fact for ideals represented as $\B({\bf g})$   following an argument  similar to the one used in their proof. 


 Let $\mathbf g=(g_n)_n$ in $C(K)$ with $K$ a compact metric space, $g_n\geq 0$. Observe that for every sequence $\mb x:=(x_k)_k$ in $K$ we have a measure $\mu_{\mb x}$ on $\N$ defined by $\mu_{\mb x}(\{n\}):= \sum_{k\in \N} g_n(x_k)/2^k$ (in fact, $\mu_{\mb{x}}(A) \le 2\nrm{\sum_{n\in A} g_n}$), so if it is the case that  $\B({\bf g})$ is non trivial, then we can choose $\mb x:=(x_k)_k$ in $K$ and some block sequence $(B_k)_k$ of $\N$ such that $\sum_{n\in B_k} g_n(x_k)\ge 2^k$. This means that $\mu_{\mb x}(B_l)=\sum_{n\in B_l}\sum_{k\in \N} g_n(x_k)/2^k \ge \sum_{n\in B_l}g_n(x_l)/2^l\ge 1$, and since the sets $(B_l)_l$ are pairwise disjoint, $\mu_{\mb x}(\N)=\infty$, that is, $\mc B( \mb g)\con \mr{Sum}(\mu_{\mb x})\subsetneq \mc P(\N)$.

\begin{rem}[Representation on $C_p(K)$]\label{i4jirj4ir4e}
Another  well-known representation of analytic ideals is the following: given a Polish space $X$, a sequence $(f_n)_n$ in $C_p(X)$ and a function $f\in C_p(X)$ one defines $\mc I:=\mc I((f_n),f)$ as the collection of all $A\con \N$ such that $f\notin \overline{\conj{f_n}{n\in A}}$. In \cite[Lemma 6.53]{Todorcevic2010} it is shown that every analytic ideal is of that form. In the case of an $F_\sig$-ideal $\mc I$ this can be done quite directly as follows:   Let $K$ be an hereditary closed subset of $\cantor$ containing all singletons and such that $\ideal:=\ideal_{K}$ is the ideal generated by $K$ (see \cite{Mazur91} or the proof of Proposition \ref{rep-of-closed}).   We claim that  $\mc I=\mc I((p_n^K)_n,0)$, where the closure is in $C_p(K)$ with the pointwise topology. In fact, let $E\in \ideal$ and consider  $V(E)=\{f\in C_p(K): f(E)=0\}$.
Then $V(E)$ is an open set, $p_\infty\in V(E)$ and $p_n\not\in V(E)$ for all $n\in E$. Conversely, let $E\not\in \ideal$ and $V$ be an open set with $p_\infty\in V$. We can assume that $V=V(E_1)\cap\ldots\cap V(E_k)$ with each $E_i\in K$. As $E\not\in \ideal$, there is $n\in E\setminus (E_1\cup\ldots \cup E_k)$. Then, $p_n(E_i)=0$ for all $i\leq k$, that is,  $p_n\in V$.

In contrast,  $\mc B(\mb p^K)$ can be quite different from $\mc I$.
When $\ideal$ is tall, then $\B(\mb p^K)=\fin$ (otherwise, every infinite set contains an infinite element $A$ of $K$, in which case $\|\sum_{n\in F} p_n^K\| = \#F$ for all $F\subseteq A$), and if $\ideal$ is not tall, it might happen that $\mc B(\mb p^K)\not\con \mc I$ (for example, consider the ideal $\ideal$ on $\N$ generated by an infinite partition of $\N$ consisting of infinite pieces).
\end{rem}

The fact that we just proved that $\B$-representation can be performed in 
$C(2^{\N})$ allows us to inquire about the complexity of families of ideals. For instance, consider the collection of summable ones. Since each summable ideal is clearly  non-pathological (and thus a $\B$-ideal), the question arises:

\begin{question}
Is  
$\{{\bf g}=(g_n)_n\in (C(2^{\N}))^{\N}: \B({\bf g})\mbox{ is summable}\}$  a Borel set?
\end{question}

\section{Coloring ideals}\label{sect:colorings}
 Given an ideal $\mc I$ its characteristic function $\mathbbm 1_{\mc I}$ can be seen as a coloring  $\mathbbm 1_{\mc I}: [\N]^\omega\to 2$.  If in addition $\mc I$ is tall, then $\mc I\cap [\N]^\omega$ coincides with the collection of  $\mathbbm 1_{\mc I}$-homogeneous  sets. Recall that $M\in [\N]^\omega$ is $c$-homogeneous for a given finite coloring $c: [\N]^\omega\to r\in \N$, when $c$ is constant on the cube $[M]^\omega$. In this section we study what we call {\em coloring ideals}. The  {\em continuous} coloring ideals are especially relevant as they provide new examples of pathological and non-pathological $F_\sig$ ideals, and exhibit interesting connections with graph theory.

\begin{defi}[Homogeneous sets, coloring ideal]
Let $c:[\N]^\omega\to r\in \N$. For $i<r$ we define $\Hom_i(c)$ as the collection of $c$-homogeneous subsets of $\N$ with constant value $i$. Let $\langle \Hom(c)\rangle$ be the ideal generated by $\Hom(c):=\bigcup_{i<r}\Hom_i(c)$.   
A {\em coloring ideal} is an ideal  of the form $\langle \Hom(c)\rangle$ for some coloring $c$.    
\end{defi}

Observe that when $c$ is an analytic coloring (i.e. $c^{-1}(i)$ is an analytic set), then it follows from the Ellentuck-Silver Theorem that $\langle \Hom(c)\rangle$ is a tall ideal. Reciprocally, as we have observed above, if $\mc I$ is a tall ideal, then $\mc I=\langle \Hom(\mathbbm 1_{\mc I})\rangle$ is a coloring ideal, and its complexity obviously coincides with that of the coloring $\mathbbm 1_{\mc I}$.

One particular case is that of \emph{open} colorings, those $c: [\N]^\omega\to 2$ where $c^{-1}(1)$ is an open set. These colorings are always induced by families of finite subsets of $\N$. More precisely, let $\mc F$ be the family of all finite subsets $s$ of $\N$ such that $\langle s\rangle:=\conj{M\in [\N]^\omega}{s\sqsubset M}\subseteq c^{-1}(1)$. Then $c(M) = 1$ if and only if $M$ has an initial segment in $\mc F$, so the Ramsey property of the coloring $c$ is in this case equivalent to the following result. We denote by $K\rest A := K\cap \mc P(A)$ the \emph{restriction} of a set $K\subseteq 2^\N$ to a subset $A\subseteq\N$.

\begin{teo}[Galvin~\cite{Galvin1968}]\label{Galvin}
    Let $\mathcal{F}\subseteq [\mathbb{N}]^{<\omega}$. For every infinite $M\subseteq \N$,  there is an infinite $N\subseteq M$ such that one of the following holds:
    \begin{enumerate}[1), wide=0pt]\addtocounter{enumi}{-1}
        \item $\mc F\rest N = \emptyset$,
        \item For all $P\in [N]^\omega$ there is $s\in \mathcal{F}$ such that $s\sqsubset P$.
    \end{enumerate}
\end{teo}

For a family $\mc F\subseteq [\N]^{<\omega}$ let $c_{\mc F}: [\N]^\omega\to 2$ be the corresponding Galvin coloring, that is, given by $c_{\mc F}(M) = 1$ exactly when $M$ has an initial segment in $\mc F$. Thus every open coloring is of the form $c_{\mc F}$ for some family $\mc F\subseteq\fin$, and $\Hom_i(c_{\mc F})$, $i = 0,1$, is the set of all $N\in [\N]^\omega$ satisfying {\it i)} in Galvin's Theorem. It is easy to see that, if $\mc F$ is a compact family (as a subset of $2^\N$), each $\Hom_i(c_{\mc F})$ is closed in $[\N]^\omega$. This gives a characterization of tall $F_{\sigma}$ ideals.

\begin{prop}\cite{GrebikUzca2018}\label{rep-of-closed} Every tall $F_\sigma$ ideal is an open coloring ideal.  
\end{prop}
\begin{proof}
    Recall first that for any $F_\sigma$ ideal $\ideal$ (containing $\fin$) there is a closed hereditary set $\mc K$ such that $\ideal=\langle \mc K\rangle$. In fact, we assume that $\ideal=\bigcup_n \mc F_n$ where each $\mc F_n$ is closed and hereditary and $\mc F_n\subseteq \mc F_{n+1}$ for all $n$. Now consider 
    \[
    \mc K = \bigcup_{n\in\N} \{A\setminus n:\;A\in \mc F_n\}.
    \]
    Then $\mc K$ is closed and hereditary. Clearly $\mc K\subseteq \ideal$ and given $A\in \mc F_n$, 
    since $A = (A\setminus n)\cup (A\cap n)$ and $\langle \mc K\rangle$ contains all $\{n\}$, $A\in \langle \mc K\rangle$.

    Define $\mathcal{F}_{\mc K} = \fin\setminus \mc K = \{s\in [\mathbb{N}]^{<\omega} : s\not\in \mc K\}$. We claim that $\Hom_1(c_{\mc F_{\mc K}}) = \emptyset$.
    Indeed assume for a contradiction that $Y\in [\N]^\omega$ satisfies the second condition in the conclusion of Galvin's theorem. Since $\mc K$ is tall there is an infinite $Z\subseteq Y$ such that $Z\in \mc K$. As $Y$ satisfies the second condition with respect to $\mc F_{\mc K}$, there is $s\in \mc F_{\mc K}$ such that $s\sqsubset Z$, but since $\mc K$ is hereditary we have $s\in \mc K$ and this contradicts the definition of $\mc F_{\mc K}$. 
    
    It remains to check that $\mc K\cap [\N]^\omega = \Hom_0(c_{\mc F_{\mc K}})$.
    Clearly $\subseteq$ holds.
    Conversely, let $X\not\in \mc K$ be infinite. Since $\mc K$ is hereditary and closed there must be some $n\in\mathbb{N}$ such that $X\cap n \not\in \mc K$. Thus $X\cap n\in \mathcal{F}_{\mc K}$ and we get $X\not\in\Hom_0(c_{\mathcal{F}_{\mc K}})$.
\end{proof}

In the case of \emph{continuous} colorings \( c: [\N]^\omega \to r \), i.e., colorings where all colors are open, these are induced by colorings of \emph{fronts} on \( \N \), a notion introduced by Nash-Williams to characterize families of finite sets with the Ramsey property. Recall that a family \( \mc{F} \) of subsets of some infinite set \( M \subseteq \N \) is a front on $M$ if it is thin—no element of \( \mc{F} \) is a proper initial segment of another—and satisfies that every infinite subset of \( M \) has a necessarily unique initial segment that belongs to \( \mc{F} \).
Now, given such a coloring, define for each color \( i <r \), the family \( \mc{F}_i \) of all finite subsets \( s \subseteq \N \) such that \( \langle s \rangle \subseteq c^{-1}(i) \). Denote by \( \mc{F}_i^{\sqsubseteq-\min} \) the minimal elements of \( \mc{F}_i \) with respect to initial segments, and define 
\[
\mc{G} := \bigcup_{i<r}\mc{F}_i^{\sqsubseteq-\min},
\]
along with a coloring \( c^*: \mc{G} \to r \) defined by \( c^*(s) = i \) if \( s \in \mc{F}_i^{\sqsubseteq-\min} \). It is easily seen that \( \mc{G} \) is a front on \( \N \) and that \( c(M) = c^*(s_M) \), where \( s_M \) is the unique initial segment of \( M \) in \( \mc{G} \).

Conversely, every coloring \( \ccal: \mc{G} \to r \) of a front \( \mc{G} \) on \( \N \) induces a continuous coloring \( \ccal_* \) on \( [\N]^\omega \) in the same way. These are called \emph{front (or continuous) colorings}, and we define \( \hom_i(\ccal) \) as the set of all \( \ccal \)-homogeneous sets of color \( i \), i.e., those \( A \subseteq \N \) such that \( \ccal \) is constant with value \( i \) on 
\[
\mc{G} \restriction A := \mc{G} \cap [A]^{<\omega}.
\]
Those sets $A$ with $\mc G\restriction A = \emptyset$ are considered $\ccal$-homogeneous for all colors, so that $\hom_i(\ccal)$ is an hereditary and compact family. This is also logically consistent, since all elements of $\mc G\rest A$ have color $i$.

It is important to observe that, in contrast to \( \Hom_i(c) \), \( \hom_i(\ccal) \) may include both finite and infinite sets. It follows that each \( \hom_i(\ccal) \) is hereditary and a compact subset of \( 2^\N \). Therefore, \( \langle \hom(\ccal) \rangle \), the ideal generated by \(\bigcup_{i<r} \hom_i(\ccal) \), is \( F_\sigma \). Recall that $\langle\hom(\ccal)\rangle$ contains, by our definition of ideal, all finite subsets of $\N$.

\begin{defi}[c-coloring ideal]\label{lj4witejigerjrgf} 
An ideal   is a  {\em continuous coloring ideal (c-coloring ideal in short)}     when it is of the form  $\langle\Hom(c)\rangle$ for some  continuous coloring $c:[\N]^\omega\to r\in \N$ or, equivalently, of the form $\langle\hom(\ccal)\rangle$  for some $\ccal:\mc F\to r$ defined on a front $\mc F$.  
\end{defi} 


A coloring \( \ccal: \mathcal{F} \to r\) on some family $\mc F$ of subsets of a set $X$ induces \(r\) hypergraphs \( \mathcal{H}_0(\ccal),\dots,\mc H_{r-1}(\ccal) \) over \( X \), where the hyperedges of   \( \mathcal{H}_i(\ccal) \) are defined as the sets of elements in \( \mathcal{F} \) that are colored by \( \ccal \) with color \(i\). In particular, when the family is a subset of  \( [X]^d \), the associated hypergraphs \( \mathcal{H}_i \) are \(d\)-uniform. Furthermore, when \(\mc F= [\N]^d \) and $r=2$, the hypergraphs \( \mathcal{H}_0 \) and \( \mathcal{H}_1 \) are complements of one another.

Within this framework, the \(c\)-coloring ideal \( \langle \hom(\ccal) \rangle \) for a 2-coloring coincides with the ideal on the vertex set of the hypergraph generated by the cliques and anticliques of the hyperedges. 
More generally, for an arbitrary \(r\)-coloring, the corresponding \(c\)-coloring ideal is the ideal generated by the complete sub-hypergraphs of each hypergraph \( \mathcal{H}_i \).
This interpretation of \(c\)-coloring ideals will be useful for establishing connections between structural properties of these ideals and those of hypergraphs.

\subsection{Examples of pathological c-coloring ideals}

A positive aspect of c-coloring ideals is their effective tallness, as highlighted in \cite{GrebikUzca2018}. However, despite their utility, c-coloring ideals can exhibit pathological behavior. These ideals, rooted in the Mazur ideal, are intricately linked to Ramsey theory. Furthermore, their interplay with the concentration of measure phenomenon introduces an additional layer of complexity. As a significant consequence of these interactions, we establish that multidimensional random ideals inherently exhibit pathological behavior.

\begin{defi}[Super-pathology] 
We say that an ideal is  {\em super-pathological} if every non-trivial $F_\sig$ ideal that covers it is pathological. 
\end{defi} 
Mazur's ideal $\mathcal{M}$ \cite{Mazur91} was the first example of an $F_\sigma$-ideal not extended by a summable one. Later, J. Mart\'inez, D. Meza-Alcántara and C. Uzcátegui proved in \cite{MMU2022} that in fact $\mc M$ is the first example of a super-pathological $F_\sig$-ideal. The proof of this  uses a simple but fundamental property of measures in relation to coverings, explicitly defined by the {\em Kelley's covering number}. This number $\de(X,\mc S)$ is defined for a covering $\mc S\con \mc P(X)$ of a finite set $X$ as 
$$\de(X,\mc S):=\frac{1}{\#\mc S}\min_{x\in X}\#\conj{S\in \mc S}{x\in S}.$$
It is proved in \cite[Lemma 3.7]{MMU2022} (see also \cite[Corollary 6]{Kelley59}) that if $\mu$ is a measure on a finite set $X$ then every covering $\mc S$ of $X$  has some element $S\in \mc S$ such that $\mu(S)\ge \de(X,\mc S)\cdot \mu(X)$. In contrast, there are submeasures that do not satisfy this. The failure of Mazur's submeasure to follow this principle is what makes $\mc M$ pathological.   We shall refine Mazur's example to construct new super-pathological ideals, among them c-coloring ones.

\begin{defi}[Cardinal intervals of sets]\label{def:card_intervals}
Let $X$ be a finite set. For $0\le\al\le \be \le 1$ we define the {\em cardinal interval} $X^{[\al,\be]}:=\conj{A\con X}{ \al \# X \le \# A\le \be \#X}$. A   {\em symmetric cardinal interval} is a cardinal interval of the form $X^{[(1-\de)/p, (1+\de)/p]}$ for some $p\in \N$ and some $0\le \de\le 1$.  
\end{defi}

Roughly speaking, the symmetric cardinal interval $X^{[(1-\de)/p, (1+\de)/p]}$ is the $\de/p$-fattening of $[X]^{n/p}$ by the  {\em normalized Hamming distance} $d_X$. Recall that this metric counts the (normalized) cardinality of the symmetric difference of two subsets $A,B$ of $X$, that is, $d_X(A,B):= \# (A\triangle B)/\#X$.

For each \( x\in X \), define
\[
\widehat{x} :=\widehat{x}^X:= \{A \in  X^{[\al,\be]}: x \notin A\}.
\]
Observe that $\mc C_X:=\{\widehat{x}\}_{x\in X}$ is a covering of $X^{[\al,\be]}$ exactly when $\be<1$, which we assume from now on. Its {\em Kelley's covering number} $\delta(X^{[\al,\be]},\mc C_X) $ is
\begin{equation}\label{eq:Kelley_number_cov}
    \delta(X^{[\al,\be]},\mc C_X) = \frac{\min_{A\in X^{[\al,\be]}}\#\conj{x\in X}{A\in \widehat{x}} }{\lvert \mc C_X \rvert} = \frac{\min_{A\in X^{[\al,\be]}} \#(X\setminus A)}{\#X} \ge 1-\be.
\end{equation}
Define the covering submeasure for $\mc A\subseteq\mc P(X)$ as
\begin{equation}\label{eq:covering_subm}
    \psi_X(\mc A) := \min\left\{ \# F : F \subseteq X \text{ and }   \mc A \subseteq \bigcup_{x \in F} \widehat{x} \right\}.
\end{equation}
For example, if $\mc A$ covers $X$, then $\psi_X(\mc A)>1$. More generally, it is easy to see that 
\[
\psi_X(\mc A)= \min\conj{r\in \N}{\mc A \text{ is not an $r$-covering}},
\]
where $\mathcal{A}$ is an $r$-covering if for every $s \in [X]^r$ there exists some $A \in \mathcal{A}$ such that $s \subseteq A$.

Next, we freely amalgamate all these submeasures. Suppose now that $(X_n)_n$ is an unbounded sequence of disjoint finite sets,  $0\le \al\le \be < 1$, and let $\mk X:= \bigcup_{n\in \N} \mk X_n$, where each $\mk X_n:=X_n^{[\al,\be]}$. On $\mk X$ we define the submeasure
\[
\psi(\mc A):=\sup_{n\in \N} \psi_{X_n}(\mc A\cap \mk X_n).
\]
In this setting we can prove the super-pathology of $\fin(\psi)$ using the following result:

\begin{teo}[{\cite[Theorem 3.8]{MMU2022}}]\label{thm:super-path_ideals}
    Let $\psi$ be a lscsm on a countable set $\mk X$ such that there is $M > 0$ and a partition $\{\mk X_n : n\in\N\}$ of $\mk X$ in finite sets satisfying:
    \begin{itemize}[wide=0pt]
        \item The family $\mc C := \{\mc A\subseteq\mk X : \psi(\mc A)\leq M\}$ covers $\mk X$.
        \item $\sup_n \psi(\mk X_n) = \infty$.
        \item $\psi(\mc A) = \sup_n \psi(\mc A\cap\mk X_n)$ for all $\mc A\subseteq\mk X$.
    \end{itemize}
    Let $\mc C_n$ be a subfamily of $\mc P(\mk X_n)\cap \mc C$ such that $\mc C_n$ covers $\mk X_n$, $\de_n := \de(\mk X_n, \mc C_n)$ and $\de := \inf_n \de_n$. If $\de > 0$ then $\fin(\psi)$ is an $F_{\sigma}$ super-pathological ideal.
\end{teo}

In our case we can take $M = 1$ and $\mc C_n := \mc C_{X_n} = \{\widehat{x} : x\in X_n\}$. Clearly $\psi_{X_n}(\widehat{x}) = 1$ for all $x$ and $\delta(\mk X_n, \mc C_n) \geq 1 - \be > 0$ by \eqref{eq:Kelley_number_cov}. Moreover, $\mk X_n$ is not an $r$-covering of $X_n$ if and only if there exists $s\in [X]^r$ such that no $A\in\mk X_n$ contains it, if and only if $r > (1-\be)\#X_n$. Thus $\psi_{X_n}(\mk X_n) > (1 - \be)\#X_n$. 
We hence conclude that the ideal $\fin(\psi)$ is super-pathological, that is, it is not covered by a non-trivial non-pathological $F_\sig$ ideal. 

We can now define associated colorings.

\begin{defi}[Mazur colorings]\label{ojrgijriogojr}
For $d\in\N$ we define a coloring $\ccal=\ccal_{\al,\be,d}: [\mk X]^d\to 2$ by, given $\mc A\in [\mk X]^d$,
\[
\ccal(\mc A):=\begin{cases}
1 & \text{ if $\psi(\mc A)=1$,}\\
0 & \text{ otherwise.}
\end{cases}
\]
For each $n$, let $\ccal_n = \ccal_{\alpha,\beta,\de}^{(n)} : [\mathfrak{X}_n]^d \to 2$ denote the restriction of $\ccal_{\al,\be,\de}$ to $[\mathfrak{X}_n]^d$.

\end{defi}
That is, $\ccal(\mc A)=1$ exactly when $\mc A\cap \mk X_n$ does not cover $X_n$ for every $n\in \N$. We now establish that zero-homogeneous subsets are not only finite but also bounded in cardinality. Furthermore, we show that Mazur coloring ideals contain the corresponding pathological Mazur-like ideals.

\begin{lema}
    \label{asdijnbosirubhog}
    For every $d\in\N$ and $0\leq\alpha\leq\beta < 1$ there exists $k = k(d,\alpha,\beta)\in\N$ such that $\hom_0(\ccal)\subseteq [\mk X]^{\leq k}$. Moreover, $\fin(\psi)\con \langle \hom(\ccal)\rangle$.  
\end{lema}
\begin{proof}
    Let $\mc A\subseteq\mk X$ be a $\ccal$-homogeneous set of color $0$. This set cannot intersect $d$ or more $\mk X_n$'s, because then taking one element in each we would have a set in $[\mc A]^d$ of color $1$. We now claim that there exists $m\in\N$ such that all intersections $\mc A\cap\mk X_n$ have cardinality at most $m$, from which the result follows with $k = m(d-1)$.

    Our proof is based on the strategy employed by J.~Gillis in his classical result \cite{Gillis1936} in measure theory. On $X_n$ consider the normalized counting measure $\mu$, given by $\mu(A) := \#A/\#X_n$. Assume that there are $m$ different sets $A_1,\ldots,A_m\in\mc A\cap\mk X_n$ and let us find a bound for $m$. Setting $E_j := X_n\setminus A_j\in X_n^{[1-\be,1-\al]}$, by Jensen's inequality we have
    \begin{align}\label{eq:Gillis}
        m^d (1-\be)^d & \leq \left(\sum_{j=1}^m \mu(E_j)\right)^d = \left(\int_{X_n} \sum_{j=1}^m \mathbbm 1_{E_j}\,d\mu\right)^d \leq \int_{X_n} \left(\sum_{j=1}^m \mathbbm 1_{E_j}\right)^d\,d\mu \nonumber \\
        & = \sum_{k_1 + \ldots + k_m = d} \binom{d}{k_1,\ldots,k_m} \,\mu\!\left(\bigcap_{k_j\neq 0} E_j\right),
    \end{align}
    where $\binom{d}{k_1,\ldots,k_m} = \frac{d!}{k_1!\cdots k_m!}$ is the corresponding multinomial coefficient. Note that all $m$-tuples $(k_1,\ldots,k_m)$ considered in the sum of the right hand side have at most $d$ non-zero coordinates.
    
    For each $1\leq l\leq d$ define $I(l,d)$ as the set of all $l$-tuples $(k_1,\ldots,k_l)$ of strictly positive integers adding up to $d$. Then, given $s\subseteq\{1,\ldots,m\}$ of cardinality $\#s = l\leq d$, the term $\mu(\bigcap_{j\in s} E_j)$ appears in the right hand side of \eqref{eq:Gillis} with coefficient
    \[
    a_l := \sum_{(k_1,\ldots,k_l)\in I(l,d)} \binom{d}{k_1,\ldots,k_l}.
    \]
    Since this coefficient depends only on the cardinality of $s$, we can rewrite
    \[
    m^d (1 - \be)^d \leq \sum_{l=1}^d a_l \sum_{s\in [\{1,\ldots,m\}]^l} \mu\!\left(\bigcap_{j\in s} E_j\right).
    \]
    However, if $\#s = d$ then $\bigcap_{k_j\neq 0} E_j = X_n\setminus\bigcup_{k_j\neq 0} A_j = 0$ because $\mc A$ is $0$-homogeneous for $\ccal$. For the remaining $s$ we use that $\mu(\bigcap_{j\in s} E_j) \leq 1 - \al$, which leads to
    \begin{equation}\label{eq:Gillis2}
        m^d (1 - \be)^d \leq \sum_{l=1}^{d-1} a_l \sum_{s\in [\{1,\ldots,m\}]^l} (1 - \al) = (1 - \al) \sum_{l=1}^{d-1} a_l \binom{m}{l}.
    \end{equation}
    Since $\binom{m}{l}$ is a polynomial on $m$ of degree $l$, the right hand side has degree $d-1$ whereas the left one has degree $d$. Thus there must exist $m_0 = m_0(d,\alpha,\beta)$ such that the inequality does not hold for $m > m_0$, and hence $\# (\mc A\cap\mk X_n)\leq m_0$.

    The inclusion $\fin(\psi)\con \langle \hom(\ccal)\rangle$ readily follows from the first part and the fact that for each $(x_n)_n\in \prod_{n\in \N}X_n$, the set $\bigsqcup_{n\in \N} \widehat{x_n}$ is 1-homogeneous for $\ccal$.
\end{proof}

\begin{rem}\label{rem:0-hom<=22}
    For example, when $\al = (1-\de)/2$ and $\be = (1+\de)/2$ with $\de > 0$ small enough, for $d = 3$ it can be shown that the corresponding $k$ can be taken $\le 22$. Indeed, for this $d$ the coefficients $a_l$ defined in the proof above are $a_1 = 1$ and $a_2 = 6$, so \eqref{eq:Gillis2} reduces to $m^2 - 12bm + 8b \leq 0$ with $b := (1+\de)/(1-\de)^3$. As $\de$ approaches $0$, the largest root of this polynomial approaches $(12 + \sqrt{112})/2\approx 11.29\ldots$ Thus, for $\de$ small enough, we must have $m\leq 11$, hence $k \leq m(d-1)\leq 22$.
\end{rem}

The next question to address is when $\langle \hom(\ccal)\rangle$ is non-trivial. It turns out that this has an interesting relation with graph theory. Recall that a hypergraph is a pair $\mc H = (V, E)$ where $V$ is the set of \emph{vertices} and $E\subseteq [V]^{<\omega}$ is the set of \emph{hyperedges}. The hypergraph is called \emph{$d$-uniform} if $E\subseteq [V]^d$, and \emph{uniform} if it is $d$-uniform for some $d\geq 1$. Given a $d$-uniform hypergraph $\mc H$, a subset $A\subseteq V$ is called \emph{complete} if $[A]^d\subseteq E$, and \emph{independent} if $[A]^d\cap E = \emptyset$. 
The \emph{chromatic number} of $\mc H$, denoted $\chi(\mc H)$, is the smallest cardinal $\kappa$ such that there exists a vertex coloring $\ccal: V\to\kappa$ satisfying that no hyperedge is monochromatic. If $\mc H$ is uniform, this is equivalent to saying that each color is an independent set.

We will focus here on hypergraphs induced by colorings of $[\N]^d$: for each $\ccal: [\N]^d\to r$ and each $i < r$ we denote by $\mc H_i(\ccal)$ the $d$-uniform hypergraph on $\N$ whose hyperedges are precisely $\ccal^{-1}(i)$. Note that if $r = 2$ then $\mc H_0(\ccal) = \overline{\mc H_1(\ccal)}$, where $\overline{\mc H}$ denotes the \emph{complement} hypergraph of $\mc H$. Moreover, in this case, a given set $A\subseteq\N$ 
is complete or independent for $\mc H_0(\ccal)$ if and only it is $0$ or $1$-homogeneous for $\ccal$, respectively. Note that subsets of $\N$ of cardinality $<d$ are, by definition, both complete and independent for every $d$-uniform hypergraph, just as they are both $0$ and $1$-homogeneous for every $2$-coloring of $[\N]^d$.


We start with the following straightforward observation, which will later be used to demonstrate that certain Mazur coloring ideals are non-trivial and, therefore, pathological.



\begin{prop}\label{ij6ioj45ytghf}
    Let $\ccal: [\N]^d \to 2$, and let $\mc H := \mathcal{H}_0(\ccal)$ be the $d$-uniform hypergraph on $\N$ whose hyperedges are $\ccal^{-1}(0)$. Then each of the following conditions implies the next:
    \begin{enumerate}[a), wide=0pt]
        \item $\chi(\mathcal{H}) = \omega$ and $\mc H$ does not contain infinite complete sets.
        \item The ideal $\langle \hom(\ccal) \rangle$ is non-trivial.
        \item $\chi(\mc H) = \chi(\overline{\mc H}) = \omega$.
    \end{enumerate}
\end{prop}
\begin{proof}
    Suppose \textit{b)} fails. Then $\N$ can be partitioned into finitely many $0$-homogeneous and $1$-homogeneous subsets with respect to the coloring $\ccal$. Assuming further that $\mc H$ does not contain infinite complete sets, all $0$-homogeneous sets must be finite. Rewritting them as a union of singletons, we obtain a finite partition of $\N$ consisting Writing each of them as a finite union of singletons, the remaining partition of $\N$ consists of $1$-homogeneous sets, i.e. independent for $\mc H$. Thus $\chi(\mc H) < \omega$, contradicting condition \textit{a)}.
    
    Now suppose \textit{c)} fails, for instance, if $\chi(\overline{\mathcal{H}}) = r < \omega$. Let $d: \N\to r$ be a coloring witnessing it. Each color $d^{-1}(i)$ is an independent set for $\overline{\mc H}$, and therefore complete for $\mc H$. This means that $\N$ is partitioned in finitely many $0$-homogeneous sets for $\ccal$, thus $\langle\hom(\ccal)\rangle$ is trivial, contradicting \textit{b)}.
\end{proof}

Since Lemma \ref{asdijnbosirubhog} establishes that, for any Mazur coloring \( \ccal \), all \(0\)-homogeneous subsets are uniformly bounded in cardinality, it follows that proving that the corresponding \(c\)-coloring ideal is non-trivial reduces—by Proposition \ref{ij6ioj45ytghf}—to showing that the hypergraph \(\mc H_0(\ccal)\) has infinite chromatic number.

To achieve this, we examine the finite subhypergraphs \(  \mc H_0(\ccal_n)\) induced by the restriction \(\ccal_n\) of the Mazur coloring \(\ccal\) to \(X_n^{[\al,\be]}\). Specifically, the vertex set is \(X_n^{[\al,\be]}\) and the hyperedges are the non-covering families of \(X_n^{[\al,\be]}\) of size \(d\).

\begin{prop}\label{ij6ioj45ytghf11}
    \[
    \sup_n \chi(\mc H_0(\ccal_n)) \leq \chi(\mc H_0(\ccal)) \leq \sup_n \chi(\mc H_0(\ccal_n)) \cdot d.
    \]
    In particular, the c-coloring ideal $\langle \hom(\ccal)\rangle$ is non-trivial exactly when $\sup_n \chi(\mc H_0(\ccal_n)) =\infty$.
\end{prop}
The final product is cardinal multiplication, which equals \(\omega\) when \(\sup_n\chi(\mc H_0(\ccal)) = \omega\), and \(\sup_n\chi(\mc H_0(\ccal)) \cdot d < \omega\) otherwise. 
\begin{proof}
    The only non-trivial inequality to verify is the second one, in the case where \(\sup_n \chi(\mathcal{H}_0(\ccal_n)) = r\) is finite. 
    For each \(n\), choose a partition 
    \[
    \mathcal H_0(\ccal_n) = \bigcup_{i < r_n} \mathcal{P}_i^{(n)}
    \] 
    consisting of independent subsets of \(\mathcal H_0(\ccal_n)\) and such that \(r_n \le r\). Now, refine  \((\mathcal{P}_i^{(n)})_{i < r_n}\) to   \((\mathcal{Q}_i^{(n)})_{i < rd}\) where each \(\mathcal{Q}_i^{(n)}\) satisfies one of the following properties:
    - it is empty, 
    - it is a singleton, or 
    - it is a subset of some unique \(\mathcal{P}_{j_i}^{(n)}\) such that \(\#\mathcal{P}_{j_i}^{(n)} \ge d\) and \(\#(\mathcal{P}_{j_i}^{(n)} \setminus \mathcal{Q}_i^{(n)}) = d-1\). 
    Define the sets 
    \[
    \mathcal{Q}_i := \bigcup_{n} \mathcal{Q}_i^{(n)}
    \] 
    for each \(i < rd\). 
    Since the families \((\mathcal{Q}_i^{(n)})_{i < rd}\) are pairwise disjoint and satisfy 
    \[
    \bigcup_{i < rd} \mathcal{Q}_i^{(n)} = \bigcup_{i < r_n} \mathcal{P}_i^{(n)} = X_n^{[\alpha, \beta]},
    \] 
    it follows that \(\{\mathcal{Q}_i\}_{i < rd}\) forms a partition of \(\mathfrak{X}\). 
    We claim that each \(\mathcal{Q}_i\) is an independent set in \(\mathcal H_0(\ccal)\). To see this, suppose there exists a hyperedge \(\mathcal{A}\) such that \(\mathcal{A} \subseteq \mathcal{Q}_i\). Then there exists some \(n\) such that \(\mathcal{A} \cap X_n^{[\alpha, \beta]}\) covers \(X_n\). By construction, \(\mathcal{A} \cap X_n^{[\alpha, \beta]} \subseteq \mathcal{Q}_i^{(n)}\), so \(\mathcal{Q}_i^{(n)}\) cannot be empty nor a singleton, and hence it must be contained in some \(\mathcal{P}_{j}^{(n)}\) with \(\#(\mathcal{P}_j^{(n)} \setminus \mathcal{Q}_i^{(n)}) = d-1\). Therefore, we can select a subset \(\mathcal{S} \subseteq \mathcal{P}_j^{(n)} \setminus \mathcal{Q}_i^{(n)}\) such that \((\mathcal{A} \cap X_n^{[\alpha, \beta]}) \cup \mathcal{S}\) has cardinality \(d\). Since this set also covers \(X_n\), it forms a hyperedge of $\mc H_0(\ccal_n)$ entirely contained in the independent set \(\mathcal{P}_j^{(n)}\), which is a contradiction.

    Thus, \(\mathcal{Q}_i\) must indeed be an independent set, completing the proof.
\end{proof}

We present the following simple example illustrating a trivial Mazur \(c\)-coloring ideal.

\begin{coro}
    The c-coloring ideal \( \langle \hom(\ccal_{\alpha, \beta, d}) \rangle \) is trivial for \(\alpha = \beta = {1}/{2}\), \(d = 2\), and each \(X_n = 2n\).
\end{coro}
\begin{proof}
    Write $\ccal := \ccal_{1/2,1/2,2}$, and for each $n$ let $\ccal_n$ be the restriction of $\ccal$ to $[2n]^n$. Choose a partition $\mathcal{P}_0^{(n)} \cup \mathcal{P}_1^{(n)}$ of $[2n]^n$ such that no set \(A\) and its complement \(2n \setminus A\) belong to the same partition class, for example by setting $\mc P_0^{(n)} := \widehat{0} = \{A\in [2n]^n : 0\notin A\}$. By construction, each \(\mathcal{P}_i^{(n)}\) is independent for $\mc H_0(\ccal_n)$, so $\chi(\mc H_0(\ccal_n)) = 2$. Consequently, $\chi(\mc H_0(\ccal))\leq 2\cdot\sup_n \chi(\mc H_0(\ccal_n))\leq 4$. By the implication \textit{b)} \(\Rightarrow\) \textit{c)} in Proposition \ref{ij6ioj45ytghf}, it follows that the corresponding c-coloring ideal is trivial.
\end{proof}

The case \(d = 2\) is, in a sense, special, as we will now demonstrate.
For each $\de>0$, $p\in \N^*$ and a finite set $X$, we write $\mc S_{\de,p}(X)$ to denote  the symmetric cardinal interval $X^{[(1-\de)/p, (1+\de)/p]}$, that is, 
\[
\mc S_{\de,p}(X):=\conjbig{Y\con X}{ \frac{\# X}p (1-\de)\le \# Y\le \frac{\#X}{p}(1+\de)}.
\]
We have the following.

\begin{teo}\label{oijt45ioj345itjretrtrtretert}
For every $p \ge 2$, $\delta > 0$, and $r \in \mathbb{N}$, there exists some $N = N(p,\delta,r) \in \mathbb{N}$ such that for every $n \ge N$, every $r$-coloring of $\mathcal{S}_{\delta,p}(n)$ contains a monochromatic set consisting of $p+1$ sets that covers $n$, or, in other words,  
$$\chi(\mathcal{H}_0(\ccal_n))\ge r+1.$$
\end{teo}

From Proposition \ref{ij6ioj45ytghf} and Proposition \ref{ij6ioj45ytghf11} we obtain the following.
\begin{coro}\label{4ijtiojrejgifkgf3}
 For every $p \ge 2$ and $0<\delta < 1$, the c-coloring ideal corresponding to the coloring $\ccal_{(1-\de)/p,(1+\de)/p,p+1}:[ \bigsqcup_n \mc S_{\de,p}(n)]^{p+1}\to 2 $ is super-pathological. Moreover the corresponding $(p+1)$-uniform hypergraphs $\mc M_{p,\de}:=\mc H_0(\ccal_{(1-\de)/p,(1+\de)/p,p+1})$ forbid some finite complete $d$-uniform hypergraph and generates a super-pathological $F_\sig$-ideal.       
\end{coro}

\begin{proof}
For simplicity, set \(\ccal := \ccal_{(1-\delta)/p,(1+\delta)/p,p+1}\).  
Since \(\delta < 1\), we have \((\delta + 1)/2 < 1\). Hence, it follows from Lemma \ref{asdijnbosirubhog} that $\fin(\psi)\con\langle \hom(\ccal)\rangle$ and that the complete sets of the $d$-uniform hypergraph \(\mc H_0(\ccal)\) are finite. On the other hand,  from Theorem \ref{oijt45ioj345itjretrtrtretert}, it follows that \(\sup_n \chi(\mathcal{H}_0(\ccal_n)) = \infty\). Consequently, by Proposition \ref{ij6ioj45ytghf11}, we conclude that the chromatic number of the   hypergraph \(\mathcal{H}_0(\ccal)\) is infinite as well. Finally, invoking   Proposition \ref{ij6ioj45ytghf}, we obtain that the corresponding \(c\)-coloring ideal for \(\ccal\) is non trivial, hence super-pathological.
\end{proof}

The proof of the Theorem \ref{oijt45ioj345itjretrtrtretert}   uses the concentration of measure of almost symmetric subfamilies of Hamming cubes, that we pass to explain. Fix $0<\de<1$, and given finite set $X$ and $p\in \N$, $p\ge 2$, let $\mr{Equi}_\de(X,p)$ the collection of {\em $\de$-equi-surjections}, that is, surjections $F: X\to p$ such that 
\[
\frac{\#X}{p}(1-\de) \le \#F^{-1}(q)\le \frac{\#X}{p}(1+\de) \text{ for every $q<p$}.
\]
 
On \(\mathrm{Equi}_\delta(X,p)\), we consider the normalized Hamming metric:
\[
d_{X,p}(F,G) := \frac{\#\{F \neq G\}}{\#X},
\]
along with its uniform probability measure
\[
\mu_{X,p}(\{F\}) := \frac{1}{\#\mathrm{Equi}_\delta(X,p)}
\]
for every singleton $\{F\}$.

It is proved in \cite[Corollary 5.34]{AmalgamationRamseyLPSpaces} that the sequence \((\mathrm{Equi}_\delta(n,p), d_{n,p}, \mu_{n,p})_{n \in \mathbb{N}}\) is a \emph{Lévy sequence}\footnote{In fact, it is an asymptotically normal Lévy sequence.}, exhibiting the concentration of measure phenomenon, which we now briefly recall.

An \emph{mm-space} is a triple \(\mathcal{X} := (X, d, \mu)\), where \(d\) is a metric on \(X\) and \(\mu\) is a Borel probability measure on \(X\). 
The \emph{extended concentration function} of \(\mathcal{X}\) is defined for \(\delta, \varepsilon > 0\) as:

\[
\alpha_{\mathcal{X}}(\delta, \varepsilon) := 1 - \inf \left\{ \mu(A_\varepsilon) : \mu(A) \ge \delta \right\},
\]
where \(A_\varepsilon\) denotes the open \(\varepsilon\)-neighborhood of \(A\).
A sequence \((X_n, d_n, \mu_n)_{n \in \mathbb{N}}\) of mm-spaces such that \(1 \leq \mathrm{diam}(X_n) < \infty\) for all \(n\) is called a \emph{Lévy sequence} when for every \(\delta, \varepsilon > 0\):
\begin{equation}\label{k43ijti4jtrre}
\lim_{n\to \infty}\alpha_{X_n}(\delta, \varepsilon\cdot\mr{diam}(X_n)) = 0.
\end{equation}

In particular, it satisfies the following property, which will play a crucial role in the proof of Theorem \ref{oijt45ioj345itjretrtrtretert}.

\begin{teo}\label{h4htu34htu4t4}
    For every $p\in \N$, $p\ge 2$, and $\eta,\vep>0$ there is some $N = N(p,\eta,\vep)\in \N$ such that for every $n\ge N$ and every $\mc S\con \mr{Equi}_\de(n,p)$ such that $\mu_{n,p}(\mc S)\ge \eta$ one has that $\mu_{n,p}(\mc S_\vep)\ge 1-\vep$, where $\mc S_\vep:=\conj{F\in \mr{Equi}_\de(n,p)}{d_{n,p}(F,\mc S)\le\vep}$ is the $\vep$-fattening of $\mc S$. \qed
\end{teo}

\begin{proof}[Proof of Theorem \ref{oijt45ioj345itjretrtrtretert}]
    Fix $\delta > 0$. The proof proceeds by induction on the number of colors \( r \). For the base case \( r = 1 \), it is clear that with $p$ many sets of cardinality $\lceil n/p\rceil$  one can cover $n$, and since the ratio $\lceil n/p\rceil /(n/p)$ tends to $1$ when $n$ tends to infinity, for large enough $n$ those sets of cardinality $\lceil n/p\rceil$ belong to $\mc S_{\de,p}(n)$. Suppose now that the result is true for $r - 1$ and let $\vep:= \de/(2p^2)$ and $n_0 := N(p,\delta/2,r-1)$. Use Theorem \ref{h4htu34htu4t4} to find $n_1> \max\{4 n_0 /3, 2p^2/\de\}$ such that if $n\ge n_1$ then  
    \begin{equation}\label{ioiouti578y78434}
        \text{for every $\mc S\subseteq \mr{Equi}_\de(n,p)$ such that $\mu_{n,p}(\mc S)\ge 1/r$ one has that $\mu_{n,p}(\mc S_\vep)>\frac{p-1}p$.}
    \end{equation}
    We claim that $n_1$ works. For suppose that  $\mc S_{\de,p}(n) = \mc P_0\cup \cdots \cup \mc P_{r-1}$ for $n\ge n_1$. We introduce the useful concept of {\em hole}. For $\mc P\con \mc S_{\de,p}(n)$ we say that $H\con n$ is a {\em hole} of $\mc P$ if no $A\in \mc P$ covers $H$, and we say that $\mc P$ has {\em $\gamma$-holes}, if $\mc P$ has a hole of size $\lceil \gamma \cdot n\rceil$. There are two cases to consider.
    
    \noindent{\sc Case 1.} There is some $\mc P_{q_0}$ with a $\de/(2p)$-hole. Let $H\con n$ be such hole.
    \begin{claim}
        $H\oplus \mc S_{\de/2, p}(n\setminus H)\con \mc S_{\de,p}(n)$, where $H\oplus \mc S_{\de/2, p}(n\setminus H) := \{H\cup A : A\in\mc S_{\de/2,p}(n\setminus H)\}$.
    \end{claim}
    Suppose that this fact is established. By the choice of $n_1$ we have that $\#(n \setminus H) = \lfloor n\cdot (1-\de/(2p))\rfloor \ge n_0$, because $1-\de/(2p)\ge 3/4$. We have the $(r-1)$-coloring of $\mc S_{\de/2, p}(n\setminus H)$ induced by the inclusion $H\oplus \mc S_{\de/2, p}(n\setminus H)\con \bigcup_{q\neq q_0} \mc P_q$, hence by the choice of $n_0$ there is some $q_1\neq q_0$ and a covering $\{A_0,\dots, A_{p+1}\}$ of $n\setminus H$ such that $\{H\cup A_j\}_{j=0}^{p-1}\con \mc P_{q_1}$, and we are done since $\{H\cup A_j\}_{j=0}^{p-1}$ covers now $n$.
    \begin{proof}[Proof of Claim.]
        Let $A\in \mc S_{\de/2, p}(n\setminus H)$. Then, 
        \begin{align*}
            \#(H\cup A) & = \#H + \#A \leq \#H + (n - \#H)\frac{1 + \frac{\delta}{2}}{p} \leq \frac{n}{p} \left(1 + \frac{\delta}{2}\right) + \left\lceil \frac{\delta\cdot n}{2p}\right\rceil\left(1 - \frac{1}{p}\right) \leq_{(\star)} \frac{n}{p} (1 + \delta) \\
            \#(H\cup A) & = \#H + \#A \geq \#H + (n - \#H)\frac{1 - \frac{\delta}{2}}{p} \geq \frac{n}{p}\left(1 - \frac{\delta}{2}\right) + \frac{\delta\cdot n}{2p} \left(1 - \frac{1}{p}\right) \geq \frac{n}{p} (1 - \delta)
        \end{align*}
        where in $(\star)$ we have used that $n\ge n_1\ge 2p^2/\de$.
     \end{proof}

\noindent{\sc Case 2.}  No $\mc P_q$ has $\de/(2p)$-holes. Let $\tau: \mr{Equi}_{\de}(n, p)\to \mc S_{\de,p}(n)$, $\tau(F):= F^{-1}(0)$. For each $ i< r$, let $\mc Q_i:=\tau^{-1}(\mc P_i)$, and choose $ i_0 < r$ such that $\mu_{n,p}(\mc Q_{i_0})\ge 1/r$. By the property in \eqref{ioiouti578y78434} that defines $n_1$ we know that $\mu_{n,p}((\mc Q_{i_0})_\vep) > (p-1)/p$. Observe that for each permutation $\pi$ of $p$, the mapping $F\in \mr{Equi}_\de(n,p)\mapsto \pi \circ F$ is a measure preserving bijection, in particular  each $\mu_{n,p}(\pi_q ((\mc Q_{i_0})_\vep)) > (p-1)/p $ where $\pi_q$ is the transposition on $p$ of $0$ by $q$.  Consequently,
\[
\bigcap_{q<p} \pi_q((\mc Q_{i_0})_\vep) \neq \buit.
\]
Let $F\in\bigcap_{q<p} \pi_q((\mc Q_{i_0})_\vep)$, i.e., $\pi_q\circ F\in (\mc Q_{i_0})_\vep$ for every $q<p$, since each transposition is idempotent. Observe that each $\tau(\pi_q\circ F)= F^{-1}(q)$. For each $q<p$, let $G_q\in \mc Q_{i_0}$ be such that $d_{n,p}(G_q,\pi_q\circ F)\le \vep $. Then
\begin{align*}
\#\bigcup_{q<p} G_q^{-1}(0) & \ge \#\bigcup_{q<p} (G_q^{-1}(0)\cap F^{-1}(q)) =  \#\left(\bigcup_{q<p}F^{-1}(q)\right) - \#\left(\bigcup_{q<p} (F^{-1}(q)\setminus G_q^{-1}(0))\right) \\
& \ge n - \vep  n p = n\left( 1- \frac{\de}{2p}\right). 
\end{align*}
Hence $\{G_q^{-1}(0)\}_{q<p}\con \mc P_{i_0}$ covers a set of cardinality $\ge n(1-\de/(2p))$, and since $\mc P_{i_0}$ has no $\de/(2p)$-holes there must be $A\in \mc P_{i_0}$ covering $n \setminus \bigcup_{q<p}G^{-1}(0)$, and we are done. 
\end{proof}

There is the following interesting corollary for the frontier values.  
\begin{coro}
 \label{u84ut985tgretgrge}
	For every $p,r\geq 1$ there exists a multiple  $\widetilde{n} = \widetilde{n}(r,p)$  of $p$   such that any $r$-coloring of $[\widetilde{n}]^{\widetilde{n}/p}$     contains a monochromatic set of cardinality $p+r$ that covers $\widetilde{n}$.   
\end{coro}

\begin{proof}

First observe that if the result is true for $\widetilde{n}$, it is also true for any multiple of it. To see this,  any $c:[k\widetilde{n}]^{k\widetilde{n}/p}\to r $ induces a coloring $\bar{c}:[\widetilde{n}]^{\widetilde{n}/p}\to r$, by simply declaring $\bar{c}(A):= c(\bigcup_{j<k} (j\widetilde{n} + A))$. Then if $\{A_l\}_{l<p+r}$ is a $\bar c$-monochromatic covering of $\widetilde{n}$, it follows that $\{\bigcup_{j<k} (j\widetilde{n}+A_l)\}_{l<p+r}$ is a $c$-monochromatic covering of $k \widetilde{n}$.

The proof is by induction on $r$. For $r=1$ the result is trivial. Suppose the result is true for $r\geq 1$, set $\de := 1/(2p^3)$ and let $n$ be such that 
\begin{enumerate}[i),wide=0pt]
    \item $n\ge N(p,\de,r+1)$ of Theorem \ref{oijt45ioj345itjretrtrtretert}, and
    \item  $(p+2)|n$ and  $\widetilde{n}(r,p+1) | (n - n(p+1)\de)$.
\end{enumerate}
We claim that $n$ works. Fix a coloring $c: [n]^{n/p}\to r+1$. Suppose that there is some color $c^{-1}(r_0)$ with a hole $H$ of size $n(p+1)\de = n/(p+2)$, where holes are defined in the proof of Theorem \ref{oijt45ioj345itjretrtrtretert}.
\begin{claim}
$H\oplus [n\setminus H]^{\#(n\setminus H)/(p+1)}\con [n]^{n/p}$.    
\end{claim}
Suppose this statement is established. The coloring $c$ induces another coloring $A\in [n\setminus H]^{\#(n\setminus H)/(p+1)}\mapsto c(H\cup A)\in (r+1)\setminus \{r_0\}$. Since $\#(n \setminus H) = n - n(p+1)\de$ is a multiple of $\widetilde{n}(r,p+1)$, by inductive hypothesis there is $r_1 < r+1$, $r_1\neq r_0$, and $\{A_j\}_{j<(p+1)+r}$ of color $r_1$ that covers $n\setminus H$, and consequently $\{H\cup A_j\}_{j<p+r+1}$ covers $n$ and it is $c$-monochromatic. 
\begin{proof}[Proof of Claim.]
Suppose that $A\con n\setminus H$ satisfies that $\#A= (n - n(p+1)\de)/(p+1)$. Then
\[
\# (H\cup A) = n(p+1)\de + \frac{n(1 - (p+1)\de)}{p+1} = \frac{n(p+1)^2\de + n - n(p+1)\de}{p+1} = \frac{n}{p}.
\]
\end{proof}
Suppose now that no color has a hole of cardinality $n(p+1)\de$. Define $\widehat{c}: \mc S_{\de,p}(n)\to r+1$ by declaring that $\widehat{c}(A) = \min\conj{j<r+1}{A\in (c^{-1}(j))_\de}$.  By the Theorem \ref{oijt45ioj345itjretrtrtretert} there is some $r_0<r+1$ and $\{A_j\}_{j<p+1}\con \widehat{c}^{-1}(r_0)$ covering $n$. For each $A_j$ choose $B_j$ such that $\#(A_j\triangle B_j)\le \de n$, and such that $c(B_j) = r_0$. This implies that
\[
\#\left(n \setminus \bigcup_{j<p+1} B_j\right) \le \bigcup_{j<p+1} \# (A_j\setminus B_j)\le n(p+1)\de.
\]
Choose $B$ with $c$-color $r_0$ covering $n \setminus \bigcup_{j<p+1}B_j$. Then $\{B_j\}_{j<p+1}\cup \{B\}$ has $c$-color $r_0$ and covers $n$. 
\end{proof}

We finish this part of ``pathological'' colorings with a new one, now of the Schreier family. 
Define \( \mk{X} := \bigcup_{n \geq 1} [2n]^n \). Order each \( [2n]^n \) using its lexicographic ordering \( <_\mathrm{lex} \), and extend this to \( \mk X \) by declaring \( A \prec B \) if \( \#A < \#B \) or \( (\#A = \#B \text{ and } A <_\mathrm{lex} B) \). Clearly, the corresponding order type is \( \omega \). 
Using this ordering, define a \( \omega \)-uniform family \( \mathcal{B} \) on \( (\mk X, \prec) \) by declaring that
\[
s \in \mathcal{B}  \text{ exactly when }  \#s = \min_{A \in s} \#A.
\]
For each \( A \in \mk X \), note that the associated family is
\[
\mathcal{B}_{\{A\}} = \left\{ s \subseteq (A, \infty)_\prec : \{A\} \cup s \in \mathcal{B} \right\} = [(A, \infty)_\prec]^{\#A},
\]
where \( (A, \infty)_\prec \) denotes the open interval of elements of \( \mk X \) strictly greater than \( A \) under \( \prec \).

Now define $\psi: \mc P(\mk X)\to [0,\infty[$ as $\psi(s) := \sup_{n\geq 1} \psi_n(s\cap [2n]^n)$, where $\psi_n$ is the covering submeasure on $[2n]^n$ as defined in \eqref{eq:covering_subm}. This is just the submeasure defined in the paragraph before \ref{thm:super-path_ideals} for $X_n = 2n$ and $\alpha = \beta = 1/2$, thus the ideal $\fin(\psi)$ is super-pathological. Associated to this submeasure we have the coloring \( \ccal : \mathcal{B} \to \{0, 1\} \) defined by
\[
\ccal(s) := 
\begin{cases} 
1 & \text{if $\psi(s) = 1$,} \\
0 & \text{otherwise.}
\end{cases}
\]
Equivalently, $\ccal(s) = 1$ if and only if $s\cap [2n]^n$ does not cover $2n$ for every $n$.

Let us show that $\langle\hom(\ccal)\rangle$ is non-trivial and super-pathological by proving that \( \fin(\psi) \subseteq \langle \hom(\ccal) \rangle \subsetneq \mathcal{P}(\mk X) \). The first inclusion follows, as before, from the fact that for any sequence \( (i_n)_n \in \prod_n 2n\) the set \( \bigcup_n \widehat{i_n}^n \) is \( 1 \)-homogeneous. Next, we see that  $\langle \hom(\ccal) \rangle$ is non-trivial. The proof  here uses arguments similar to that of the last inequality in Proposition  \ref{ij6ioj45ytghf11}. Suppose for a contradiction that $\mk X$ is a finite union of $\ccal$-homogeneous subsets. Since $0$-homogeneous sets are finite, that means that for some $n_0$ one has that $\bigcup_{n\ge n_0} [2n]^n = \mc P_0\sqcup \cdots \sqcup \mc P_{r-1}$, where each $\mc P_i$ is $1$-homogeneous. We see that this impossible by induction on $r$. If $r=1$, it is obvious that the tails $\bigcup_{n \ge \bar n}[2n]^n$ are never $1$-homogeneous. Now suppose that $r>1$. We apply the Corollary \ref{u84ut985tgretgrge} to $p=2$, and $r$ to obtain the corresponding $\widetilde{n}$. Let $n\ge \max\{n_0/2, \widetilde{n}/2 ,2+r\}$, and let $\ccal$ be the coloring of $[2n]^n$ induced by the partition $\mc P_0,\dots, \mc P_{r-1}$. Then there are $r_0 < r$, and $A_0,\dots, A_{r+1}\in \mc P_{r_0}$ such that $\bigcup_{j<r+2} A_j = 2n$. Since $\mc P_i$ is $1$-homogeneous, we cannot find $s\in \mc B$ containing $\{A_j\}_{j<r+2}$ and included in $\mc P_i$, so $\mc P_i\cap (\bigcup_{m>\bar n} [2m]^m) = \buit$ for some $\bar n > n$, and we have that $\bigcup_{m>\bar n} [2m]^m$ is the union of $r-1$ many 1-homogeneous sets, that by inductive hypothesis is impossible.

Next, we deduce as a consequence that several random hypergraphs generate pathological $F_\sigma$-ideals. 
Recall that for \(d \geq 2\), the \emph{random \(d\)-uniform hypergraph} \(R^{(d)}\) is the unique countably infinite \(d\)-uniform hypergraph that is both universal for countable \(d\)-uniform hypergraphs and \emph{homogeneous}. The latter property means that every hypergraph embedding of a finite sub-hypergraph of \(R^{(d)}\) into itself can be extended to a hypergraph automorphism of \(R^{(d)}\).
The existence of \(R^{(d)}\) follows from the fact that the class \(\mathcal{F}^d\) of finite \(d\)-uniform hypergraphs has the amalgamation property. By the \emph{Fraïssé correspondence}, this guarantees the construction of \(R^{(d)}\) as the generic limit of the class \(\mathcal{F}^d\), known as the \emph{Fraïssé limit} of \(\mathcal{F}^d\).

Similarly, the class \(\mathcal{F}^d_{n\text{-Free}}\) of finite \(d\)-uniform hypergraphs that forbid the complete \(d\)-uniform hypergraph \(\mathcal{K}^{(d)}_n\) of size \(n\) also satisfies the amalgamation property. Consequently, there exists a Fraïssé limit of \(\mathcal{F}^d_{n\text{-Free}}\), denoted by \(R^d_{n\text{-free}}\). This hypergraph is uniquely determined by its universality for countable \(d\)-uniform hypergraphs forbidding \(\mathcal{K}^{(d)}_n\) and its homogeneity.

\begin{coro}
For every \(d \geq 3\), there exists an integer \(n(d)\) such that the ideal generated by the hypergraph \(R_{n(d)\text{-free}}^d\) is a super-pathological \(F_\sigma\)-ideal, and, consequently, the same holds for the ideal generated by \({R}^d\).
\end{coro}
We have observed before that, for example, for $d=3$ $n(3)= 22$ works (see Remark \ref{rem:0-hom<=22}).
\begin{proof}
By Corollary \ref{4ijtiojrejgifkgf3}, for every \(d \geq 3\) there exists an integer \(n(d)\) and a \(d\)-uniform hypergraph \(\mathcal{H}\) on \(\mathbb{N}\) that forbids the complete \(d\)-uniform hypergraph of size \(n(d)\) and generates a super-pathological ideal. We claim that this choice of \(n(d)\) satisfies the desired property.
By the universality of \(\mathcal{R}^d_{n(d)\text{-free}}\), there exists a hypergraph embedding \(\gamma: \mathcal{H} \to \mathcal{R}^d_{n(d)\text{-free}}\). Let \(\mathcal{I}\) denote the ideal generated by \(\mathcal{H}\) and \(\mathcal{R}\) the ideal generated by \(\mathcal{R}^d_{n(d)\text{-free}}\). Then:

\[
\gamma(\mathcal{I}) = \mathcal{R} \restriction \gamma(\mathbb{N}).
\]

Since \(\gamma(\mathbb{N})\) is \(\mathcal{R}\)-positive (by the equality above), it follows that the restriction \(\mathcal{R} \restriction \gamma(\mathbb{N})\) is super-pathological, which in turn implies that the entire ideal \(\mathcal{R}\) is super-pathological as well.
\end{proof}

The next is not covered by our techniques, and, to our knowledge, it remains open. 
\begin{question}
Is the random ideal $\mc R$  pathological?    
\end{question}

In practice, it seems challenging  to find pathological $F_\sigma$ ideals, so the following is a natural question. 

\begin{question}
What is the (Borel or projective) complexity of the collection of non-pathological $F_\sigma$-ideal? More precisely, what is the complexity of the following collection: 
\[
\{A\in K(2^\N): \langle A\rangle \; \mbox{is non-pathological}\}
\]
where $\langle A\rangle$ denotes the ideal generated by $A$ (recall, that every $F_\sigma$ ideal can be generated by a closed collection of subsets of $\N$).
\end{question}

\begin{question}
Determine when a c-coloring ideal is non-pathological. In particular,  determine the 
the (Borel or projective) complexity of the collection 
\[
\{c\in 2^{[\N]^2}: \langle 
hom(c)\rangle \; \mbox{is non-pathological}\}.
\]
\end{question}

We have seen in Proposition \ref{ij6ioj45ytghf} some relations between hypergraph properties and the non-triviality of the corresponding ideal, but this relation is far from being fully understood.

\begin{question}
Are there hypergraph-theoretical properties that determine when the ideal generated by a hypergraph (through its complete and independent sets) is non-pathological?    
\end{question}

The following provides a partial answer.

\subsection{Examples of non-pathological c-coloring ideals}

We do not know the full answer to this question, but we have some partial ones. 
 We will study colorings $c: [\N]^2\to 2$ so that $c(\{m,n\}_{<})=i$ induces a partial ordering on $\N$. To analyze these partial orders we will use {\em Dilworth's theorem.}  
  Recall that the {\em width} 
$\mathrm{width}(\mc P)$ of a partial ordering $\mc P$ is the supremum of the cardinalities of antichains of $\mc P$.

\begin{teo}[Dilworth's theorem]
 If $\mathrm{width}(\mc P)$ is finite, then it coincides with the smallest decomposition of $\mc P$ as a union of chains. \qed
\end{teo}
The dual of Dilworth theorem, called {\em Mirsky's theorem}, is the following:  if $\mc P$ is a partial ordering such that its chains are finite and $\sup\conj{\#C}{C\con \mc P \text{ is a chain}}=c<\infty$, then $\mc P$ is the union of $c$-many antichains. 

\begin{defi}    
\label{def-comparability}
A coloring $\ccal:[\N]^2\to 2$ is called
\begin{enumerate}[a), wide=0pt]
    \item {\em $i$-asymmetric} when  $\hom_i(\ccal)\con \fin$;
    \item a {\em $i$-comparability coloring} when the relation $m\prec_\ccal^i n$   defined by $m<n$ and  $\ccal(\{m,n\})=i$ defines a partial ordering on $\N$; 
\item   a {\em  Sierpinski coloring}   when $\ccal$ is an $i$-comparability coloring for both $i=0,1$, that is, when there is a total ordering $\prec$ on $\N$ such that $\ccal (\{m,n\}_<)=1$ exactly when $m \prec n$. 
 \end{enumerate}
 \end{defi}

\begin{teo}
\label{iojtoijgoijot4ref}
Fix a coloring $\ccal: [\N]^2\to 2$.  
If $\ccal$ is an $i$-comparability then
\begin{align*}
\langle\hom_{1-i}(\ccal)\rangle & = \mc B(\mb p^{\hom_i(\ccal)}),\\ 
\langle\hom_i(\ccal)\rangle & = \mc B( \mb p^{\hom_{1-i}(\ccal)}),\text{ and}\\
\langle\hom(\ccal)\rangle & = \mc B(\mb p^{\hom_0(\ccal)})\sqcup\mc B(\mb p^{\hom_1(\ccal)}). 
\end{align*}
Consequently,

\begin{enumerate}[(i), wide=0pt]
       
\item if $\ccal$ is an $i$-comparability and $i$-asymmetric coloring, then
\[
\langle\hom(\ccal)\rangle = \langle\hom_{1-i}(\ccal)\rangle=\mc B( \mb p^{\hom_i(\ccal)}).
\]
Hence, every $i$-comparability and $i$-asymmetric coloring ideal is non-pathological. 
\item if $\ccal$ is an $i$-comparability $(1-i)$-asymmetric coloring, then 
\[
\langle\hom(\ccal)\rangle= \langle \hom_i(\ccal)\rangle=\mc B(\mb p^{\hom_{1-i}(\ccal)}).
        \]
Hence, every $i$-comparability $(1-i)$-asymmetric coloring ideal is non-pathological.
    \end{enumerate} 
\end{teo}

\begin{proof} 
Notice first that $\langle\hom(\ccal)\rangle  = \langle\hom_0(\ccal)\rangle \sqcup \langle \hom_1(\ccal)\rangle$.  It follows from \eqref{eq:simpler_expression_p^K} that  $\hom_{1-j}(\ccal)\subseteq\mc B(\mb p^{\hom_j(\ccal)})$ for $j = 0,1$. For the converse inclusions note that, for $m < n$, we have $m\not\prec_\ccal^i n$ if and only if $\ccal(\{m,n\}) = 1 - i$, and consequently the antichains of $\prec_\ccal^i$ are exactly the $(1-i)$-homogeneous sets for $\ccal$. Take now $M\in\mc B(\mb p^{\hom_i(\ccal)})$, and let
\[
L := \sup_{F\con M\text{ finite}} \norm{\sum_{n\in F} p_n^{\hom_i(\ccal)}} = \sup_{H\in\hom_i(\ccal)\rest M} \#H < \infty.
\]
This means that the cardinality of a maximal chain of $(M,\prec_\ccal^i)$ is $L$, thus by Mirsky's Theorem $M$ is a union $M = \bigcup_{i=1}^L A_i$ of $\prec_\ccal^i$-antichains $A_1,\ldots,A_L$. Each $A_j$ is therefore $(1-i)$-homogeneous for $\ccal$, so $M\in\langle\hom_{1-i}(\ccal)\rangle$. Analogously, given $M\in \mc B(\mb p^{\hom_{1-i}(\ccal)})$ the width of $(M,\prec_\ccal^i)$ is finite, so by Dilworth's Theorem $M$ is a finite union of $i$-homogeneous sets, hence it belongs to $\langle\hom_i(\ccal)\rangle$. 
    {\it (i)} and {\it (ii)} are trivially deduced.  
\end{proof}

\begin{ex}
$\mr{ED}_{\fin}$, the ideal generated by the transversals to $(s_n)_n$, $s_n=[2^{-1}-1,2^n-1[$, is a coloring ideal representable in $c_0$. To see this, define $k\lessdot l$ exactly when $k<l$ and they belong to different $s_n$'s. This is a partial ordering, whose chains are the transversals and its antichains consist of subsets of some $s_n$, hence finite. Theorem  \ref{iojtoijgoijot4ref} applies to the corresponding coloring $\ccal: \N^{[2]}\to 2$, $\ccal(\{k,l\}_<)=1$ if $k\lessdot l$, and consequently $\mr{ED}_{\fin}=\langle \hom(c)\rangle= \mc B(\mb p^{\hom_0(\ccal)})$ is non-pathological. Observe that the Cantor-Bendixson index of $\hom_0(\ccal)$ is 1, hence $C(\hom_0(\ccal))$ is isomorphic to $c_0$, and consequently $\mr{ED}_{\fin}$ is $\mc B$-representable in $c_0$.     
\end{ex}
We have seen that when $\ccal$ is a comparability color, then the corresponding ideal  $\langle \hom(\ccal)\rangle$ is the square union of two $\mc B$-ideals. We do not know the answer of the following.   

\begin{question}
\label{square-union-B}
Is the square sum of $\mc B$-ideals an $\mc B$-ideal?   
\end{question}

\subsection{\texorpdfstring{Universal $d$-colorings}{Universal d-colorings}}

\begin{defi}[$\Q$-colorings]
\label{Sierp-coloring}

 A {\em $\Q$-coloring} is a Sierpinski coloring $\widehat{\theta}:[\N]^2\to 2$ associated to some enumeration  $\theta: \N\to \Q$, that is,  $\widehat{\theta}(\{m,n\}_<):=1$ exactly when $\theta(m)<\theta(n)$. The Sierpinski ideal associated to $\theta$, $\langle\hom(\widehat{\theta})\rangle$, will be called a {\em $\Q$-coloring ideal}.
\end{defi}

In other words, $\langle\hom(\widehat{\theta})\rangle$ is the collection of finite unions of increasing and decreasing sequences with respect to $\theta$, and it follows from Theorem \ref{iojtoijgoijot4ref}  that $\langle\hom(\widehat{\theta})\rangle= \mc B(\mb p^{\hom_0(\widehat{\theta})})\sqcup \mc B(\mb p^{\hom_1(\widehat{\theta})})$ is the square-union of two non-pathological $F_\sig$-ideals. 

$\Q$-colorings have the following universal property.

\begin{teo}\label{j4itogjoidfg}\label{Sierpinski}
    Suppose that $\theta: \N\to \Q$ is an enumeration. Then for every coloring $\ccal:[\N]^2\to l$ there is some $M\con \N$ such that $\theta(M)$ is order-isomorphic to $\Q$ and such that $\hom(\widehat{\theta})\rest M\con \hom(\ccal)$.     
\end{teo}
The proof uses the following result on canonical colorings.

\begin{teo}[{Galvin; see \cite[Corollary 6.24]{Todorcevic2010}}]\label{GalvinQ2}  For every coloring  $\ccal:[\Q]^2\longrightarrow l$, there is $X\subseteq\Q$ order-isomorphic to $\Q$ such that $|\ccal [X]^2|\leq 2$. \qed
\end{teo}

\begin{proof}[Proof of Theorem \ref{j4itogjoidfg}] Fix a coloring  $\ccal:[\N]^2\longrightarrow l$ and an enumeration 
$\theta: \N\to \Q$, and define the coloring $d:[\Q]^2\to l\times 2$ by 
\[
d(\{\theta(m),\theta(n)\}):=(\ccal(\{m,n\}),\widehat{\theta}(\{m,n\})).
\]
By Galvin's Theorem \ref{GalvinQ2}, there is $X\subseteq\Q$ order-isomorphic to $\Q$, and $a,b$ two pairs in $l\times 2$ such that $d[X]^2\subseteq\{a,b\}$. Let $M:= \theta^{-1}(X)$. Notice that it is not possible to have $\{a,b\}=\{(i,k),(j,k)\}$ for $i<j<l$ and $k\in \{0,1\}$, since that would imply that $\widehat\theta$ is constant  on $[M]^2$, and thus that $\theta$ is monotone on $M$, contradicting that $\theta(M)$ is order-isomorphic to $\Q$. 
Thus, there are $i\leq j<l$ such that $\{a,b\} = \{(i,0),(j,1)\}$. Then, if $H\in\Hom(\widehat{\theta} \rest M)$ is of color $0$, then $d[\theta(H)]^2=\{(i,0)\}$ and $\ccal[H]^2=\{ i\}$. It follows that $H\in\Hom(\ccal)$. The argument is analogous if $\widehat{\theta}[H]^2$ is of color $1$. 
\end{proof}

We do not know if $\Q$-coloring ideals are non-pathological. However, as we will show below, they exhibit a local non-pathological behavior.

\begin{prop}
Every $\Q$-coloring ideal $\mc G$ is {\em locally  non-pathological}, that is, every  $\mc G$-positive set $A$ has a $\mc G$-positive subset $B$ such that $\mc G\rest B$ is a $\mc B$-ideal. 
\end{prop}

\begin{proof} Fix a bijection $\theta: \N\to \Q$, set $\mc G:= \langle\hom(\widehat{\theta})\rangle$
Suppose that $X\not\in\mc G$. We distinguish two cases: Suppose first that $\theta(X)$ has finitely many accumulation points $a_1<\dots<a_n$ in $\R\cup \{-\infty,\infty\}$. Choose disjoint intervals $I_1,\dots, I_n$ of $\R$ such that each $a_j\in I_j$ and such that $\theta(X)\con I_1\cup \cdots \cup I_n$. Let $1\le i\le n$ be such that  $Y:=X\cap \theta^{-1}(I_j)\notin \mc G$. Then either $Y\cap ]-\infty, a_j]\notin \mc G$ or $Y\cap [a_j,\infty[\notin \mc G$, so without loss of generality we may assume that $\theta(Y)\con ]-\infty, a_j]\notin \mc G$, and $a_j$ is its unique accumulation point. Hence the coloring $\widehat{\theta}$ on $Y$ is 0-asymmetric and since it is obviously a 1-comparability, it follows from Theorem \ref{iojtoijgoijot4ref} {\it b)} that $\mc G\rest Y$ is a $\mc B$-ideal. 

Suppose now that $\theta(X)$ has infinitely many accumulation points $\mc A$. Observe that if $Y\in \mc G$, then $\theta(Y)$ has only finitely many accumulation points.  Choose now consecutive intervals $(I_n)_n$ ($\max I_n<\min I_{n+1}$ for every $n$, or $\max I_{n+1}<\min I_n$ for every $n$) such that each $I_n\cap \mc A\neq \buit$, and now for each $n$  set $X_n:=  X\cap \theta^{-1}(I_n)$, and $Y:=\bigcup_n X_n$. Since $\theta(Y)$ has infinitely many accumulation points, it follows that $Y\notin \mc G$.  We start by the case when $\max I_n <\min I_{n+1}$ for every $n$.  
 Let
\[
K=\{s\in \hom_0(\widehat\theta\rest Y):\# (s\cap X_n)\leq 1 \;\;\mbox{for all $n\in \N$} \}.
\]
Then $K$ is an hereditary compact subset of  $\fin$. We claim that $\mc G\rest Y=\mc B(\mb p^{K})\rest Y$.  Recall that by \eqref{eq:simpler_expression_p^K}, 
$$
\text{$A\in\mc B(\mb p^K)$ if and only if $\max\conj{\#s}{s\in K\cap [A]^{<\infty}}<\infty$.}
$$
Suppose that  $H\in \hom_1(\widehat\theta\rest Y)$, then clearly  $K\cap [H]^{<\infty}\con [H]^{\le 1}$ as $K\subseteq\hom_0(\widehat\theta\rest Y)$, and consequently, $H\in \mc B(\mathbf p^K)$.  Suppose now that $H\in \hom_0(\widehat\theta\rest Y)$. Since $(I_n)_n$ is upwards consecutive, it follows that   $\{n\in \N\,:\, H\cap X_n\neq \emptyset\}$ is finite, thus, let $m$ be such that $H\cap X_n= \emptyset$ for all $n\geq  m$. Then
$K\cap [H]^{<\infty}\con [H]^{\le m}$, and thus $H\in \mc B(\mb p^K) $. Conversely, let $A\in \mc B(\mb p^K)\rest Y$.  Then $\limsup_{n} \#(A\cap X_n)<\infty$ (since otherwise we can find  an increasing $\{n_j\}_{j\in \N}$   such that $\#(A\cap X_{n_j})\ge j$, and produce then arbitrarily large subsets of $A$ in $K$). Let $m,r$ be such that $\# (A \cap X_n)\le r$ for every $n\ge m$.  Set $B:=A\setminus \bigcup_{n\le m}X_n$. Since 
each $X_i\in \hom_0(\widehat\theta)$, it suffices to show that $B\in \mc G$.   Let $l$ be such that $K\cap [A]^{<\infty}\con [A]^{\le l}$. It follows that if $F\in\hom_0(\widehat{\theta})\rest B$, then $I:=\conj{n<m}{F\cap X_n \neq \buit}$ has at most cardinality $l$, hence using that $F=\bigcup_{n\in I} F\cap X_n$ has cardinality at most $l \cdot r$. We define now the partial ordering on $Y$ by $k\prec l$ when $k<l$ and $\theta(k)<\theta(l)$. Observe that for this partial ordering an antichain is simply  a $0$-homogeneous subset of $Y$. So, width of $\prec$ is at most $l \cdot r$, and consequently by Dilworth's theorem we obtain that $Y$ is a union of at most $l\cdot r$ many chains of $\prec$, or, in other words, a union of at most $l\cdot r$ many 1-homogeneous sets of $\widehat{\theta}$, hence $B\in \mc G$.

If $(I_n)_n$ is downwards consecutive, then we similarly define the compact and hereditary family
\[
K = \{ s \in \Hom_1(\widehat{\theta} \restriction Y) : \#(s \cap X_n) \leq 1 \text{ for all } n \in \mathbb{N} \},
\]
and proceed to prove that $\mathcal{G} \restriction Y = \mathcal{B}(\mathbf{p}^K) \restriction Y$ by changing the roles of color 0 and color 1. 
\end{proof}

Since $\langle\hom(\widehat{\theta})\rangle= \mc B(\mb p^{\hom_0(\widehat{\theta})})\sqcup \mc B(\mb p^{\hom_1(\widehat{\theta})})$, the previous proposition suggests the following weakening of Question \ref{square-union-B}. 

\begin{question}
Let $\ideal$ be the square sum of two $\mc B$-ideals and $Y\not\in \ideal$. Is there $X\subseteq Y$ with $X\not\in \ideal$  such that $\ideal\rest X$ is 
 a $\B$-ideal (equivalently, non-pathological)?
\end{question}

A natural dual to hereditary non-pathology arises naturally.

\begin{question}
Does there exist a hereditarily pathological \( F_\sigma \)-ideal? Specifically, is there an \( F_\sigma \)-ideal such that every restriction to a  positive set is pathological?    
\end{question}

Theorem \ref{j4itogjoidfg} can be extended to any dimension \( d \) using the concept of \emph{Devlin types}.  
\begin{teo}[{\cite{devlin1979}}]
 For every $d\in \N$ there is a number $t_d\in \N$ satisfying:
  \begin{enumerate}[(i),wide=0pt]
     \item If $\ccal:[\Q]^d\to r$ is an arbitrary (finite) coloring, then there is some order-isomorphic copy $X$ of $\Q$ such that $\ccal$ uses at most $t_d$
colors on $[X]^d$.
    \item There exists a coloring $\ccal_d: [\Q]^d\to t_d$ such that if $X\con \Q$ is order-isomorphic to $\Q$, then $\ccal_d$ uses all colors on $[X]^d$.\qed
\end{enumerate}
\end{teo}
The number \( t_d \)  represents the count of canonical patterns of arbitrary \( d \) points on the binary tree \( 2^{<\mathbb{N}} \). It is interesting to note that the sequence $(t_d)_d$ is a well known sequence of numbers known as the \emph{odd tangent numbers}, because \( t_d = T_{2d-1} \), where 
\(
\tan(z) = \sum_{n=0}^\infty ({T_n}/{n!}) z^n.
\) For a comprehensive discussion and detailed information, we refer the reader to \cite[Section 6.2]{Todorcevic2010}. From this we get the following generalization of Theorem \ref{j4itogjoidfg}. 
\begin{defi}[$\Q_d$-colorings]   
A {\em $\Q_d$-coloring} is some $\widehat{\theta}^d:=\ccal_d\circ \theta^n:[\N]^d\to t_d$, $\{n_j\}_{j<d}\mapsto \ccal_d(\{\theta(n_j)\}_{j<d})$ for a bijection $\theta:\N\to \Q$.       
\end{defi}

\begin{teo}\label{j4itogjoidfg1}\label{Sierpinski1}
    Suppose that $\theta: \N\to \Q$ is an enumeration. Then for every coloring $\ccal:[\N]^d\to l$ there is some $M\con \N$ such that $\theta(M)$ is order-isomorphic to $\Q$ and such that $\hom(\widehat{\theta}^d)\rest M\con \hom(\ccal)$.     
\end{teo}
\begin{proof}
Given an enumeration  $\theta: \N\to \Q$  and a coloring  $\ccal:[\N]^d\to l$, we define $\mathscr{d}: [\Q]^d \to t_d\times l$, 
$$\mathscr{d}(\{q_j\}_{j<d}):= (\ccal_d(\{q_j\}_{j<d}), \ccal(\{\theta^{-1}(q_j)\}_{j<d})).$$
By Devlin's Theorem there is an order-isomorphic copy $Q\con \Q$ of $\Q$ such that $\mathscr{d}$ uses at most $t_d$ many colors on $[Q]^d$. In fact, by the definition of $\mathscr{d}$, it has to use exactly $t_d$ many colors, which are exactly determined by the first coordinate of the colors used in $[Q]^d$, that is, there is $(k_j)_{j<d}\in l^d $ such that $\mathscr{d}:[Q]^d\to \{(j, k_j)\}_{j<d}$. Let $N:=\theta^{-1}(Q)$. It follows that if $A\con N$ is $\widehat{\theta}^d$-homogeneous with color  $j<d$, then   it is  $\ccal$-homogeneous with color $k_j$. This means that $\langle \hom(\widehat{\theta}^d)\rangle \rest N\con \langle \ccal \rangle$.  \end{proof}

\begin{question}
Are the \( \mathbb{Q}_d \)-coloring ideals non-pathological? Are they hereditarily non-pathological?    
\end{question}

\section{Tall colorings}

In this section we analyze tall $\mc B$-ideals with the aim of associating to them natural c-colorings ideals. We known that \(c\)-coloring ideals are effectively tall. Our focus now shifts to investigating the converse: when a tall \(\mathcal{B}\)-ideal contains a \(c\)-coloring ideal. 
We will demonstrate that every tall ideal that can be \(\mathcal{B}\)-represented in some space \(C(K)\) with \(K\) countable must contain a \(c\)-coloring ideal. Notably, these colorings arise from classical properties of Schauder bases. 

As a consequence, we will construct an example of a \(\mathcal{B}\)-ideal—specifically a \(c\)-coloring ideal—that is not representable in \(c_0\).

\subsection{\texorpdfstring{The block sequence coloring and $c_0$}{The block sequence coloring and c0}}\label{subsec:block_seq_coloring}

Recall that a {\em basic sequence} in a Banach space is a sequence which is a (Schauder) basis of the closed subspace it generates. Throughout this part we assume that every basic sequence is seminormalized. This implies, in particular, that the sequence of {\em coordinate functionals} $\mb f = (f_n)_n$, which map each $x\in X$ to its $n$-th coefficient in terms of the basis, is uniformly bounded. Given $x\in X$, its {\em support} (with respect to the basis) is the set $\supp x := \{n\in\N : f_n(x)\neq 0\}$. 
Finally, a {\em block subsequence} of $\mb e$ is a finitely supported sequence $\mb x$ satisfying $\max\supp x_n < \min\supp x_{n+1}$ for all $n$. Vectors with finite support are clearly dense in $X$, thus in order to study ideals of the form $\mc B(\mb x)$ we can always assume that every $x_n$ has finite support. Note that in $C(K)$ spaces we can also assume that the $x_n$'s are non-negative functions by Proposition \ref{prop:positive_functions}, but this is not the case in arbitrary function spaces, as shown by Example \ref{ex:ideal_not_positive_coords}.

\begin{defi}[The block sequence coloring]\label{def:block_seq_coloring}
    Let $\mb e$ be a basis for a Banach space $X$ and let $\mb x=(x_n)_n$ be a finitely supported sequence in $X$. We define the {\em block sequence ($\mk{bs}$) coloring} (with respect to $\mb f$) $\mk{bs}= \mk{bs}_{\mb x}: [\N]^2\to \{0,1\}$ for $m<n$ by $\mk{bs}(\{m,n\}_<) = 1$ exactly when
    \begin{equation}\label{cantor5}
    \text{$\max\supp x_m\le \max\supp x_n$ and for all $k \le \max\supp x_m $ one has that $|f_k^*(x_n)|\le \frac{1}{2^{k+1}\, 2^{m+1}}$.}
    \end{equation}
\end{defi}

Clearly $\mk{bs}$ is a $1$-comparability coloring. Fix now a Banach space $X$ with a normalized basis $\mb e$ and a finitely supported sequence $\mb x = (x_n)_n$ in $X$. For each $F\subseteq\N$ define $(x_n^F)_{n\in F}$ by
\begin{equation}
    x_n^F := \sum_{k\in\supp x_n \setminus [0,\max_{m\in F\cap n}\max \supp x_m]} f_k(x_n) e_k.
\end{equation}
Observe that $(x_n^F)_{n\in F, x_n^F\neq 0}$ is a block sequence.

 For 1-homogeneous sets $H$ of $\mathfrak{bs}$ the sequence $(x_n^H)_{n\in H}$ are equivalent:

\begin{prop}\label{iji5gojhtinbgv}
    Suppose that $H\in \hom_1(\mk{bs})$. Then for 
    every sequence of scalars $(a_n)_{n\in H}$ with at most finitely many non-zero we have
    \begin{equation}\label{ojtyohiyterfds}
        \left|\norm{\sum_{n\in H} a_n x_n} - \norm{\sum_{n\in H} a_n x_n^H}\right|\le \frac{\max_{n\in H} |a_n|}{2^{\min H}}.     
    \end{equation}
    Consequently, if $H$ is infinite, $\mc B((x_n)_{n\in H}) = \mc B((x_n^H)_{n\in H})$ and, if in addition $(x_n)_{n\in H}$ is a basic sequence, then $(x_n)_{n\in H}$ and $(x_n^H)_{n\in H}$ are equivalent.
\end{prop}
\begin{proof}
    Enumerate $H = \{n_i\}_i$ in an increasing way. Then $(\max\supp x_{n_i})_i$ is increasing and for every $i$ one has that
    \begin{align*}
        \nrm{x_{n_{i+1}} - x_{n_{i+1}}^H} & = \norm{\sum_{k\in \supp x_{n_{i+1}}\cap [0,\max\supp x_{n_i}]} f_k(x_{n_{i+1}}) e_k} \le \sum_{k\in \supp x_{n_{i+1}}\cap [0,\max\supp x_{n_i}]} |f_k(x_{n_{i+1}})| \\
        & \le \sum_{k\in \supp x_{n_{i+1}}\cap [0,\max\supp x_{n_i}]} \frac{1}{2^{k+1}\, 2^{n_i+1}} \le \frac{1}{2^{n_i+1}},
    \end{align*}
    hence
    \[
    \left|\norm{\sum_{i=0}^\infty a_{n_i}x_{n_i}} - \norm{\sum_{i=0}^\infty a_{n_i} x_{n_i}^H}\right| \le \sum_{i=0}^\infty |a_{n_i}| \norm{x_{n_{i}}-x_{n_i}^H} = \sum_{i=1}^\infty |a_{n_i}| \norm{x_{n_i} - x_{n_{i}}^H} \le \max_i |a_{n_i}| \sum_{i=0}^\infty \frac{1}{2^{n_i+1}} \le \frac{\max_{n\in H}|a_n|}{2^{\min H}}.
    \]
 \end{proof}

\begin{prop}\label{j2iojsdfjnsdfsd}
Suppose that $\mb x$ is weakly null. Then 
 $\hom_0(\mk{bs}_\mb x)\con\fin$, and consequently $\mk{bs}_\mb x$ is a 0-asymmetric  1-comparability coloring and
    \[
    \langle\hom(\mk{bs}_\mb x)\rangle = \mc B(\mb p^{\hom_0(\mk{bs}_\mb x)}).
    \]
\end{prop}
\begin{proof}
Set $\mk{bs}:=\mk{bs}_\mb{x}$. Suppose otherwise that $H\in \hom_0(\mk{bs})$  is infinite, and let $H=\{n_j\}_j$ be its increasing enumeration. Passing to an infinite subset if necessary we assume that $\max \supp x_{n_j}\le \max\supp x_{n_{j+1}}$ for every $j$. Then for every $j>0$ 
there is some $k(j)\in [0,\max\supp x_{n_0}]$ such that $f_{k(j)}^*(x_{n_j})\ge 1/2^{k(j)+n_0+2}$, so for infinitely many $j'$s there is some $k_\infty\in [0,\max\supp x_{n_0}]$ such that $f_{k_\infty}^*(x_{n_j})\ge 1/2^{k_\infty+n_0+2}$, which is impossible because we are assuming that $\mb x$ is weakly null.
\end{proof}

Tall $\mathcal{B}$-ideals represented in $c_0$ exhibit a particular structural property. This property will be used in the next subsection to provide an example of a $\mathcal{B}$-ideal that cannot be $\mathcal{B}$-represented in $c_0$.

\begin{prop}
\label{fin-antichain}
Suppose that a tall $\mc B$-ideal $\mc I$ can be $\B$-represented in $c_0$. Then there is a $0$-asymmetric $1$-comparability coloring $\ccal: [\N]^2\to 2$ such that $\hom(\ccal) \con \mc I$. 
\end{prop}
\begin{proof} 
    For suppose now that $\mc I = \mc B(\mb x)$, for some sequence in $\mb x=(x_n)_n$, that we may assume that it is finitely supported and positive, and since we are assuming that $\mc I$ is tall, this means that $\mb x$ is weakly null.  Set $\mk{bs}:=\mk{bs}_\mb x$.  It suffices to show that $\hom_1(\mk{bs})\subseteq \B({\bf x})$. Suppose that $H\in \hom_1(\mk{bs})$ and fix  $F\con H$ and $k\in\bigcup_{n\in F}\supp x_n$. Let $I := \conj{n\in F}{k\in\supp x_n} = \{n_l\}_{l=0}^p$ in increasing order. Then we have that for every $l < k$ one has that $x_{n_{l+1}}(k)\le 1/(2^{k+1} 2^{n_l})$, so
    \[
    \left|\sum_{n\in F} x_n(k)\right|\le \sum_{l=0}^p |x_{n_l}(k)|\le |x_{n_0}(k)|+\frac{1}{2^{k+1}}\sum_{l<p} \frac{1}{2^{n_l}}\le \sup_n \nrm{x_n}+1.
    \]
    Since $k$ and $F$ are arbitrary, we obtain that $H\in \mc B(\mb x)$.
\end{proof}

\begin{prop}
Suppose that $\mb x$ is a sequence in $c_0$ with $\sup_n \|x_n\|\leq 1$. Then $\mc B(\mb x)$ is tall if and only if $\mc B_C(\mb x)$ is tall for every $C>1$.
 
\end{prop}
\begin{proof}
    Suppose that $\mc B(\mb x)$ is tall. We know that this implies that $\mb x$ is weakly-null and that $\hom_0(\mk{bs})\con \fin$. Fix $C > 1$ and suppose that $H$ is $1$-homogeneous. Then for every $F\con H$ we know by Proposition \ref{iji5gojhtinbgv} that 
    \[
    \norm{\sum_{n\in F} x_n}\le \norm{\sum_{n\in F} x_n^F} +\frac{1}{2^{\min F}} = \max_{n\in F} \nrm{x_n^F}+\frac{1}{2^{\min F}}\le 1 + \frac{1}{2^{\min H}}.
    \]
     If we choose $H$ such that $1 + 1/2^{\min H} \le C$, then $H\in \mc B_C(\mb x)$, proving that $\mc B_C(\mb x)$ is tall.
\end{proof}

\subsection{\texorpdfstring{A tall $\mc B$-ideal not representable in $c_0$}{A tall B-ideal not representable in c0}}
\label{nonc0}

We fix an enumeration $\theta$ of $\Q\cap ]0,1[$.
Let $\mk{conv}_\theta=\mk{conv}: [\N]^3\to 2$ be the coloring that declares 
$\mk{conv}\{k,l,m\}_<\,=1$ exactly when $|\theta(l)-\theta(m)|<1/({k+1})$.

\begin{teo}
\label{exG}  Neither of $\hom(\mk{conv})$ or $\mc B(\mathbf p^{\hom_0(\mk{conv})})$   contain any $\hom(\ccal)$ for  $\ccal:[\N]^2\to 2$, and   $\mc B(\mathbf p^{\hom_0(\mk{conv})})$ is  $\mc B$-representable in $C(\mc S)$, and not in  $c_0$, where $\mc S$ is the Schreier family.
\end{teo}

\begin{proof}

First we show  that every  infinite $\mk{conv}$-homogeneous set is of color 1. In fact, given an infinite set $A$, we can assume that $(\theta(n))_{n\in A}$ is convergent. Thus for every $k\in A$ there are $l<m$ in $A$ such that 
$d\{k,l,m\} =1$ and thus $[A]^3$ is  $\mk{conv}$-color 1. Moreover,
\begin{claim}\label{jjijtrejteritoe}
$\hom_0(\mk{conv})\con \conj{s\in \fin}{ \# s \le \min s+2}$:    
\end{claim}
In fact, suppose that $H\in \hom_0(\mk{conv})$,  set $k:=\min H$ and consider the partition of $[0,1]$ into $k+1$ intervals of length $1/(k+1)$. Let $I$ be any interval in that partition, then $I\cap \{\theta(n): n\in H\setminus\{k\}\}$ has at most one element. 

Next, we see that there is no $\ccal:[\N]^2\longrightarrow 2$ with $\hom(\ccal)\subseteq \hom(\mk{conv})$. It suffices to show, by the universality of $\Q$-colorings (see  Theorem \ref{Sierpinski}), that for any given $A\in[\N]^\omega$ with $\theta(A)$ order isomorphic to $\Q$, there is an infinite $\widehat\theta$-homogeneous set $H\subseteq A$ such that $H\notin\h(\mk{conv})$. Which in turn is achieved when  $H$ contains $\mk{conv}$-homogeneous sets of color $0$  of  any finite cardinality.

\begin{claim}\label{lioj5tji5rjgrdf}
If   $A\in[\N]^\omega$  is such that  $\theta(A)$  is order isomorphic to $\Q$, then $A$ contains $\mk{conv}$-homogeneous sets of color 0 of any finite cardinality.
\end{claim}

We first show that in any interval $I$ of $\theta(A)$ there are $\mk{conv}$-homogeneous sets  with color $0$ of any finite cardinality. In fact,  for a fixed $r\in \N$ we find   $D\in \hom_0(\mk{conv})$ such that   $D\subseteq A$, $|D|=r$  and  $\theta(D)\subseteq I$ is an increasing sequence.   
Let $i<j\in A$ with $\theta(i)< \theta(j)$ and $I=]\theta(i), \theta(j)[$. For $k=0,\ldots , 2r$,  pick $a_k\in A$ increasing such that $a_0=i$ and $\theta({a_0})\leq \theta({a_k})<\theta(a_{k+1})\leq \theta(j)$ for $k<2r$. Let $a\in A$ be such that 
\[
\frac1{a+1}<\min\{\theta(a_{k+1})-\theta({a_{k})}:\; 0\leq k< 2r\}.
\]
Next, pick $d_k\in A$ for $1\leq k\leq r$ such that $a<d_k<d_{k+1}$ and  $\theta(a_{2k})<\theta(d_k)<\theta(a_{2k+1})$.
Let $D=\conj{d_{k}}{1\leq k\leq r}$. 
Then $\mk{conv}\{a,d_k,d_l\}=0$ for all $ 1\leq k<l\leq r$. From this it follows that $D\in \hom_0(\mk{conv})$ and $\theta(D)\subseteq I$ is an increasing sequence.

Now, using the previous fact, define a sequence of intervals $I_n=]x_n,y_n[$ of $\theta(A)$ and sets $D_n\subseteq A$ such that: 
\begin{enumerate}[(a), wide=0pt]
    \item $y_n<x_{n+1}$,
    \item $|D_n|=n$,
\item $\theta(D_n)\subseteq I_n$ is an increasing sequence and   
\item $D_n\in\hom_0(\mk{conv})$. 

\end{enumerate}

Let $H$ be the union of all $D_n$. Then $H$ is as promised, that is, $H$ is $\widehat\theta$-homogeneous and $H\not\in \langle \hom(\mk{conv})\rangle$. This finishes the proof of the claim. 

Now we construct a non-pathological $F_\sigma$  ideal not representable in $c_0$.  Let $K=\hom_0(\mk{conv})$ and consider the sequence of evaluations $\mb p=\mb p^K=(p_n)_n$ in $C(K)$. 

\begin{claim}
The ideal $\mc I:=\mc B(\mb p)$  is tall and $\mc B$-representable in $C(\mc S)$ but not in $c_0$.
\end{claim}

To see this, it follows from the Claim \ref{jjijtrejteritoe} that the compact and hereditary family $\hom_0(\mk{conv})$ has Cantor-Bendixson index at most $\omega + 1$. Hence,  $C(\hom_0(\mk{conv}))$ is either isomorphic to $c_0$ or to $C(\mc S)$. We are going to see that $\mc I$ does not contain any $\hom(c)$ with $c:[\N]^2\to 2$, hence $\mc I$ cannot be $\mc B$-representable in $c_0$, and in addition $\mc I$ is $\mc B$-representable  in $C(\mc S)$.  By \eqref{eq:simpler_expression_p^K},  $M\in \mc B(\mb p)$ exactly when the size of the $\mk{conv}$-homogeneous subsets  of $M$ with color $0$ is  uniformly bounded, hence $\hom(\mk{conv})\subseteq \mc I$, and by Claim \ref{lioj5tji5rjgrdf} it follows that $\hom(\widehat{\theta})\rest A\not\con \mc I$ if $\theta(A)$ is order isomorphic to $\Q$, so by the universality of those restrictions, $\mc I$ and the Proposition \ref{fin-antichain} we obtain that $\mc I$ cannot be $\mc B$-represented in $c_0$.
    \end{proof}

\begin{rem}
The ideal $\mathcal{B}(\mathbf{p}^{\hom_0(\mathfrak{conv})})$ is not a $P$-ideal. To see this, note that by Claim \ref{lioj5tji5rjgrdf}, for every $n \in \mathbb{N}$, we can find infinite subsets $A_n$ whose tails contain $0$-homogeneous subsets of cardinality $n$. This implies that every pseudo-union $A$ of $\{A_n\}_n$ will contain arbitrarily large $0$-homogeneous subsets. Consequently, $\sup_{F \in [A]^{<\infty}} \left\|\sum_{n \in F} p_n \right\| = \infty$, and hence $A \notin \mathcal{B}(\mathbf{p}^{\hom_0(\mathfrak{conv})})$.

\end{rem}

\begin{question}
Is $\hom(\mk{conv})$ a $\mc B$-ideal?    
\end{question}

\begin{rem}
As far as we know, the ideals $\hom(\mk{conv})$ and $\mc B(\mb p^{\hom_0(\mk{conv})})$  are the first concrete examples of tall Borel ideals not Katetov above $\mc R$ but admitting a Borel selector; a phenomenon that was known to occur  in an abstract sense without explicit examples (see \cite{GrebikUzca2018}). 
    \end{rem}

\begin{rem} 
In contrast with what we have seen,  there is a important difference between the global and local version of the problem of finding colorings for tall ideals.  Indeed, it is shown in  \cite{HMTU2017} that every tall $F_\sigma$ ideal contains some restriction $\hom(c)\rest A$ of a coloring $c:[\N]^2\to 2$. To see this, recall that an ideal $\mathcal{I}$ is $p^+$, if for every decreasing sequence $(A_n)_n$ of sets in $\mathcal{I}^+$, there is $A\in \mathcal{I}^+$ such that $A\setminus A_n$ is finite for all
$n\in\N$. It is well known that   every $F_\sigma$ ideal is $p^+$ (see \cite[Lemma 3.3]{HMTU2017} for a proof).
An ideal $\mathcal{I}$ is $q^+$, if for every $A\in \mathcal{I}^+$ and every
partition $(F_n)_n$ of $A$ into finite sets, there is $S\in\mathcal{I}^+$
such that $S\subseteq A$ and $S\cap F_n$  has at most one element for each $n$.  Equivalently, if  $c$ is the coloring on $A$ associated to the partition, that is, $c\{i,j\}=0$ if $i,j$ are in the same piece of the partition, then  $\hom(c)\not\subseteq \ideal$. An ideal is {\em selective} if it is $p^+$ and $q^+$ (this is an equivalent formulation of this notion, see \cite[Lemma 7.4]{Todorcevic2010}).  
Analytic selective ideal have the following  Ramsey property: For any partition of $[\N]^\omega$ into two sets one of which is analytic, there is $A\in\ideal^+$ such that $[A]^\omega$ is a subset of one of the pieces of the partition (for a proof see \cite[Lemma 7.23]{Todorcevic2010}). Now suppose  $\ideal$ is an analytic selective ideal. By using $\ideal$ as one of the pieces of the partition, we conclude that there is $A\in\ideal^+$ such that $[A]^\omega\subseteq \ideal^+$. In particular, $\ideal$ is not tall. Thus, if $\ideal$ is a tall $F_\sigma$ ideal, then it is not selective and necessarely is not $q^+$ and from above we have a coloring $c$ and an infinite set $A$ such that  $\hom(c)\rest A\subseteq \ideal$.
\end{rem}

\subsection{\texorpdfstring{Ideals $\mc B$-representable in $C(K)$, $K$ countable}{Ideals B-representable in C(K), K countable}}
\label{lmklfmklsdfs}

In $c_0$, we have that $\mathcal{B}(\mathbf{x}) = \mathcal{P}(\N)$ for block sequences $\mathbf{x}$ of the unit basis of $c_0$. Thus, $\mathcal{B}$-ideals represented by weakly null sequences in $c_0$ are always tall (see Theorem \ref{c0saturacion} for more details). However, this is not generally true when $\mathbf{x}$ is a block sequence of a basis of some Banach space. Nevertheless, the spaces $C(K)$, where $K$ is compact and countable, can be treated in a similar but more complex manner than $c_0$. We will show that tall ideals that can be $\mathcal{B}$-represented in $C(K)$, where $K$ is countable, contain some non-pathological c-coloring ideal, and consequently have a Borel selector. Every such space of functions is isometric to $C(\mathcal{F})$, where $\mathcal{F} \subseteq \fin$ is a compact and hereditary family of finite sets. This representation is more convenient for our work.
In particular, the spaces $C(\mathcal{F})$ have natural Schauder bases: for each $t \in \mathcal{F}$, we let
\[
[t] := \{u \in \mathcal{F} : u \sqsubseteq t\}.
\]
Let $(t_k)_k$ be an enumeration of $\mathcal{F}$ such that $t_0 = \emptyset$ and if $t_m \sqsubset t_n$, then $m < n$, and let $\theta : \mathcal{F} \to \mathbb{N}$ be the inverse of the enumeration. Define
\begin{equation}
\label{Schau-S}
f_k := \mathbbm{1}_{[t_k]}.
\end{equation}
Then $(f_k)_k$ is a normalized, monotone, shrinking Schauder basis for $C(\mathcal{F})$, called the \textit{node basis} of $C(\mathcal{F})$ (see for instance, \cite[Lemma 5.6. and above]{AJO}). We will also denote $f_t := f_{\theta(t)}$. Observe that if $x = \sum_{k=0}^\infty f_k^*(x) f_k \in C(\mathcal{F})$, then
\begin{equation}\label{eq:node_basis_evaluation}
    x(t) = \sum_{t_k \sqsubseteq t} f_k^*(x) = \sum_{u \sqsubseteq t} f_u^*(x)
\end{equation}
for all $t \in \mathcal{F}$.

Take a sequence $\mb x = (x_n)_n$ in $C(\F)$ which is finitely supported (with respect to the basis). By Proposition \ref{prop:positive_functions}, to study the ideal $\mc B(\mb x)$ we can assume that every $x_n$ is a non-negative function on $\F$ (note that, by \eqref{eq:node_basis_evaluation}, we have $\supp |\mb x| \subseteq \supp \mb x$). For each $s\subseteq\N$ define, as in Subsection \ref{subsec:block_seq_coloring}, $(x_n^s)_{n\in s}$ through
\[
x_n^s := \sum_{k \in S(n,s)} f_k^*(x_n) f_k,
\]
where
\(
S(n,s) := \supp(x_n) \setminus [0, \max_{m \in s \cap n} \max(\supp(x_m))].
\)
We define now $\mc G$ as the family of $s\con\N$ for which there is some $t\in\F$ such that for all $n\in s$ there is $u\sqsubseteq t$ with $f_u^*(x_n)\geq 2^{-\theta(u)}$ and $f_u^*(x_m) = 0$ for all $m\in s\cap n$. 


\begin{lema}
    $\mc G$ is a compact and hereditary family of finite sets and, if $\mb x$ is weakly null, then its Cantor-Bendixson rank is at most that of $\F$.
\end{lema}
\begin{proof}
    We first show that if $s\in\mc G$ and $t\in\F$ witnesses it, then $\#s\leq\#t + 1$, so in particular $\mc G$ does not contain infinite sets. To see this write $s = \{n_j\}_{j=0}^k$ in increasing order. For each $j$ there exists $u_j\sqsubseteq t$ such that $f_{u_j}^*(x_{n_j})\geq 2^{-\theta(u_j)}$ but $f_{u_j}^*(x_{n_k}) = 0$ for all $k < j$. In particular $u_j\neq u_k$ for $j\neq k$ and they are all initial segments of $t$, thus $\#s\leq\#t + 1$.

    Now fix a sequence $(s_k)_k$ in $\mc G$ converging to some $A\subseteq\N$, and let us show that $A\in\mc G$. For each $k$ let $u_k\in\F$ be a witness that $s_k\in\mc G$. Passing to a subsequence, we can assume that $(u_k)_k$ is a $\Delta$-sequence with root $t\in\F$, i.e. $t\sqsubseteq u_k$ and $u_k\setminus t < u_{k+1}\setminus t$ for all $k$. Fix $n\in A$ and $k_0$ such that $s_k\cap [0,n] = A\cap [0,n]$ for all $k\geq k_0$. For each such $k$, since $n\in s_k$, we can find $v_{n,k}\sqsubseteq u_k$ such that $f_{v_{n,k}}^*(x_n)\geq 2^{-\theta(v_{n,k})}$ but $f_{v_{n,k}}^*(x_m) = 0$ for $m\in s_k\cap n = A\cap n$. For each $v_{n,k}$ we can have either $v_{n,k}\sqsubseteq t$ or $t\sqsubsetneq v_{n,k}$. If all $v_{n,k}$'s satisfy the latter then they are in particular all different, but they all belong to the support of $x_n$ with respect to the basis $\mb f$, thus contradicting the fact that $x_n$ has finite support. Hence at least one $v_{n,k}$ is an initial segment of $t$, so $t$ witnesses $A\in\mc G$, and in particular $A$ is finite.

    Finally, for each ordinal $\alpha < \omega_1$ denote by $\F^{(\alpha)}$ and $\mc G^{(\alpha)}$ the $\alpha$-th Cantor-Bendixson derivatives of $\F$ and $\mc G$ respectively. We claim that, if $\mb x$ is weakly null, for each $s\in\mc G^{(\alpha)}$ there is $t\in\F^{(\alpha)}$ witnessing it. Fix a countable ordinal $\alpha$ and assume that the result holds for all smaller ones. Let $(\alpha_k)_k$ be a sequence satisfying $\alpha_k + 1 = \alpha$ for all $k$ if $\alpha$ is a successor, or increasing to $\alpha$ if it is limit. Pick $s\in\mc G^{(\alpha + 1)}$ and a non-trivial $\Delta$-sequence $(s_k)_k$ with $s_k\in\mc G^{(\alpha_k)}$ converging to it. For each $k$ pick $u_k\in\F^{(\alpha_k)}$ witnessing it and, passing to a subsequence, assume that $(u_k)_k$ is a $\Delta$-sequence with root $t\in\F$. The argument in the previous paragraph shows that $t$ is a witness for $s\in\mc G$, so it remains to show that $t\in\F^{(\alpha + 1)}$. If this was not the case then $t$ would be a witness of $s_k\in\mc G$ for all but finitely many $k$. Since the sequence $(s_k)_k$ is non-trivial, set $n_k := \min s_k\setminus s$ for each $k$, and find an initial segment $v_k\sqsubseteq t$ such that, in particular, $f_{v_k}^*(x_{n_k})\geq 2^{-\theta(v_k)}$. 
    Passing again to a subsequence we can assume that $v_k = v\sqsubseteq t$ for all $k$, but then the condition $f_v^*(x_{n_k})\geq 2^{-\theta(v)}$ for all $k$ contradicts the fact that $\mb x$ is weakly null. 
\end{proof}


We want to use the {\em Ramsey property} of the family $\mc G^{\max}$ of $\sqsubset$-maximal elements of $\mc G$, but there might be an infinite set $M$ such that $\mc G^{\max}\rest M = \emptyset$. This can be solved with the help of the following notion, introduced by Pudlák and Rödl \cite{PudlakRodl82}.

\begin{defi}\label{o4jtiojrejtre}
    Given a countable ordinal $\alpha < \omega_1$ and $M\in [\N]^\omega$, we say that a family $\mc A\subseteq\fin$ is an {\em $\alpha$-uniform front on $M$} if $\alpha = 0$ and $\mc A = \{\emptyset\}$, or $\alpha > 0$, $\emptyset\notin\mc A$, and for each $n\in M$ the family
    \[
    \mc A_{\{n\}} := \{s\in\fin : \text{$n < s$ and $\{n\}\cup s\in\mc A$}\}
    \]
    is an $\alpha_n$-uniform front on $M/n$, where $\alpha_n + 1 = \alpha$ for all $n$ if $\alpha$ is a successor ordinal, or $(\alpha_n)_{n\in M}$ is a sequence increasing to $\alpha$, if it is limit. We say that $\mc A$ is {\em uniform on $M$} if it is $\alpha$-uniform on $M$ for some $\alpha < \omega_1$, in which case $\alpha$ is called its {\em uniform rank}.
\end{defi}

Uniform fronts on $M$ are indeed fronts on $M$, thus they possess the Ramsey property. Conversely, every Ramsey family has a restriction that forms a uniform front. Moreover, uniformity is preserved under restrictions, and the topological closure of a uniform front coincides with its $\sqsubseteq$-closure, $\overline{\mathcal{A}}^{\sqsubseteq}$. An easy inductive argument shows that the Cantor-Bendixson rank of an $\alpha$-uniform front is $\alpha + 1$. Therefore, closures of uniform fronts exhaust, up to homeomorphism, all countable compact Hausdorff spaces. 

A classical example of an $\omega$-uniform front is the Schreier barrier $\mathfrak{S} := \{s \in \mathrm{Fin} : \#s = \min s + 1\}$, which can be used to construct uniform fronts of arbitrarily high rank. For example, if $\mathcal{A}$ and $\mathcal{B}$ are $\alpha$- and $\beta$-uniform fronts on $M$, respectively, then
\[
\mathcal{A} \oplus \mathcal{B} := \{s \cup t : \text{$s \in \mathcal{B}$, $t \in \mathcal{A}$, and $s < t$}\}
\]
is $(\alpha + \beta)$-uniform. One can also define a product 
$\mathcal{A} \otimes \mathcal{B} := \{s_1\cup\cdots \cup s_n: \text{$\{s_i\}_i\con \mc A$ block, and  $\{\min s_i\}_i\in\mathcal{B}$}\}$.
Note also that if $\alpha < \beta$ and we are given an $\alpha$-uniform front $\mathcal{A}$, then we can always find a $\beta$-uniform front $\mathcal{B}$ such that $\mathcal{A} \subseteq \overline{\mathcal{B}}$: indeed, just take $\gamma$ such that $\alpha + \gamma = \beta$ and let $\mathcal{B} := \mathcal{A} \oplus \mathcal{A}'$, where $\mathcal{A}'$ is $\gamma$-uniform. For each non-empty $s \subseteq \mathbb{N}$, we denote ${}_*s := s \setminus \{\min s\}$.


\begin{prop}
    Every compact family $\mc G$ on $M$ is included in the closure of a uniform front $\mc A$ on $M$.
\end{prop}
\begin{proof}
    We further claim that the uniform rank of $\mc A$ can be taken equal to the rank $\rho$ of the well-founded tree $(\mc G, \sqsubset)$. The proof goes by induction on $\rho$. Define $\eta_n + 1 = \rho$ for $n\in M$ if $\rho$ is a successor ordinal, or let $(\eta_n)_{n\in M}$ be a sequence increasing to $\rho$ if it is limit. For each $n\in M$ the rank $\rho_n$ of the tree $(\mc G_{\{n\}}, \sqsubset)$ is strictly smaller than $\rho$, so by the induction hypothesis there is a $\rho_n$-uniform front $\mc A_n$ on $M/n$ such that $\mc G_{\{n\}}\subseteq\overline{\mc A_n}$. By the remark made above, replacing $\mc A_n$ if necessary we can assume that $\mc A_n$ is $\max\{\rho_n, \eta_n\}$-uniform. Let
    \[
    \mc A := \{s\in [M]^{<\omega}\setminus\{\emptyset\} : {}_*s\in\mc A_{\min s}\}.
    \]
    It easily follows that $\mc A_{\{n\}} = \mc A_n$ for all $n$, thus $\mc A$ is a $\rho$-uniform front on $M$. Moreover, given $s\in\mc G$ non-empty, for $n := \min s$ we have ${}_*s\in\mc G_{\{n\}}\subseteq\overline{\mc A_n} = \overline{\mc A_{\{n\}}}^{\sqsubseteq}$, thus there is $t\in\mc A$ such that $s\sqsubseteq t$, and hence $s\in\overline{\mc A}^{\sqsubseteq} = \overline{\mc A}$.
\end{proof}

\begin{defi}[The bounded coloring]
Let $\mc A$ be a uniform front of rank at least $1$ such that $\mc G\con \overline{\mc A}$, and let $\mc V := \mc A\oplus \mc A$. Each element $s\in \mc V$ has a unique decomposition $s = s[0]\cup s[1]$ with 
$s[0],s[1]\in \mc A$ and $s[0] < s[1]$.
We define the {\em bounded coloring} $\mk{b} = \mk{b}_{\mb x}: \mc V\to \{0,1\}$ by declaring that $\mk{b}(s) = 1$ exactly when
\[
\norm{\sum_{n\in s[0]} x_n^{s[0]}}\ge \frac12\norm{\sum_{n\in s[1]} x_n^{s[1]}}.
\]
\end{defi}



Observe that $\mk{b}$ is a 0-comparability coloring.


\begin{prop}\label{lkmkgmndsfgsd}
    If $\mc B(\mb x)$ is tall, then $\hom_1(\mk{bs})\cap \hom_0(\mk{b})\con \fin$.
\end{prop}
\begin{proof}
    For suppose otherwise that $H\in\hom_1(\mk{bs})\cap \hom_0(\mk{b})$ is infinite. Then there is a unique partition $(s_n)_n$ of $H$ consisting of $<$-consecutive elements of $\mc A$. Since $\mk{b}(s_n\cup s_{n+1})=0$ for every $n$, it follows that $\nrm{\sum_{k\in s_{n+1}} x_k^{s_{n+1}}}>2\nrm{\sum_{k\in s_n} x_k^{s_{n}}}$. In particular $\nrm{\sum_{k\in s_{1}} x_k^{s_{1}}}>2\nrm{\sum_{k\in s_0} x_k^{s_{0}}}\ge 0$, so $C := \nrm{\sum_{k\in s_{1}} x_k^{s_{1}}}>0$, and by induction
    \[
    \norm{\sum_{k\in s_{n+1}} x_k^{s_{n+1}}}\ge 2^n C\text{ for every $n$.}
    \]
    Since $H\in \hom_1(\mk{bs})$, it follows from this and \eqref{ojtyohiyterfds} that
    \[
    \norm{\sum_{k\in s_{n+1}} x_k}\ge 2^{n}C -1 \text{ for every $n$,}
    \]
    and consequently $H\notin \mc B(\mr x)$. Since $\hom_1(\mk{bs})\cap \hom_0(\mk{b})$ is hereditary, this implies that $[H]^\omega\cap \mc B(\mb x)=\buit$, a contradiction. 
\end{proof}

\begin{defi}[The tall coloring]
By combining the tall coloring with the block sequence coloring of Definition \ref{def:block_seq_coloring}, we define the {\em tall coloring} $\mk{t} = \mk{t}_{\mb x}:\mc V\to 2^2$ for $s \in \mc V$ by 
\[
\mk t(s) := (\mk{bs}(\{\min s, \min{}_*s\}), \mk{b}(s)).
\]
\end{defi} 

Observe that $\hom_{(i,j)}(\mk t)\cap [\N]^\omega = (\hom_i(\mk{bs})\cap \hom_j(\mk{b}))\cap [\N]^\omega$ for every $i,j$. 

\begin{prop}
\label{tall-coloring2}
If $\mc B(\mb x)$ is tall, then $\hom(\mk t)\con \mc B(\mb x)$.
\end{prop}
\begin{proof}
    We have $\hom_{(i,j)}(\mk t)\subseteq\fin\subseteq\mc B(\mb x)$ for all $(i,j)\neq (1,1)$ by Propositions \ref{j2iojsdfjnsdfsd} and \ref{lkmkgmndsfgsd}, so it rests to show that $\hom_{(1,1)}(\mk t)\con \mc B(\mb x)$. For suppose that $H$ is an infinite $(1,1)$-homogeneous set for $\mk t$. Fix $s_0\in\mc A$ such that $s_0\sqsubseteq H$ and $F\subseteq H$ finite. Set $F_0 := F\cap s_0$ and $F_1 := F\setminus s_0$. Observe that $\nrm{\sum_{n\in F_0} x_n} \leq \#s_0\cdot \max_{n\in s_0} \|x_n\|$, so to show $H\in\mc B(\mb x)$ it is enough to find an upper bound for $\nrm{\sum_{n\in F_1} x_n}$ which is independent of $F$.
    
    \begin{claim}
        For every $R\subseteq H$ finite we have
        \[
        \norm{\sum_{n\in R} x_n} \leq 1 + \sup_{s\in\mc G\rest R} \norm{\sum_{n\in s} x_n}.
        \]
    \end{claim}
    \begin{proof}[Proof of claim:]
        Fix $t\in\F$ and set
        \[
        s := \{n\in R : \text{$\exists u\sqsubseteq t$ with $f_u^*(x_n)\geq 2^{-\theta(u)}$}\}.
        \]
        Given $n\in s$, the corresponding $u\sqsubseteq t$ and some $m\in s\cap n$, we must have $f_u^*(x_m) = 0$, because otherwise $\theta(u)\in\supp x_m$ and, from the fact that $\mk{bs}(\{m,n\}) = 1$, we would conclude that $|f_u^*(x_n)|\leq 2^{-(n+1)} 2^{-\theta(u)}$, a contradiction. Thus $s\in\mc G$, so using \eqref{eq:node_basis_evaluation}
        \[
        \sum_{n\in R} x_n(t) = \sum_{n\in s} x_n(t) + \sum_{n\in R\setminus s} x_n(t) \leq \sup_{s'\in\mc G\rest R} \norm{\sum_{n\in s'} x_n} + \sum_{n\in R\setminus s} \sum_{u\sqsubseteq t} f_u^*(x_n).
        \]

        To bound the second term enumerate $R\setminus s = \{n_j\}_{j=0}^k$ in increasing order and fix $u\sqsubseteq t$. Let $j(u)$ be the minimum $j$ such that $\theta(u)\leq\max\supp x_{n_j}$. Then for $j < j(u)$ we have $f_u^*(x_{n_j}) = 0$, and for $j > j(u)$ we have $\mk{bs}(\{j(u),j\}) = 1$, so $|f_u^*(x_{n_j})| \leq 2^{-(n_j+1)} 2^{-\theta(u)}$. Finally, $f_u^*(x_{n_{j(u)}}) < 2^{-\theta(u)}$ because $n_{j(u)}\notin s$, so all together
        \[
        \sum_{n\in R\setminus s} f_u^*(x_n) = \sum_{j=0}^k f_u^*(x_{n_j}) < 2^{-\theta(u)} + 2^{-\theta(u)} \sum_{j=j(u)+1}^k 2^{-(n_j+1)} < 2^{-\theta(u)+1}. 
        \]
        Hence
        \[
        \sum_{n\in R} x_n(t) \leq \sup_{s'\in\mc G\rest R} \norm{\sum_{n\in s'} x_n} + \sum_{n\in R\setminus s} \sum_{u\sqsubseteq t} f_u^*(x_n) \leq \sup_{s'\in\mc G\rest R} \norm{\sum_{n\in s'} x_n} + \sum_{u\sqsubseteq t} 2^{-\theta(u) + 1} \leq \sup_{s'\in\mc G\rest R} \norm{\sum_{n\in s'} x_n} + 1.
        \]
        Since $t\in\F$ was arbitrary, the claim holds.
        %
        %
        %
    \end{proof}
    
    Choose then $s\in\mc G\rest F_1$ such that $\nrm{\sum_{n\in F_1} x_n} \leq 1 + \norm{\sum_{n\in s} x_n}$. Let $s_1\in\mc A\rest H$ be such that $s\sqsubseteq s_1$. Since $s_0\cup s_1\in \mc V\rest H$ with $s_0 < s_1$, it follows that $\mk{b}(s_0\cup s_1) = 1$, i.e.
    \[
    \norm{\sum_{n\in s_1} x_n^{s_1}}\le 2 \norm{\sum_{n\in s_0} x_n^{s_0}}.
    \]
    Consequently, using Proposition \ref{iji5gojhtinbgv} and the fact that $x_n\geq 0$ for all $n$,
    \begin{align*}
        \norm{\sum_{n\in F_1} x_n} & \le 1 + \norm{\sum_{n\in s} x_n} \le 1 + \norm{\sum_{n\in s_1} x_n} \leq 1 + \norm{\sum_{n\in s_1} x_n^{s_1}} + \frac{1}{2^{\min s_1}} \leq 2\norm{\sum_{n\in s_0} x_n^{s_0}} + 2.
    \end{align*}
    The last expression does not depend on $F$, so $H\in\mc B(\mb x)$.
\end{proof}

\begin{prop}
If $\mc B(\mb x)$ is tall then $\hom(\mk t)= \langle \hom_{(1,1)}(\mk t)\rangle= \mc B(\mb p^{\hom_{(1,0)}(\mk t)\cup \hom_0(\mk{bs})})$.
\end{prop}

\begin{proof}
We set $K := \hom_{(1,0)}(\mk t)\cup \hom_0(\mk{bs})$, and it easily follows from \eqref{eq:simpler_expression_p^K} that
\[
    \mc B(\mb p^K)=\mc B(\mb p^{\hom_{(1,0)}(\mk t)})\cap \mc B(\mb p^{\hom_0(\mk{bs})}).
\]
The first inclusion follows again from Propositions \ref{j2iojsdfjnsdfsd} and \ref{lkmkgmndsfgsd}: if $H$ is $(i,j)$-homogeneous for $\mk t$ then $\hom_{(i,j)}(\mk t)\subseteq\fin$ for $(i,j)\neq (1,1)$, whereas if $(i,j) = (1,1)$ then $H\in \mc B(\mb p^{\hom_{(1,0)}(\mk t)})\cap \mc B(\mb p^{\hom_0(\mk{bs})})$.
    
    
    For the reverse inclusion, suppose that $M\in \mc B(\mb p^K)$. In particular, $M\in \mc B(\mb p^{\hom_0(\mk{bs})})=\langle \hom_1(\mk{bs})\rangle$, so $M$ is a finite union $M=M_1\cup \cdots\cup M_k $ of $1$-homogeneous sets for $\mk{bs}$. Fix $j=1,\dots,k$. Then, $M_j\in \mc B(\mb p^{\hom_0(\mk{b})})$ for if $H\con M_j$ is $0$-homogeneous for $\mk{b}$, then it is also $(1,0)$-homogeneous for $\mk t$, and since $M_j\in \mc B(\mb p^K) $, we have that $\#H\le \max_{F\con M_j\text{ finite}}\nrm{\sum_{n\in F} p_n^K}<\infty$ and consequently, $\sup_{F\con M_j\text{ finite}}\nrm{\sum_{n\in F} p_n^{\hom_0(\mk{b})}}\le \max_{F\con M_j\text{ finite}}\nrm{\sum_{n\in F} p_n^K} $ must be finite.  We have observed that $\mk{b}$ is a 0-comparability coloring, hence 
    \[
    M_j\in \mc B(\mb p^{\hom_{0}(\mk{b})}) = \langle \hom_{1}(\mk{b})\rangle,
    \]
    and this means that $M\in \langle\hom_{(1,1)}(\mk t)\rangle$.
    %
    %
\end{proof}

We have just proved the following.

\begin{teo}\label{lkmkdlmfmskl4we4ret}
If a tall ideal can be $\mc B$-represented in some $C(K)$, $K$ countable, then it contains some non-pathological front ideal (see Definition \ref{lj4witejigerjrgf}), and consequently has a Borel selector.  \qed 
\end{teo}
This result gives a necessary condition that aids in comprehending the following fundamental question.

\begin{question}
Which $\mc B$-ideals can be represented in $C(K)$ for  a countable compactum $K$? 
\end{question}

A question similar to the previous one could yield a negative answer when no simple characterization is possible. This might be the case here, as the Cantor-Bendixson rank of a countable compactum can take any countable ordinal value. It is a classical result that the collection of codes for countable ordinals is a co-analytic, non-Borel set.

Since all $\B$-ideals are represented in the Polish space $C(2^{\N})$, we formulate the problem  in a way that is well-suited for applying tools from descriptive set theory.

\begin{question}
Find the (Borel or projective) complexity of the following set:
\[
\{{\bf g}\in C(2^{\N})^{\N}: \B({\bf g})\mbox{ is representable in $C(K)$ for some countable compactum $K$}\}.
\]
An instance of the previous question, which is  of particular interest, is to determine the complexity of
\[
\{{\bf g}\in C(2^{\N})^{\N}: \B({\bf g})\mbox{ is representable in $c_0$}\}.
\]
\end{question}


\subsection{\texorpdfstring{An ideal characterization of $c_0$}{An ideal characterization of c0}}
Recall that the isomorphic classification of the spaces $C(K)$ with $K$ countable compact is completely determined: the  Mazurkiewicz-Sierpinski theorem states that every such compact space is homeomorphic to  the ordinal $\omega^\alpha m+1$, where $\alpha+1$ is the Cantor-Bendixson rank of $K$ and $m$ is the cardinality of the $\alpha$-derivated set of $K$. And Bessaga and Pelczynski proved that given two ordinal numbers $\alpha\le \beta$, the corresponding Banach spaces $C(\alpha+1)$ and $C(\beta+1)$ are isomorphic exactly when $\beta<\alpha^\omega$. For more information, see \cite[subsections 8.6 and 21.5]{Semadeni}. 

It follows from here that when the Cantor-Bendixson rank of $K$ is finite, being $C(K)$ isomorphic to $c_0$, the ideal $\mathcal B(\mb g)$ is tall exactly when  $\mb g$ is a weakly-null sequence. This result was shown in \cite[Thm 5.14]{MMU2022} for $c_0$ using an argument similar to that presented in Section \ref{sect:colorings} (see Proposition \ref{fin-antichain}). However we can give a direct proof as follows. Every  weakly-null sequence $\mb x$ in $c_0$  has a subsequence that is equivalent to a  block sequence of the unit basis of $c_0$, and consequently, either norm-null or equivalent to the unit basis of $c_0$. Theorem \ref{c0saturacion} guarantees that $\mathcal B({\mb g})$ is tall.  We have the following characterization.

\begin{teo}
\label{978y9834yhtghfgdb}  
The following are equivalent for a metrizable compactum $K$.

\begin{enumerate}[(i), wide=0pt]
\item $\mathcal B(\mb g)$ is tall for every weakly-null sequence $\mb g$ in $C(K)$.
\item $\mathrm{Fin}\subsetneq  \mathcal B(\mb g)$ for every weakly-null sequence $\mb g$ in $C(K)$.

\item 
$\mathcal R \leq_K\mathcal B(\mb g)$  for every weakly-null sequence $\mb g$ in $C(K)$. 


\item $K$ has finite Cantor-Bendixson rank, or, equivalently, $C(K)$ is isomorphic to $c_0$. 
\end{enumerate}
\end{teo}

\begin{proof}
We are going to use here the well-known fact (see e.g. \cite{Semadeni}) that a $C(K)$ space contains isomorphic copies of $C(L)$ provided that the corresponding Cantor-Bendixson ranks $\rho_\mr{CB}(K)$ and $\rho_\mr{CB}(L)$ satisfy that $\rho_\mr{CB}(L) \cdot \omega < \rho_\mr{CB}(K)$. Consequently, if $K$ is a countable compactum with infinite Cantor-Bendixson rank, then $C(K)$ contains an isomorphic copy of $C(\mathcal S)$.

Let us start now proving the equivalence between those statements.   {\it (i)} trivially implies  {\it (ii)}. {\it (ii)} implies {\it (iv)}: If $K$ is countable and has infinite Cantor-Bendixson rank, then, by the comment above, $C(K)$ has an isomorphic copy of $C(\mc S)$, while if  $K$ is uncountable, then by Miljutin's Theorem we know that $C(K)$ is isomorphic to $C[0,1]$, hence it also contains an isomorphic copy of $C(\mc S)$. But the evaluation sequence $\mb p^\mc S$ (see Definition \ref{j4iotorejgidf}) corresponding to the Schreier family $\mc S:=\conj{s\con \N}{\# s\le \min s+1}$  is a weakly-null sequence  satisfying that 
$$
\norm{\sum_{n\in F} p_n^\mc S}= \max\conj{\# s}{s\in \mc S\upharpoonright F}\ge \frac{\#F}2
$$
for every finite set $F\con \N$ (see \eqref{eq:simpler_expression_p^K}). Consequently there is a weakly null sequence in $C(K)$ whose $\mc B$-ideal is the trivial one $\fin$. 

To see that {\it (iv)} implies {\it (i)}, let $K$ be  a finite set. Then $C(K)$ is isomorphic to $\ell_\infty^{\# K}$, and we know that $\mc B$ ideals from weakly-null sequences in a finite dimensional space are tall summable ideals. Now, if $K$ is an infinite compactum with finite Cantor-Bendixson rank, then $C(K)$ is isomorphic to $c_0$, and we have already mentioned previous to the statement of the theorem that $\mc B$-ideals from weakly null sequences in $c_0$ are always tall. 


Finally we see the equivalence between {\it (iii)} and {\it (iv)}. If $K$ has finite Cantor-Bendixson rank, then either $C(K)$ is finite dimensional or isomorphic to $c_0$; in the first case, since every $\mc B$-ideal from $C(K)$ is summable, hence it contains a coloring ideal (see \S \ref{klkl4tmerlkgrfg}); in  the second case, $C(K)$ is isomorphic to $c_0$, so, by   Theorem \ref{c0saturacion} {\it (iv)}, $\mc B(\mb g)$ is tall, if $\mb g$ is weakly-null,  and now we can apply Proposition \ref{fin-antichain} to conclude that $\mc B(\mb g)$ contains a coloring ideal.  Finally, if $K$ is uncountable or countable with infinite Cantor-Bendixson rank, then in either case $C(K)$ contains $C(\mc S)$, and consequently, by Theorem \ref{exG} there is a weakly null sequence $\mb g$ in $C(K)$ such that $\mc R \not\le_K\mc B(\mb g)$. 
\end{proof}

\end{document}